\documentclass{article}

\usepackage[utf8]{inputenc}
\usepackage[left=1in,right=1in,top=1in,bottom=1.5in]{geometry}
\usepackage{bm,bbm}
\usepackage{amsmath}
\usepackage{amsfonts}
\usepackage{amssymb}
\usepackage{amsthm}
\usepackage{enumerate}
\usepackage{color}
\usepackage{mathrsfs}
\usepackage{stmaryrd}
\usepackage{hyperref}
\usepackage{authblk}

\SetSymbolFont{stmry}{bold}{U}{stmry}{m}{n}

\newtheorem{thm}{Theorem}[section]

\newtheorem{lem}[thm]{Lemma}

\newtheorem{assu}[thm]{Assumption}
\newtheorem{prop}[thm]{Proposition}
\theoremstyle{definition}
\newtheorem{defn}[thm]{Definition}

\theoremstyle{remark}
\newtheorem{rem}[thm]{Remark}

\newcommand{\norm}[1]{\left\lVert#1\right\rVert}

\DeclareMathAccent{\wideparen}{\mathord}{largesymbols}{"F3}

\newcommand{\rO} {\mathrm{O}}

\newcommand{\deq}{\mathrel{\mathop:}=}
\newcommand{\e}[1]{\mathrm{e}^{#1}}
\newcommand{\R} {\mathbb{R}}
\newcommand{\C} {\mathbb{C}}
\newcommand{\N} {\mathbb{N}}
\newcommand{\Z} {\mathbb{Z}}

\newcommand{\adj}{^*}

\newcommand{\dist} {\mathrm{dist}}
\newcommand{\CR} {\C\setminus\R_{+}}

\DeclareMathOperator{\diag}{diag}
\DeclareMathOperator{\tr}{tr}
\DeclareMathOperator{\Tr}{Tr}

\DeclareMathOperator{\supp}{supp}

\DeclareMathOperator{\re}{\mathrm{Re}}
\DeclareMathOperator{\im}{\mathrm{Im}}

\newcommand{\caD}{{\mathcal D}}
\newcommand{\caE}{{\mathcal E}}

\newcommand{\caG}{{\mathcal G}}
\newcommand{\caH}{{\mathcal H}}

\newcommand{\caL}{{\mathcal L}}

\newcommand{\caN}{{\mathcal N}}

\newcommand{\caP}{{\mathcal P}}
\newcommand{\caQ}{{\mathcal Q}}

\newcommand{\caS}{{\mathcal S}}
\newcommand{\caT}{{\mathcal T}}

\newcommand{\ub}{\mathbf{u}}
\newcommand{\vb}{\mathbf{v}}

\newcommand{\bbS}{{\mathbb S}}

\newcommand{\frd}{{\mathfrak d}}

\newcommand{\frN}{{\mathfrak N}}

\newcommand{\frX}{{\mathfrak X}}
\newcommand{\frY}{{\mathfrak Y}}
\newcommand{\frZ}{{\mathfrak Z}}

\newcommand{\bsd}{{\boldsymbol d}}
\newcommand{\bse}{{\boldsymbol e}}

\newcommand{\bsg}{{\boldsymbol g}}
\newcommand{\bsh}{{\boldsymbol h}}

\newcommand{\bsr}{{\boldsymbol r}}

\newcommand{\bsv}{{\boldsymbol v}}

\newcommand{\bsy}{{\boldsymbol y}}

\newcommand{\bsalp}{{\boldsymbol{\alpha}}}
\newcommand{\bsbet}{{\boldsymbol{\beta}}}

\newcommand{\wt}{\widetilde}
\newcommand{\ol}{\overline}
\newcommand{\wh}{\widehat}

\newcommand{\mr}{\mathring}

\newcommand{\beq}{ \begin{equation} }
	\newcommand{\eeq}{ \end{equation} }
\newcommand{\beqs}{	\begin{equation*}	}
	\newcommand{\eeqs}{	\end{equation*}	}

\newcommand{\lone}{\mathbb{I}} 

\newcommand{\dd}{\mathrm{d}}
\newcommand{\ii}{\mathrm{i}}

\renewcommand{\P}{\mathbb{P}}

\newcommand{\angi}{\langle i\rangle}

\newcommand\Var[1]{\mathrm{Var}\left[#1\right]}
\newcommand\expct[1]{\mathbb{E}\left[#1\right]}

\newcommand\prob[1]{\mathbb{P}\left[#1\right]}

\newcommand\Absv[1]{\left\vert#1\right\vert}
\newcommand\absv[1]{\vert#1\vert}
\newcommand\llbra{\llbracket}
\newcommand\rrbra{\rrbracket}

\newcommand\AND{\quad\text{and}\quad}

\numberwithin{equation}{section}

\date{\today}

\allowdisplaybreaks

\title{Local laws for multiplication of random matrices} 

\author[1]{Xiucai Ding \thanks{E-mail: xcading@ucdavis.edu}}
\author[2]{Hong Chang Ji \thanks{E-mail: hongchang.ji@ist.ac.at}}

\affil[1]{Department of Statistics, University of California, Davis}

\affil[2]{Institute of Science and Technology Austria}
\date{\today}

\begin{document}
\maketitle
\abstract{
	Consider the random matrix model $A^{1/2} UBU^* A^{1/2},$ where $A$ and $B$ are two $N \times N$ deterministic matrices and $U$ is either an $N  \times N$ Haar unitary or orthogonal random matrix. It is well-known that on the macroscopic scale \cite{Voiculescu1991},  the limiting empirical spectral distribution (ESD) of the above model is given by the free multiplicative convolution of the limiting ESDs of $A$ and $B,$ denoted as $\mu_\alpha \boxtimes \mu_\beta,$ where $\mu_\alpha$ and $\mu_\beta$ are the limiting ESDs of $A$ and $B,$ respectively. In this paper, we study the asymptotic microscopic behavior of the edge eigenvalues and eigenvectors statistics.  We prove that both the density of $\mu_A \boxtimes \mu_B,$ where $\mu_A$ and $\mu_B$ are the ESDs of $A$ and $B,$ respectively and the associated subordination functions have a regular behavior near the edges. Moreover, we establish the local laws near the edges on the optimal scale. In particular,  
	we prove that the entries of the resolvent are close to some functionals depending only on the eigenvalues of $A, B$ and the subordination functions with optimal convergence rates. Our proofs and calculations are based on the techniques developed for the additive model $A+UBU^*$ in \cite{Bao-Erdos-Schnelli2016,BAO2,BAO3,BEC}, and our results can be regarded as the counterparts of \cite{BEC} for the multiplicative model.
}

\vspace{3mm}
	{\textit{AMS Subject Classification (2020)}: 46L54, 60B20, 15B52
	
	}

\section{Introduction}\label{sec_introduction}
Large dimensional random matrices play important roles in high dimensional statistics. {More specifically, given a data matrix $Y,$ studying the eigenvalues and eigenvectors of $YY^*$ and $Y^*Y$ has been known to be an effective approach to analyze the data. There have been many different models for $Y$ depending on the nature of data it represents,} and the most fundamental one is the sample covariance matrix \cite{yao2015large}. In this context, $Y$ can be written as $Y=A^{1/2}X,$ where $A$ is the population covariance matrix and $X$ contains i.i.d. centered random variables. An extension of the sample covariance matrix is the separable covariance matrix \cite{bun2017,DYaos,ding2020tracy,PAUL200937,yang2019}, where $Y=A^{1/2}XB^{1/2}$ with another positive definite matrix $B$.

Even though the assumption that $X$ has i.i.d. entries is popular and useful in the literature, {its applications are limited to data composed of linear functions of independent samples.} An important example {that such an assumption cannot cover} is the Haar distributed random matrix, which has been used in the literature of statistical learning theory \cite{ELnips,nips20201,Liu2020Ridge,2020arXiv200500511Y}. More specifically, we consider $X=U$ to be either an $N \times N$ random Haar unitary or orthogonal matrix so that  
\begin{equation}\label{eq_defndatamatrixtype}
	Y=A^{1/2}UB^{1/2}.
\end{equation}
In other words, we study a general class of separable random matrices beyond the i.i.d. assumption; {indeed, the model \eqref{eq_defndatamatrixtype} covers the case where $X$ consists of i.i.d. Gaussian random variables due to invariance}. We mention that the data matrix (\ref{eq_defndatamatrixtype}) has appeared in the study of high dimensional data analysis, for instance, data acquisition \cite{DW}, matrix denoising \cite{7587390,bun2017} and random sketching \cite{ELnips,2020arXiv200500511Y}. 

The empirical spectral distribution (ESD) of $YY^*$ in (\ref{eq_defndatamatrixtype}) has been studied in the literature of free probability theory. Denote 
\begin{equation}\label{defn_freemultiplicationmodel}
	H=AUBU^*,
\end{equation}
which has the same eigenvalues with $YY^*.$ In the influential work \cite{Voiculescu1991}, Voiculescu studied the limiting spectral distribution of the eigenvalues of $H$ and showed that it {was} given by the {free multiplicative convolution} of $\mu_\alpha$ and $\mu_\beta,$ denoted as $\mu_\alpha \boxtimes \mu_\beta$, where $\mu_\alpha$ and $\mu_\beta$ are the limiting ESDs of $A$ and $B,$ respectively; see Definition \ref{def:freeconv} for a precise statement. More recently, in \cite{JHC}, the author investigated the behavior of $\mu_{\alpha} \boxtimes \mu_{\beta}$ by analyzing a system of deterministic equations, known as {subordination equations},  that defines the free convolution; see equation (\ref{eq_suborsystemPhi}) for details. They also proved that under certain regularity assumptions, the density of $\mu_{\alpha}\boxtimes \mu_{\beta}$ {had} a regular square root behavior near the edges of its support.  

However, on the microscopic level, the singular value and vector statistics of $Y,$ as well as the local laws, have not been established so far. The aim of this paper is to fill this gap near the regular edges. {Before proceeding to our main focus, we pause to discuss the {additive model,} that is, $A+UBU^*.$ The ESD of the additive model {converges to} the free additive convolution of $\mu_\alpha$ and $\mu_\beta,$ denoted as $\mu_\alpha \boxplus \mu_\beta$ \cite{Voiculescu1991}. More recently, the local laws as well as eigenvalues and eigenvectors statistics have been extensively studied in the series of papers \cite{Bao-Erdos-Schnelli2016,BAO2,BAO3,BEC,MR4205272}. Our arguments are strongly inspired by these works and our results can be regarded as multiplicative counterparts of \cite{BEC}. In what follows, we highlight and summarize the results and techniques of the additive model \cite{Bao-Erdos-Schnelli2016,BAO2,BAO3, BEC} in Section \ref{subsec_additionresults}. Then we explain how we adapt their approaches with some modifications to obtain the results for the multiplicative model (\ref{defn_freemultiplicationmodel}) in Section \ref{subsec_multiplicativemodelresults}. }

\subsection{Local laws for addition of random matrices}\label{subsec_additionresults}
In this subsection, we review the results and techniques for the addition of random matrices $A+UBU^*$ in the series of papers \cite{Bao-Erdos-Schnelli2016,BAO2,BAO3,BEC}.

In \cite{Bao-Erdos-Schnelli2016, BAO2, BAO3}, the authors studied the local laws in regular bulk spectrum of the free additive convolution. Chronologically, in \cite{Bao-Erdos-Schnelli2016}, the authors proved that the system of the subordination equations, defining the free additive convolution, {was} stable away from the edges of the support and singularities. In particular, on one hand, they showed that the system {was} stable and the imaginary parts of the subordination functions {were} bounded below in the regular bulk; on other hand, they proved a local stability result of the free additive convolution. Based on \cite{Bao-Erdos-Schnelli2016}, in \cite{BAO2}, they proved that the local laws {held} in the bulk of the spectrum down to the optimal scale $N^{-1+\gamma},$ for any $\gamma>0,$ which improved a result obtained in \cite{Bao-Erdos-Schnelli2016}. Particularly, they proved a version of averaged local law that the ESD of $A+UBU^*$ {concentrated} around $\mu_A \boxplus \mu_B$ {where $\mu_{A}$ and $\mu_{B}$ denote the ESDs of $A$ and $B$, respectively. They also proved the entry-wise local law} that the every entry of the resolvent $G:=(A+UBU^*-z)^{-1}, \ z=E+\ii \eta \in \mathbb{C}_+,$ {was} well estimated {at deterministic functions of $z$}. As a byproduct, they showed that the bulk eigenvectors {were} completely delocalized.  Later on,  in \cite{BAO3}, the authors obtained the optimal convergence rate $(N\eta)^{-1}$ in the bulk for the local laws which improved the result of \cite{BAO2} where the convergence speed was shown to be of {smaller order than} $(N \eta)^{-1/2}.$ 

We highlight several important technical components and insights of the aforementioned three works. Since \cite{BAO3} established the local laws down to the optimal scale with optimal precision which refined the results of \cite{Bao-Erdos-Schnelli2016, BAO2}, we focus our discussion on \cite{BAO3}.  The core is to explore the system of subordination functions globally and locally. First, since the additive model lacks the independence of matrix elements, they employed a partial randomness decomposition (see (\ref{eq_prd}) in the present paper) of the Haar measure which enabled them to take partial expectations of the entries of the resolvent. Second, to connect the resolvent with the subordination functions, they used the approximate subordination functions which depend only on the resolvent of $A+UBU^*$. In particular, their choices for these approximates can be considered as a random version of those used in \cite{Pastur-Vasilchuk2000,MR3353823}; see equation (3.18) of \cite{BAO3}. With such choices, they were able to work on a new system of self-consistent equations. Surprisingly, it suffices to monitor only two auxiliary quantities to analyze the system. With the aid of the local stability results of \cite{Bao-Erdos-Schnelli2016}, they connected the partial expectations of the entries of the resolvent with the subordination functions. Third, they proposed a novel strategy to handle the fluctuation averaging mechanism for Haar random matrices. More specifically, instead of working directly with $N^{-1} \sum_{i} G_{ii},$ they first {considered generic averages of an auxiliary quantity which was a carefully chosen linear combination of $G_{ii}$ and $(UBU^*G)_{ii}.$ Such a particular choice made the leading order terms within its average cancel algebraically, and the auxiliary quantity can be passed to $G_{ii}$ by taking different weights for this average.} Finally, to streamline the calculation, instead of directly computing high moments of the essential auxiliary quantities, they used the so-called recursive moment estimates, in which high-moments {were} estimated in terms of the lower moments with the aid of integration by parts. 

Armed with the above techniques and results, in \cite{BEC}, they were able to investigate the local laws near the regular edges in the sense that $\mu_A \boxplus \mu_B$ {had} a regular behavior near the edges. More specifically, they presented the local laws near the edges on the optimal scale with optimal precision. Based on these results, they were able to prove the edge eigenvalue rigidity and edge eigenvector delocalization. On the technical level, they used and generalized the strategies and inputs  of \cite{Bao-Erdos-Schnelli2016, BAO2, BAO3} as summarized in the previous paragraph. Since the eigenvalues around the edges are sparse and fluctuate more, in order to guarantee the regular behavior of $\mu_A \boxplus \mu_B,$ they first established the square root decay of their limiting counterpart $\mu_{\alpha} \boxplus \mu_{\beta}.$ Under suitable assumptions on the Levy distances, with the local stability, the measure $\mu_A \boxplus \mu_B$ inherits the regularity up to the optimal scale. In addition, the probabilistic part of \cite{BAO3} is not sufficient around the edges as the subordination functions become unstable and the improvement from fluctuation averaging in \cite{BAO3} is suboptimal. In order to compensate this instability, they established a very accurate estimate on the approximation error. To achieve this goal, they carefully identified a {new pair of auxiliary quantities}; see equations (4.14) and (4.15) of \cite{BEC}. In particular, {one of the auxiliary quantities in \cite{BEC} has an additional counter term compared to the one used in \cite{BAO3}.} We mention that \cite{BEC} required the assumption that at least one of the Stieltjes transforms of $\mu_\alpha$ and $\mu_\beta$ {was} bounded from above. This assumption can be removed using their recent results in   \cite{MR4205272}.  

In summary, using addition of random matrices as an example, the authors in \cite{Bao-Erdos-Schnelli2016,BAO2,BAO3,BEC} have developed a general framework and powerful techniques to study the local laws of random matrix models where the main source of randomness is the Haar matrix. Since multiplication of random matrices is another typical example using Haar matrix, it is natural to study the multiplication of random matrices using the techniques developed for the additive model. This will be discussed in next subsection, Section \ref{subsec_multiplicativemodelresults}. 

{
	\subsection{From addition to multiplication: an overview of our results}\label{subsec_multiplicativemodelresults}
	
	In this subsection, we explain how to adapt the techniques of the additive model \cite{Bao-Erdos-Schnelli2016,BAO2,BAO3,BEC} as summarized in Section 
	\ref{subsec_additionresults} to obtain the results for the multiplicative model (\ref{defn_freemultiplicationmodel}). The main purpose of this paper is to present a comprehensive edge local law on the optimal scale and with optimal convergence rates for the multiplicative model, which is the counterpart of \cite{BEC}. In what follows, we give an overview of our results and explain how to handle the multiplicative model adapting the techniques of additive model in \cite{BAO3,BEC}.   
	
	The first part of our results concerns the regularity of $\mu_A \boxtimes \mu_B$ and the subordination functions. More specifically, in Proposition \ref{prop:stabN} below, we establish the stability properties of the subordination functions near the regular edges and provide some crucial estimates. Here we point out that, instead of using the conventional $\eta$-transform \cite{ Belinschi-Bercovici2007, Voiculescu1987} } to define the subordination functions and free multiplicative convolution, we use a simple conjugate of it known as $\mathsf{M}$-transform (c.f. Definition \ref{defn_transform})\cite{chistyakov2011arithmetic, JHC}.  One technical advantage  of using    $\mathsf{M}$-transform is that it makes the similarity between the additive and multiplicative models more evident, which enables us to adapt the techniques of \cite{BAO3,BEC} more directly and easily. The proof of Proposition \ref{prop:stabN} follows from its counterpart for the additive model in \cite{BEC} (see Proposition 3.1 therein) which can be split into two steps. In the first step, the results are proved for the limiting measures $\mu_{\alpha}$ and $\mu_\beta$ under some regularity assumptions (c.f. Assumption \ref{assu_limit}). In the second step, assuming that $\mu_{A}, \mu_B$ and $\mu_{\alpha}, \mu_\beta$ are close enough (c.f. Assumption \ref{assu_esd}), the statements can be carried over to the measures $\mu_A$ and $\mu_B.$ As mentioned earlier, in \cite{BEC}, the authors proved analogous results for the additive model assuming that at least one of the Stieltjes transforms of $\mu_\alpha$ and $\mu_\beta$ was bounded from above (see (iii) of Assumption 2.1 in \cite{BEC}), which could be removed using their recent results in \cite{MR4205272}. For our multiplicative model, since the analog of  \cite{MR4205272} has been established by the second author in \cite{JHC}, we will not need this condition in our Assumption \ref{assu_limit}.

The second part of our results focuses on establishing the optimal edge local laws on the optimal scale for   the multiplicative model. In Theorem \ref{thm:main} below, we provide accurate estimates for the entries of the resolvent and also prove the averaged local law. The convergence rates are optimal up to some $N^{\epsilon}$ factor. As two consequences, we prove the rigidity of the edge eigenvalues in Theorem \ref{thm_rigidity} and the complete delocalization of the edge eigenvectors in Theorem \ref{thm_delocalization}.  On the technical level, the proof of Theorem \ref{thm:main} follows closely from its counterpart for additive model \cite{BEC} (see Theorem 2.5 therein) as summarized in Section \ref{subsec_additionresults}. In what follows, we highlight the key ingredients on the adaption of their arguments. Thanks to the $\mathsf{M}$-transform,  the approximate subordination functions for the multiplicative model (c.f. Definition \ref{defn_asf}) can be easily identified. To control the errors between the subordination functions and their approximates, we first explore some hidden relations. {For instance, in (\ref{eq:Lambda}) and (\ref{eq:BGii-S}) we represent the error in terms of the resolvents}. This enables us to find the key auxiliary quantities to work with. Then we use integration by parts to start the recursive estimates to obtain {bounds for} high moments of these essential quantities. In order to establish the optimal convergence rates, as mentioned in \cite{BEC}, the weights in the fluctuation averaging mechanism needed to be properly chosen. In our case, these weights (c.f. (\ref{eq_optimalfaquantitiescoeff}) and (\ref{eq_optimalfaquantitiescoeffextra})) can be constructed using the hidden identities obtained earlier. Finally, we point that due to the structural difference between the additive and multiplicative models, many errors in our model need more careful treatment. For example, in the fluctuation averaging mechanism,  our error terms $\mathsf{e}_{i1}$ in (\ref{eq_epsilon1})  and $\mathsf{e}_{i2}$ in (\ref{eq_defnmathsfe2}) will generate some $\rO_{\prec}(N^{-1/2})$ terms. The weighted summations of these terms will be canceled out algebraically after we explore some hidden identities; see (\ref{eq_epsilon1details})--(\ref{eq:expan_ei2}) and the associated discussion for more details.

As mentioned in \cite{BEC}, the results of addition of random matrices demonstrate that the Haar randomness in the additive model leads to an analogous behavior to the Wigner matrices \cite{2017dynamical} in the sense of strong concentration of the eigenvalues and eigenvectors. In the same spirit, the Haar randomness in our multiplicative model (\ref{eq_defndatamatrixtype}), results in a similar behavior as the separable covariance matrices as in \cite{yang2019, DYaos}. Finally,  we mention that the arguments of the current paper can be carried out to study the bulk eigenvalues and eigenvectors  as in \cite{Bao-Erdos-Schnelli2016, BAO2, BAO3} which deals with additive model. The results obtained here can also be used to study other models and statistics, for example, the deformed invariant model \cite{belinschi2017} and the Tracy-Widom distribution for the edge eigenvalues. We will pursue these topics in future works.

The rest of the paper is organized as follows. In Section \ref{sec:mainresult}, we introduce the necessary notations  and state the main results. In Section \ref{sec_proofroute}, we present a structural summary of our proof. In Section \ref{sec_entrylaw}, we  prove a subordination property for the resolvent entries. The proof of fluctuation averaging lemmas, along with some auxiliary lemmas and technical proofs, are collected in the appendix.

\vspace{3pt}
\noindent {\bf Conventions.}
For $M,N\in\N$, we denote $\{k\in\N:M\leq k\leq N\}$ by $\llbra M,N\rrbra$. For $N\in\N$ and $i\in\llbra 1,N\rrbra$, we denote by $\mathbf{e}_{i}^{(N)} \in \mathbb{R}^N$ with $(\mathbf{e}_{i})_{j}=\delta_{ij}$. We often omit the superscript $N$ to write $\mathbf{e}_{i}^{(N)}\equiv \mathbf{e}_{i}$. We use $I$ for the identity matrix of any dimension. 
For an $N$-dimensional real or complex random vector $\bm{g}=(g_1,\cdots, g_N)$, we write $\bm{g} \sim \mathcal{N}_{\mathbb{R}}(0,\sigma^2 I_N)$ if $g_1,\cdots, g_N$ are i.i.d. $\mathcal{N}(0,\sigma^2)$ random variables, and we write $\bm{g} \sim \mathcal{N}_{\mathbb{C}}(0, \sigma^2 I_N)$ if $g_1, \cdots, g_N$ are i.i.d. $\mathcal{N}_{\mathbb{C}}(0,\sigma^2)$ variables, where $g_i \sim \mathcal{N}_{\mathbb{C}}(0, \sigma^2)$ means that $\re g_i$ and $\im g_i$ are independent $\mathcal{N}(0,\frac{\sigma^2}{2})$ random variables. 
For any matrix $A,$ we denote its operator norm by $\norm{A}$ and for a vector $\bm{v},$ we use $\| \bm{v}\|$ for its $\ell_2$ norm.

\section{Main results}\label{sec:mainresult}

\subsection{Notations and assumptions}\label{sec:subsecnotationandassumption}
For any $N \times N$ matrix $W,$ we denote its normalized trace by $\tr W$, that is, 	
\begin{equation}\label{eq_defntrace}
	\tr W=\frac{1}{N}\sum_{i=1}^{N} W_{ii}.
\end{equation}
Moreover, its empirical spectral distribution (ESD) is defined as  	
\begin{equation*}
	\mu_W=\frac{1}{N}\sum_{i=1}^N \delta_{\lambda_i(W)}. 
\end{equation*}
In the present paper, even if the matrix is not of size $N\times N$, the trace is always normalized by $N^{-1}$ unless otherwise specified. 

Consider two $N\times N$ real deterministic positive definite matrices 
\beqs
A\equiv A_{N}=\diag(a_{1},\cdots,a_{N}), \quad B\equiv B_{N}=\diag(b_{1},\cdots,b_{N}),
\eeqs
where the diagonal entries are ordered as $a_{1}\geq a_{2}\geq \cdots\geq a_{N}>0$ and $b_{1}\geq b_{2}\geq\cdots\geq b_{N}>0$. Let $U\equiv U_{N}$ be a random unitary or orthogonal matrix, Haar distributed on the unitary group $U(N)$ or the orthogonal group $O(N)$. Denote $\wt{A}\deq U\adj AU$, $\wt{B}\deq UBU\adj$, and 
\beq\label{defn_eq_matrices}
H\deq AUBU\adj,	\quad	\caH\deq U\adj AU B,	\quad	
\wt{H}\deq A^{1/2}\wt{B}A^{1/2}, \AND \wt{\caH}\deq B^{1/2}\wt{A}B^{1/2}. 
\eeq 
Note that we only need to consider diagonal 	matrices $A$ and $B$ since $U$ is a Haar random unitary or orthogonal  matrix. Moreover, $\wt{H}$ and $\wt{\caH}$ are Hermitian random matrices. 

Since $H$, $\caH$, $\wt{H}$ and $\wt{\caH}$ have the same eigenvalues, in the sequel,  we denote the eigenvalues of all of them as $\lambda_1 \geq \lambda_2 \geq \cdots \geq\lambda_N$ without causing any confusion. Further, we define the ESDs of the above matrices by
\beqs
\mu_{A}\equiv \mu_{A}^{(N)}\deq \frac{1}{N}\sum_{i=1}^{N}\delta_{a_{i}},\quad 	\mu_{B}\equiv\mu_{B}^{(N)}\deq \frac{1}{N}\sum_{i=1}^{N}\delta_{b_{i}}, \quad
\mu_{H}\equiv\mu_{H}^{(N)}\deq \frac{1}{N}\sum_{i=1}^{N}\delta_{\lambda_{i}}.
\eeqs

For $z\in\C_{+}:=\{z \in \mathbb{C}: \im z >0\}$, we define the {resolvent} of $H$ as 
\beq\label{defn_greenfunctions} 
G(z)\deq (H-zI)^{-1}.
\eeq
Similarly, the resolvents of $\caH, \widetilde{H}$ and $\widetilde{\caH}$ are defined  as $ \caG(z), \wt{G}(z)$ and $\wt{\caG}(z),$ respectively. In the rest of the paper, we usually omit the dependence on $z$ and simply write $G, \caG, \wt{G}$ and $\wt{\caG}.$
The following transforms will play important roles in the current paper. 
\begin{defn}\label{defn_transform} For any probability measure $\mu$ defined on $\mathbb{R}_+,$ its \emph{Stieltjes transform} $m_{\mu}$ is defined as
	\beqs
	m_{\mu}(z)\deq\int\frac{1}{x-z}\dd\mu(x),\quad \text{for }z\in\CR.
	\eeqs 
	Moreover,  we define the $\mathtt{M}$-transform $M_{\mu}$ and $\mathtt{L}$-transform $L_{\mu}$ on $\CR$ as 
	\begin{align}\label{eq_mtrasindenity}
		M_{\mu}(z)&\deq 1-\left(\int\frac{x}{x-z}\dd\mu(x)\right)^{-1}= \frac{zm_{\mu}(z)}{1+zm_{\mu}(z)}, & 
		L_{\mu}(z)&\deq \frac{M_{\mu}(z)}{z}.
	\end{align}
\end{defn}
Let $m_{H}(z)$ be the Stieltjes transform of the ESD of $H.$ Since $H,\caH,\wt{H}$ and $\wt{\caH}$ are similar to each other, we have that $m_H(z)=\tr G=\tr \caG =\tr \wt{G}=\tr \wt{\caG}.$ Moreover, we have 
\begin{equation}\label{eq_connectiongreenfunction}
	G_{ij}=\sqrt{a_i/a_j} \wt{G}_{ij}, \ \mathcal{G}_{ij}=\sqrt{b_j/b_i} \wt{\caG}_{ij}.
\end{equation}


With the above preparation, we introduce the main assumptions. {Analogous to \cite{BEC}, throughout the paper, we assume that $\mu_A$ and $\mu_B$ converge to some $N$-independent absolutely continuous probability measures $\mu_{\alpha}$ and $\mu_{\beta}.$ We start with stating the assumptions on $\mu_{\alpha}$ and $\mu_{\beta},$ which is an analog of the additive model as in \cite[Assumption 2.1]{BEC}.  }  


\begin{assu}\label{assu_limit} 
	Suppose the following assumptions hold true: 
	\begin{itemize}
		\item[(i).] $\mu_{\alpha}$ and $\mu_{\beta}$ have densities $\rho_{\alpha}$ and $\rho_{\beta}$, respectively. For the ease of discussion, we assume that both of them have means $1$. 
		\item[(ii).] Both $\rho_{\alpha}$ and $\rho_{\beta}$ have single non-empty intervals as supports, denoted as $[E_-^{\alpha}, E_+^{\alpha}]$ and $[E_-^{\beta}, E_+^{\beta}],$ respectively. Here $E_-^{\alpha}, E_+^{\alpha}, E_-^{\beta}$, and $E_+^{\beta}$ are all positive numbers. Moreover, both of the density functions are strictly positive  in the interior of their supports. 
		
		\item[(iii).]  There exist constants $-1<t^{\alpha}_{\pm},t^{\beta}_{\pm}<1$ and $C>1$ such that 
		\begin{align*}
			&C^{-1}\leq \frac{\rho_{\alpha}(x)}{(x-E_{-}^{\alpha})^{t_{-}^{\alpha}}(E_{+}^{\alpha}-x)^{t_{+}^{\alpha}}}\leq C,\quad\forall x\in[E_{-}^{\alpha},E_{+}^{\alpha}],\\
			&C^{-1}\leq \frac{\rho_{\beta}(x)}{(x-E_{-}^{\beta})^{t_{-}^{\beta}}(E_{+}^{\beta}-x)^{t_{+}^{\beta}}}\leq C,\quad\forall x\in[E_{-}^{\beta},E_{+}^{\beta}].
		\end{align*}
	\end{itemize}
\end{assu}
\begin{rem}\label{rem_nolowerconstraint}
	First, {the assumption that both $\mu_{\alpha}$ and $\mu_{\beta}$ have means 1 in (i) is introduced for technical simplicity} and can be removed easily; see Remark 3.2 of \cite{JHC}. Second, the assumption (iii) is introduced to guarantee the square root behavior near the edges of the free multiplicative convolution of  $\mu_{\alpha}$ and $\mu_\beta.$ When this condition fails, the behavior of $\mu_{\alpha} \boxtimes \mu_{\beta}$ near the edge can be very different from our current discussion; see \cite{KLP} for more details. {Third, as we are only interested in the edge statistics near the upper edge in Sections \ref{sec_locallawresults}, the assumptions (ii) and (iii) can be relaxed by only imposing conditions on $E_+^{\alpha}$ and $E_+^{\beta}.$ We keep the current form involving $E_-^{\alpha}$ and $E_-^{\beta}$ since our results also hold near the lower edge with minor modification.} {Finally, for the additive model, the counterpart is Assumption 2.1 of  \cite{BEC}. It requires that at least one of the Stieltjes transforms of $\mu_\alpha$ and $\mu_\beta$ is bounded from above (see (iii) of Assumption 2.1 in \cite{BEC}), which could be removed using their recent results in   \cite{MR4205272}. We will no more need this condition since the counterpart of  \cite{MR4205272} for our model has been established in \cite{JHC}. }   
\end{rem}

The following Assumption \ref{assu_esd} ensures that $\mu_A$ and $\mu_B$ are close to $\mu_{\alpha}$ and $\mu_{\beta},$ respectively. Specifically, it demonstrates that the convergence rates of $\mu_{A}$ and $\mu_{B}$ to $\mu_{\alpha}$ and $\mu_{\beta}$ are bounded by an order of $N^{-1}$, so that their fluctuations do not dominate that of $\mu_{H}$. {Its counterpart for the additive model is \cite[Assumption 2.2]{BEC}. }
\begin{assu}\label{assu_esd} Suppose the following hold true  when $N$ is sufficiently large:
	\begin{itemize}
		\item[(iv).] For the Levy distance $\mathcal{L}(\cdot, \cdot),$  we have that for any small constant $\epsilon>0$
		\begin{equation*}
			\bsd\deq \mathcal{L}(\mu_{\alpha}, \mu_A)+\mathcal{L}(\mu_{\beta}, \mu_{B}) \leq N^{-1+\epsilon}.
		\end{equation*}

		\item[(v).] For the supports of $\mu_{A}$ and $\mu_{B}$, we have that for any constant $\delta>0$
		\beqs
		\supp\mu_{A}\subset [E_{-}^{\alpha}-\delta,E_{+}^{\alpha}+\delta], \quad \supp\mu_{B}\subset[E_{-}^{\beta}-\delta,E_{+}^{\beta}+\delta].
		\eeqs
	\end{itemize}
\end{assu}

\begin{rem}
	We remark that we will consistently use $\epsilon$ as a generic sufficiently small constant whose value may change from one line to the next. The assumption (v) assures that both of the upper edges of $\mu_A$ and $\mu_B$ are bounded.  
\end{rem}

As proved by Voiculescu in \cite{Voiculescu1987,Voiculescu1991}, under Assumptions \ref{assu_limit} and \ref{assu_esd}, $\mu_{H}$ converges weakly to the free multiplicative convolution of $\mu_{\alpha}$ and $\mu_{\beta}$, denoted as $\mu_{\alpha}\boxtimes\mu_{\beta}$. 
In the present paper, instead of using the original definitions proposed by Voiculescu, we use the $\mathsf{M}$-transform  in  (\ref{eq_mtrasindenity}) to define the free multiplicative convolution. The following lemma summarizes the properties of the analytic subordination functions. {The counterpart for the additive model is summarized in \cite[Proposition 2.3]{BEC}.} 
\begin{lem}[Proposition 2.5 of \cite{JHC}]\label{lem_subor}
	There exist unique analytic functions $\Omega_{\alpha},\Omega_{\beta}:\CR\to\CR$ satisfying the following: \\
	\noindent{(1).} For all $z\in\C_{+}$, we have \beq\label{eq_subsys3}
	\arg \Omega_{\alpha}(z)\geq \arg z \AND \arg\Omega_{\beta}(z)\geq \arg z.
	\eeq
	\noindent{(2).} For all $z\in\C_{+},$		
	\beq \label{eq_subsys2}
	\lim_{z\searrow-\infty}\Omega_{\alpha}(z)=\lim_{z\searrow-\infty}\Omega_{\beta}(z)=-\infty.
	\eeq
	
	\noindent{(3).} For all $z\in\CR$, we have 
	\beq\label{eq_suborsystem}
	zM_{\mu_{\alpha}}(\Omega_{\beta}(z))=zM_{\mu_{\beta}}(\Omega_{\alpha}(z))=\Omega_{\alpha}(z)\Omega_{\beta}(z).
	\eeq
	
\end{lem}
The analytic functions $\Omega_{\alpha}$ and $\Omega_{\beta}$  are referred to as the  {subordination functions}. We remark that the same functions as well as $\mathtt{M}$-transforms also appeared in \cite{chistyakov2011arithmetic}, called $Z$ and $K$-functions, respectively.
Similarly, we denote $\Omega_{A}$ and $\Omega_{B}$ by replacing $(\alpha,\beta)$ with $(A,B).$ With the aid of Lemma \ref{lem_subor}, we can define the free multiplicative convolution. 
\begin{defn}\label{def:freeconv}
	Denote the analytic function $M: \mathbb{C} \backslash \mathbb{R}_+ \rightarrow \mathbb{C} \backslash \mathbb{R}_+ $ by
	\beq\label{eq_defn_eq}
	M(z):=M_{\mu_{\alpha}}(\Omega_{\beta}(z))=M_{\mu_{\beta}}(\Omega_{\alpha}(z)).
	\eeq
	Then the {free multiplicative convolution} of $\mu_{\alpha}$ and $\mu_{\beta}$ is defined as the unique probability measure $\mu,$ 
	denoted as $\mu \equiv \mu_{\alpha}\boxtimes\mu_{\beta}$ such that (\ref{eq_defn_eq}) holds
	for all $z\in\CR$. In other words, $M(z) \equiv M_{\mu_{\alpha} \boxtimes \mu_{\beta}}(z)$ is the $\mathsf{M}$-transform of $\mu_{\alpha} \boxtimes \mu_{\beta}.$ Furthermore, we define $\mu_{A}\boxtimes\mu_{B}$ so that $M_{\mu_{A}}(\Omega_{B}(z))=M_{\mu_{B}}(\Omega_{A}(z))=M_{\mu_{A}\boxtimes\mu_{B}}(z)$ holds for all $z\in\CR$.
\end{defn}
Note that a consequence of (\ref{eq_suborsystem}) and the definition of $M_{\mu}(z)$ is the following identity
\begin{equation}\label{eq_multiidentity}
	\int \frac{x}{x-z} d (\mu_{\alpha} \boxtimes \mu_{\beta}) (x)=\Omega_{\beta}(z) m_{\mu_{\alpha}}(\Omega_{\beta}(z))+1=\int \frac{x}{x-\Omega_{\beta}(z)} d \mu_{\alpha}(x).    
\end{equation}

\begin{rem}
	Since all of $\mu_{\alpha},\mu_{\beta},\mu_{A}$, and $\mu_{B}$ are compactly supported on $(0,\infty)$, similar results hold for $\mu_{\alpha}\boxtimes\mu_{\beta}$ and $\mu_{A}\boxtimes\mu_{B}$. Specifically,  according to \cite[Remark 3.6.2. (iii)]{Voiculescu-Dykema-Nica1992}, we have 
	\begin{align}\label{eq:priorisupp}
		\supp \mu_{\alpha}\boxtimes\mu_{\beta}&\subset [E_{-}^{\alpha}E_{-}^{\beta},E_{+}^{\alpha}E_{+}^{\beta}],	&
		\supp \mu_{A}\boxtimes\mu_{B}&\subset[a_{N}b_{N}, a_1 b_1].
	\end{align}
	In fact, we can conclude from  \cite[Theorem 3.1]{JHC} that, if  (i) and (ii) of Assumption \ref{assu_limit} hold, $\mu_{\alpha} \boxtimes \mu_{\beta} $ is absolutely continuous and  supported on a single non-empty compact interval on $(0, \infty),$ denoted as $[E_-, E_+],$ that is, 
	\beq\label{eq_edge}
	E_{-}\deq\inf \supp(\mu_{\alpha}\boxtimes\mu_{\beta}),\quad E_{+}\deq \sup\supp(\mu_{\alpha}\boxtimes\mu_{\beta}).
	\eeq
	Let the density of $\mu_{\alpha} \boxtimes \mu_{\beta}$ be $\rho.$ For small constant $\tau>0,$ with (iii) of Assumption \ref{assu_limit}, we have
	\begin{equation}\label{eq_originalsquarerootbehavior}
		\rho(x) \sim \sqrt{E_+-x}, \ x \in [E_+-\tau, E_+]. 
	\end{equation}
	Furthermore, as we will see in Lemma \ref{lem:suborsqrt}, the subordination functions $\Omega_{\alpha}$ and $\Omega_{\beta}$ also have square root behaviors near the edges. The regularity behavior is assured by the fact that the subordination functions $\Omega_{\alpha}$ and $\Omega_\beta$ are well separated from the supports of $\mu_\beta$ and $\mu_\alpha$, respectively; see (ii) of  Lemma \ref{lem:stabbound}. In fact, from the proof of Proposition 5.6 of \cite{JHC}, we see that the assumption (iii) of Assumption \ref{assu_limit} implies this stability condition. 
	
\end{rem}

\begin{rem} \label{rem_ctxexpansubordination}
	It is known from \cite{Belinschi2006, JHC} that Assumption \ref{assu_limit} ensures that the subordination functions $\Omega_{\alpha}\vert_{\C_{+}}$ and $\Omega_{\beta}\vert_{\C_{+}}$ can be extended continuously to the real line. Throughout the paper, we will write $\Omega_{\alpha}(x)$ or $\Omega_{\beta}(x)$  for $x \in \mathbb{R}$ to denote the continuous extensions. In particular, $\Omega_{\alpha}(x)$ and $\Omega_{\beta}(x)$ always have nonnegative imaginary parts for all $x\in\R$.
\end{rem}


\subsection{Properties of subordination functions}\label{sec_subordiationproperties}

In this subsection, we state the results regarding the local properties of the subordination functions and related quantities near the regular edge. These results will be used in the proof of the local laws.  
We first introduce some notations. Note that the system of subordination equations \eqref{eq_suborsystem} can be rewritten as 
\begin{align}\label{eq_suborsystemPhi}
	\Phi_{\alpha\beta}(\Omega_{\alpha}(z),\Omega_{\beta}(z),z)=0,
\end{align}
where we denote $\Phi_{\alpha\beta}\equiv(\Phi_{\alpha},\Phi_{\beta}):\{(\omega_{1},\omega_{2},z)\in\C_{+}^{3}: \arg \omega_{1},\arg\omega_{2}\geq \arg z\}\to \C^{2}$ by
\begin{align}\label{eq:def_Phi_ab}
	&\Phi_{\alpha}(\omega_{1},\omega_{2},z)\deq \frac{M_{\mu_{\alpha}}(\omega_{2})}{\omega_{2}}-\frac{\omega_{1}}{z}, &\Phi_{\beta}(\omega_{1},\omega_{2},z)\deq \frac{M_{\mu_{\beta}}(\omega_{1})}{\omega_{1}}-\frac{\omega_{2}}{z}.
\end{align}
Here $\Phi_{\alpha\beta}$ should be regarded as a function of three complex variables. 
We will also use the following quantities, which are closely related to the first and the second derivatives of the system (\ref{eq_suborsystemPhi}). Recall (\ref{eq_mtrasindenity}). Denote 
\begin{align}
	&\caS_{\alpha\beta}(z)\deq z^{2}L_{\mu_{\beta}}'(\Omega_{\alpha}(z))L_{\mu_{\alpha}}'(\Omega_{\beta}(z))-1,	\label{eq_defn_salphabeta} \\
	&\caT_{\alpha}(z)\deq \frac{1}{2}\left[zL_{\mu_{\beta}}''(\Omega_{\alpha}(z))L_{\mu_{\alpha}}'(\Omega_{\beta}(z))
	+(zL_{\mu_{\beta}}'(\Omega_{\alpha}(z)))^{2}L_{\mu_{\alpha}}''(\Omega_{\beta}(z))\right],	\label{eq_defn_talpha} \\
	&\caT_{\beta}(z)\deq \frac{1}{2}\left[zL_{\mu_{\alpha}}''(\Omega_{\beta}(z))L_{\mu_{\beta}}'(\Omega_{\alpha}(z))+(zL_{\mu_{\alpha}}'(\Omega_{\beta}(z)))^{2}L_{\mu_{\beta}}''(\Omega_{\alpha}(z))\right]. \nonumber
\end{align}
By replacing the pair $(\alpha,\beta)$ with $(A,B)$, we can define $\Phi_{AB}$, $\caS_{AB}$, $\caT_{A}$, and $\caT_{B}$ analogously. {We remark that analogous quantities have been defined and used for the additive model in \cite{BEC}; see equation (3.1) therein. }

\begin{rem}
	We provide a few remarks on the usefulness for the above quantities. 
	First, the edges $E_{\pm}$ of $\mu_{\alpha} \boxtimes \mu_{\beta}$ can be  completely characterized by the equation $\mathcal{S}_{\alpha \beta}(E \pm)=0$; see \cite[Section 5]{JHC} for mode details. Second, the above quantities are closely connected with the subordination equation system (\ref{eq_suborsystemPhi}). Let $\mathrm{D}$ be the differential operator with respect to $\omega_{1}$ and $\omega_{2}.$ Then we 
	find that the first derivative of $\Phi_{\alpha\beta}$ is given by
	\beq\label{eq_diferentialoperator}
	\mathrm{D}\Phi_{\alpha\beta}(\omega_{1},\omega_{2},z)\deq\begin{pmatrix}
		-z^{-1} & L'_{\mu_{\alpha}}(\omega_{2}) \\
		L'_{\mu_{\beta}}(\omega_{1}) & -z^{-1}
	\end{pmatrix}.
	\eeq
	Moreover, its determinant is equal to $-z^{-2}\caS_{\alpha\beta}(z)$ at the point $(\Omega_{\alpha}(z),\Omega_{\beta}(z),z)$. Similarly, using $\Phi_{\alpha\beta}(\Omega_{\alpha}(z),\Omega_{\beta}(z),z)=0$, we find that
	\begin{align*}
		\caT_{\alpha}(z)&= z\left[\frac{\partial}{\partial\omega_{1}}\det \mathrm{D} \Phi_{\alpha\beta}(\omega_{1},zL_{\mu_{\beta}}(\omega_{1}),z)\right]_{\omega_{1}=\Omega_{\alpha}(z)}, \\
		\caT_{\beta}(z)&=z\left[\frac{\partial}{\partial\omega_{2}}\det \mathrm{D} \Phi_{\alpha\beta}(zL_{\mu_{\alpha}}(\omega_{2}),\omega_{2},z)\right]_{\omega_{2}=\Omega_{\beta}(z)}.
	\end{align*}
	As will be seen later in the proof, we need to show that $\Omega_{\alpha}$ and $\Omega_\beta$ are close to $\Omega_A$ and $\Omega_B,$ respectively. The arguments are based on the stability analysis of $\Phi_{AB},$ which require sharp estimates of the above quantities.  
	
\end{rem}

We collect the key properties of the subordination functions in Proposition \ref{prop:stabN}.  {It is the counterpart of Proposition 3.1 of \cite{BEC} which concerns the additive model.  } For the ease of statements, we only provide the results near the upper edge $E_+$ defined in (\ref{eq_edge}). Similar results hold for the lower edge $E_-.$ 
For $z=E+\ii \eta \in \mathbb{C}_+,$ denote 
\begin{equation}\label{eq_defnkappa}
	\kappa \equiv \kappa(z):=|E-E_+|.
\end{equation}
{
	For some given constants $0\leq a\leq b$ and  $0<\tau<\min\{\frac{E_{+}-E_{-}}{2},1\}$, we define the following set of spectral parameters by
	\beq \label{eq_fundementalset}
	\caD_{\tau}(a,b)\deq \{z=E+\ii\eta\in\C_{+}:E_+-\tau \leq E \leq \tau^{-1}, a\leq \eta\leq b\}.
	\eeq
	Further, for any small positive constant $\gamma>0$, we let 
	\beq\label{eq_eltalgamma}
	\eta_{L}\equiv\eta_{L}(\gamma)\deq N^{-1+\gamma},
	\eeq and let $\eta_{U}>1$ be a large $N$-independent constant. }
\begin{prop}\label{prop:stabN}
	Suppose Assumptions \ref{assu_limit} and \ref{assu_esd} hold. Then for any fixed small constant $\tau>0$ and sufficiently large $N$, the following hold:
	\begin{itemize}
		\item[(i)] There exists some constant $C>1$ such that
		\begin{align*}
			&\min_i\absv{a_{i}-\Omega_{B}(z)}\geq C^{-1},& &\min_i \absv{b_{i}-\Omega_{A}(z)}\geq C^{-1},&\\
			& C^{-1}\leq \absv{\Omega_{A}(z)}\leq C,& &C^{-1}\leq \absv{\Omega_{B}(z)}\leq C,&
		\end{align*}
		uniformly in $z\in\caD_{\tau}(\eta_{L},\eta_{U})$.
		
		\item[(ii)] For all $z \in\caD_{\tau}(\eta_{L},\eta_{U})$, we have 
		\beqs
		\im m_{\mu_{A}\boxtimes\mu_{B}}(z)\sim\left\{
		\begin{array}{lcl}
			\sqrt{\kappa+\eta}, & \ \text{if }& E\in\supp\mu_{A}\boxtimes\mu_{B},\\
			\dfrac{\eta}{\sqrt{\kappa+\eta}}, &\ \text{if }& E\notin\supp\mu_{A}\boxtimes\mu_{B}.
		\end{array}
		\right.
		\eeqs
		
		\item[(iii)] For all $z\in\caD_{\tau}(\eta_{L},\eta_{U}),$ we have the following bounds for $\caS_{AB}$, $\caT_{A}$, and $\caT_{B}$, 
		\begin{align*}
			&\caS_{AB}\sim\sqrt{\kappa+\eta},& &\absv{\caT_{A}(z)}\leq C,& &\absv{\caT_{B}(z)}\leq C.&
		\end{align*}
		Furthermore, if $\absv{z-E_{+}}\leq\delta$ for sufficiently small constant $\delta>0$, we also have the lower bounds for $\caT_{A}$ and $\caT_{B}$ such that for some small constant $c>0$ 
		\begin{align*}
			\absv{\caT_{A}(z)}\geq c,	\quad  \absv{\caT_{B}(z)}\geq c.
		\end{align*}

		\item[(iv)] For the derivatives of $\Omega_{A}$, $\Omega_{B}$ and $\caS_{AB}$, we have
		\begin{align*}
			\absv{\Omega_{A}'(z)} \leq C\frac{1}{\sqrt{\kappa+\eta}}, &&
			\absv{\Omega_{B}'(z)}\leq C\frac{1}{\sqrt{\kappa+\eta}}, &&
			\absv{S_{AB}'(z)}\leq C\frac{1}{\sqrt{\kappa+\eta}},
		\end{align*}
		uniformly in $z\in\caD_{\tau}(\eta_{L},\eta_{U})$.
	\end{itemize}
\end{prop}

Proposition \ref{prop:stabN} will be proved in Section \ref{proof_prop32}.
First, the first equation in (i) states that the subordination functions are well separated from the supports of $\mu_A$ and $\mu_B.$ This regularity further  implies the square root behavior of the subordination functions; see  Lemmas \ref{lem:suborsqrt} and \ref{lem:OmegaBound}.
The second equation in (i) shows that the subordination functions are bounded from both below and above. Second, (ii) offers a standard estimate for the Stieltjes transform, which follows from the square root behavior of $\mu_A \boxtimes \mu_B.$ Third, (iii) and (iv) prepare some estimates for the related quantities. All these will be used to prove the closeness between $\Omega_\alpha, \Omega_{\beta}$ and $\Omega_A, \Omega_B$; see Section \ref{sec:omegabound} for more details.

\subsection{Local laws for free multiplication of random matrices} \label{sec_locallawresults}
In this subsection, we state the results of the local laws. We will need the notion of \emph{stochastic domination.} It was first introduced in \cite{MR3119922} and subsequently used in many works on random matrix theory. It simplifies the
presentation of the results by systematizing statements of the form ``$X_N$ is bounded by $Y_N$ with
high probability up to a small power of $N$''. 

\begin{defn}
	For two sequences of random variables $\{X_{N}\}_{N\in\N}$ and $\{Y_{N}\}_{N\in\N}$, we say that $X_{N}$ is \emph{stochastically dominated} by $Y_{N},$ written as  $X_{N}\prec Y_{N}$ or $X_{N}=\rO_{\prec}(Y_{N}),$ if for all (small) $\epsilon>0$ and (large) $D>0$, we have
	\beqs
	\prob{\absv{X_{N}}\geq N^{\epsilon}\absv{Y_{N}}}\leq N^{-D},
	\eeqs
	for sufficiently large $N\geq N_{0}(\epsilon,D)$. If $X_{N}(v)$ and $Y_{N}(v)$ depend on some common parameter $v$, we say $X_{N}\prec Y_{N}$ \emph{uniformly in $v$} if the threshold $N_{0}(\epsilon,D)$ can be chosen independent of the parameter $v$. Moreover, we say an event $\Xi$ holds with high probability if for any constant $D>0,$ $\mathbb{P}(\Xi) \geq 1-N^{-D}$ for large enough $N.$
\end{defn}

The following theorem establishes the local laws for the matrices $H, \wt{H}, \mathcal{H}$ and $\wt{\caH}$ near the upper edge $E_+.$ {Analogous results can be obtained for the lower edge $E_-.$ } {It can be regarded as the counterpart of \cite[Theorem 2.5]{BEC}.} 


{ 
	\begin{thm}\label{thm:main}
		Suppose  Assumptions  \ref{assu_limit} and \ref{assu_esd} hold. Let $\tau$ and $\gamma$ be fixed small positive constants. Given any deterministic vector $\bm{v}=(v_1, \cdots, v_N) \in \mathbb{C}$ such that $\| \bm{v} \|_{\infty} \leq 1,$
		\\	
		\noindent{(1)}. For the matrix $H$ and its resolvent $G(z),$ we have 
		\beq \label{eq:main}
		\Absv{\frac{1}{N}\sum_{i=1}^{N}v_{i}\left(zG_{ii}(z)+1-\frac{a_{i}}{a_{i}-\Omega_{B}(z)}\right)}\prec \frac{1}{N\eta},
		\eeq
		uniformly in $z\in\caD_{\tau}(\eta_{L},\eta_{U})$ with $\eta_L$ in (\ref{eq_eltalgamma}) and any fixed constant $\eta_{U}$. Particularly, 
		\beq \label{eq_averagelawmain}
		\absv{m_{H}(z)-m_{\mu_{A}\boxtimes\mu_{B}}(z)}\prec\frac{1}{N\eta}.
		\eeq		
		{
			Moreover, we have the following entry-wise local law 
			\begin{equation}\label{eq_off1}
				\max_{i, j} \left|G_{ij}(z)-\delta_{ij}  \frac{\Omega_B(z)}{z(\Omega_B(z)-a_i)}  \right| \prec \frac{1}{\sqrt{N \eta}}. 
		\end{equation}}
		Similar results hold true by replacing $H$ and $G(z)$ with $\wt{H}$ and $\wt{G}(z)$, respectively. \\ 
		\noindent{(2).}   For the matrix $\mathcal{H}$ and its resolvent $\mathcal{G}(z),$ we have  
		\beq \label{eq:main1}
		\Absv{\frac{1}{N}\sum_{i=1}^{N}v_{i}\left(z\mathcal{G}_{ii}(z)+1-\frac{b_{i}}{b_{i}-\Omega_{A}(z)}\right)}\prec \frac{1}{N\eta},
		\eeq
		and
		\beqs
		\absv{m_{\mathcal{H}}(z)-m_{\mu_{A}\boxtimes\mu_{B}}(z)}\prec\frac{1}{N\eta},
		\eeqs		
		uniformly in $z\in\caD_{\tau}(\eta_{L},\eta_{U})$. Moreover, for the entry-wise local law, we have 
		\begin{equation}\label{eq:mainoff2}
			\max_{i, j} \left|\mathcal{G}_{ij}(z)-\delta_{ij} \frac{\Omega_A(z)}{z(\Omega_A(z)-b_i)} \right| \prec \frac{1}{\sqrt{N \eta}}. 
		\end{equation}
		Similar results hold true by simply replacing $\caH$ and $\caG(z)$ with $\wt{\caH}$ and $\wt{\caG}(z)$, respectively.
	\end{thm}
	
}

{
	\begin{rem} 
		We provide a few remarks for Theorem \ref{thm:main}.  First, since the goal of \cite{BEC} is to establish the spectral rigidity for the additive model, they only need the averaged local laws so the entry-wise local laws are not presented explicitly there. However, it is easy to check that such entry-wise laws also hold for the additive model following their proofs. In fact, the entry-wise local laws are stated explicitly for the additive model in the regular bulk in their work \cite{BAO2} (see Theorem 2.5 therein). Second, it is not hard to check that for the entry-wise local laws, the convergence rates can be replaced by 
		\begin{equation*}
			\sqrt{\frac{\im m_{\mu_A \boxtimes \mu_B}(z)}{N\eta}}+\frac{1}{N\eta}, 
		\end{equation*}
		which matches the typical forms of the bounds of local laws in the random matrix theory literature; see the monograph \cite{2017dynamical}. We keep the current form to highlight the similarities  between our multiplicative model and the additive model in \cite{BEC}. Third, in \cite{BEC}, the authors also state the averaged local law far away from the edges such that the error bound $(N \eta)^{-1}$ could be replaced by $(N(\kappa+\eta))^{-1}.$ Such an improvement also holds for our multiplicative model. In fact, in this case, we can also improve the convergence rates for the entry-wise local laws from $(\sqrt{N\eta})^{-1/2}$ to $N^{-1/2}(\kappa+\eta)^{-1/4}.$ Together with these results, we will be able to study the deformed invariant model \cite{belinschi2017}. These will be studied in our future works. Finally, while we restricted ourselves to the edge local law for the sake of simplicity, the same argument can be used to prove  Theorem \ref{thm:main} in the bulk, by replacing the spectral domain $\mathcal{D}_{\tau}(\eta_L, \eta_U)$ with  
			\beq \label{eq_bulkspectraldomain}
			\caD_{\mathrm{bulk}}\deq\{z=E+\ii\eta\in\C_{+}: E_{-}+\tau<E<E_{+}-\tau,\eta_{L}<\eta<\eta_{U}\},
			\eeq
			for any fixed constant $\tau>0.$  In fact, the proof will be simpler in this regime. We refer the readers to Remark \ref{rem:bulk2} for more details.
	\end{rem}
}

Next, we state two important consequences of the local laws: edge eigenvalue rigidity and edge eigenvector delocalization. Denote $\gamma_j$ as the $j$-th $N$-quantile (or classical location) of $\mu_{\alpha} \boxtimes \mu_{\beta}$ such that 
\begin{equation*}
	\int_{\gamma_j}^\infty d \mu_{\alpha} \boxtimes \mu_{\beta}(x)=\frac{j}{N}.
\end{equation*}
Similarly, we denote $\gamma_j^*$ to be the $j$-th $N$-quantile of $\mu_A \boxtimes \mu_B.$ Recall that $\lambda_1 \geq \lambda_2 \geq \cdots \geq  \lambda_N$ are the eigenvalues of $AUBU^*.$

{\begin{thm} \label{thm_rigidity} Suppose Assumptions \ref{assu_limit} and \ref{assu_esd} hold true. For any small constant $0<c<1/2,$ we have that for all $1 \leq i \leq cN,$
		\begin{equation*}
			|\lambda_i-\gamma_i^*| \prec i^{-1/3}N^{-2/3}. 
		\end{equation*}  
		Moreover, the same conclusion holds if $\gamma_i^*$ is replaced with $\gamma_i.$
	\end{thm}

	Denote the singular value decomposition (SVD) for $Y$ in (\ref{eq_defndatamatrixtype}) as 
	\begin{equation*}
		Y=\sum_{i=1}^N \sqrt{\lambda_i} \ub_i \vb_i^*,
	\end{equation*}
	where $\{\ub\}_i$ and $\{\vb_i\}$ are the left and right singular vectors of $Y,$ respectively. 
	\begin{thm} \label{thm_delocalization}
		Suppose Assumptions \ref{assu_limit} and \ref{assu_esd} hold true. For any fixed small constant $0<c<1/2,$ we have that for all $1 \leq i \leq cN,$
		\begin{equation*}
			\max_k|\ub_i(k)|^2+\max_\mu|\vb_i(\mu)|^2 \prec \frac{1}{N}. 
		\end{equation*}
	\end{thm}
}

	\begin{rem}
		We provide some further remarks here. First, in the current paper, the deterministic matrices $A$ and $B$ are both assumed to be positive definite so that $A^{1/2}$ and $B^{1/2}$ are well-defined. {This ensures the model is symmetric in the sense that $A$ and $B$ can be interchanged freely, and we often use such an argument along our proofs. In particular, we actually use all four matrices in (\ref{defn_eq_matrices}) and their resolvents; see (\ref{eq:a_i}) for an illustration, where we apply the Ward identity to the resolvents of $\widetilde{H}$ and $\widetilde{\mathcal{H}}$.} Moreover, this symmetry played an important role in \cite{JHC} and in earlier appearances of subordination functions on which our results relied, for example \cite{Biane1997,Belinschi-Bercovici2007}. In this sense, even though taking (only) one of $A$ and $B$ to be non-positive in $H=AUBU^*$ still gives real eigenvalues, some arguments in the present paper cease to work {because the model is no longer symmetric and either $\wt{H}$ or $\wt{\caH}$ in \eqref{defn_eq_matrices} is not defined. Nonetheless, we believe that the result remains true in this case, and it might be possible to prove with an application of linearization trick. We explain more details on difficulties arising in applications of the linearization trick in the next paragraph.}
		
		Second, in \cite{ho2022local} the author proved a weak local law for generic self-adjoint polynomials of $A$ and $\wt{B}$ (recall $\wt{B}=UBU^*$), on the scale of $N^{-1/12}$. This work and its deterministic precursor \cite{Belinschi-Mai-Speicher2017} suggest that an analogous result to the present paper and \cite{BEC} should hold for generic polynomials. The main idea in \cite{ho2022local} was to use the linearization trick to consider the sum of tensors $x_{\alpha}\otimes A+x_{\beta}\otimes \wt{B}$ for suitably finite and Hermitian matrices $x_{\alpha},x_{\beta},$ instead of the given polynomial. Moreover, we point out that this can also apply to the case with a non-positive matrix, say $B$, by considering $\sqrt{A}\wt{B}\sqrt{A}$ as a polynomial of $\sqrt{A}$ and $\wt{B}$. {Although it is feasible that the techniques of the current paper and those in \cite{Bao-Erdos-Schnelli2016,BAO2,BAO3,BEC} can apply to the general model in \cite{ho2022local}, there are two major difficulties in accommodating these arguments to the linearized models.} On one hand, it is a nontrivial task to study the limiting distribution and its regularity for the general free polynomial model. In particular, there {is no known} natural and suitable conditions on $\mu_{\alpha}$ and $\mu_{\beta}$ like Assumption \ref{assu_limit}, and on the generic polynomials so that the free polynomial is regular, especially near the edge. On the other hand, as will be seen in Section \ref{subsec_sketch} below, the proof of local laws relies on many auxiliary scalar quantities; see \eqref{eq_shorhandnotation} and \eqref{eq_pk} for examples. However, for the general model in \cite{ho2022local}, due to the linearization, all these quantities should be matrices instead of scalars. Hence, finding the non-scalar equivalents of these auxiliary quantities in the general setting can be challenging. {At the current stage} we are not aware of a systematic approach to find them in general, except that in the bulk regime some related techniques have been developed for polynomials of Wigner matrices recently in \cite{Erdos-Kruger-Nemish2020}.
	\end{rem}
{
	\subsection{Statistical applications}
	
	In this subsection, we briefly discuss some applications of our results to high dimensional statistics. First, our results can be used to detect the existence of signals in the signal-plus-noise model when the noise part is of the form $A^{1/2}UBU^*A^{1/2}.$ Consider 
	\begin{equation*}
		Y=S+Z,
	\end{equation*}   
	{where $S$ and $Z$ stand for the signal and noise parts, respectively.} Such a model finds important applications in many scientific endeavors. Especially in many cases $S$ is a low-rank symmetric matrix, for example, diffusion tensor imaging (DTI) analysis \cite{schwartzman2008inference}, $\mathbb{Z}_2$ synchronization \cite{javanmard2016phase}, community detection using stochastic block model (SBM) \cite{abbe2017community},  matrix denoising and recovery \cite{donoho2013phase, lelarge2019fundamental} and signal processing \cite{tulino2004random}. While most of the existing literature focuses on the setting that $Z$ is a Wigner matrix, the free multiplicative noise is also considered in the literature \cite{7587390,bun2017}. Therefore, we can apply our results to study the signal-plus-noise model when $Z=A^{1/2}UBU^*A^{1/2},$ that is, 
	\begin{equation}\label{eq_statmodel}
		Y=S+A^{1/2}UBU^*A^{1/2}.
	\end{equation}

	A fundamental task is to recover the signal matrix $S$ from observed sample $Y$ in (\ref{eq_statmodel}), and the very first step is to know whether there exists any such signal. From the random matrix theory viewpoint, the eigenvalues of the signal part $S$ can be viewed as outliers which detach from the bulk of $\mu_A \boxtimes \mu_B.$  Our Theorem \ref{thm_rigidity} can be used to achieve this goal, especially we can employ the following Onatski's statistic \cite{onatski2009testing} 
	\begin{equation}\label{eq_tform}
		\mathbb{T}:=\max_{1 \leq i \leq C}\frac{\lambda_i(Y)-\lambda_{i+1}(Y)}{\lambda_{i+1}(Y)-\lambda_{i+2}(Y)}, 
	\end{equation} 
	where $C>0$ is some pre-chosen large integer and $\{\lambda_i(Y)\}$ are the eigenvalues of $Y$ which are ordered decreasingly. {We now explain how the statistic $\mathbb{T}$ works by considering the simplest case with rank-one alternative.} Note that rank-one $S$ already has important applications in $\mathbb{Z}_2$ synchronization \cite{javanmard2016phase} and SBM \cite{abbe2017community}. Formally, we consider {the hypothesis test
		\begin{equation}\label{eq_testingproblem}
			\mathbf{H}_0: S=0 \qquad \mathrm{versus} \qquad \mathbf{H}_a: S=d \mathbf{u} \mathbf{u}^*, 
		\end{equation}
		for some large constant $d>0.$} On one hand, according to Theorem \ref{thm_rigidity}, under the null hypothesis $\mathbf{H}_0,$ the statistic $\mathbb{T}$ should satisfy that $\mathbb{T}=1+o_{\prec}(1)$. On the other hand, if $\mathbf{H}_a$ holds, when $d$ is above some threshold, that is, the signal is relatively strong so that $\lambda_1(Y)$ detaches from the spectrum of $\mu_A \boxtimes \mu_B$, we shall have that with high probability $\mathbb{T}>1+\tau$ for some constant $\tau>0.$ Consequently, we can use $\mathbb{T}$ to detect the existence of the signal matrix. We mention that the exact characterization of $d$, that is, the BBP transition, often requires a sophisticated perturbation argument involving optimal local laws, Theorem \ref{thm:main}. {Moreover, to formally perform the test \eqref{eq_testingproblem}, we need to find the distribution of $\mathbb{T}$.} Since these are beyond the focuses of the current paper, we defer these problems in future works.
	
	Second, some simple extensions of our results can be applied to provide some insights on the performance of random sketching in the setting of high dimensional least square regression \cite{ELnips}. Suppose that we observe $N$ data points $(x_i, y_i),$ where $x_i \in \mathbb{R}^p$ are the predictors, and $y_i \in \mathbb{R}$ are the responses.  Consider the linear model that $y_i=x_i^\top \beta +\epsilon_i,$ where $\beta \in \mathbb{R}^p$ is an unknown parameter and $\epsilon_i's$ are the white noise error.  The ordinary least squares (OLS) estimator for $\beta$ can be written as
	\begin{equation*}
		\widehat{\beta}=(X^\top X)^{-1}X^\top Y,
	\end{equation*}
	where $X \in \mathbb{R}^{N \times p}$ collects all $x_i$ and $y \in \mathbb{R}^N$ collects all $y_i, 1 \leq i \leq N.$ The OLS estimator is a gold standard when rank($X$)$=p.$ However, when both $N$ and $p$ are large, it is computationally expensive to get the OLS estimator. In the high dimensional setting, sketching is an effective approach to reduce the size of the problem by multiplying an $M \times N$ matrix $S$ to obtain the sketched data $(\widetilde{X}, \widetilde{Y})=(SX, SY).$ Then the sketched OLS estimator is \cite{ELnips}
	\begin{equation*}
		\widehat{\beta}_s=(\widetilde{X}^\top \widetilde{X})^{-1} \widetilde{Y}.
	\end{equation*}   
	The computational cost will drop from $Np^2$ to $Mp^2$ by introducing the sketching matrix $S.$ The goal of sketching is to find an $M<N$ so that the performance of $\widehat{\beta}$ and $\widehat{\beta}_s$ is similar.   
	
	One popular and efficient choice for $S$ is the truncated Haar orthogonal matrix in the sense that $S$ is a submatrix of an $N \times N$ Haar orthogonal matrix. To analyze the performance of $\widehat{\beta}_s,$ the core is to study the projection matrix $\widetilde{X}^\top \widetilde{X}.$ It is not hard to see that we can rewrite $\mathsf{D}_s:=\widetilde{X}^\top \widetilde{X}$ as (see Section A.7 of \cite{ELnips})
	\begin{equation}\label{eq_defnds}
		\mathsf{D}_s=\begin{pmatrix}
			I_M & 0\\
			0 & 0
		\end{pmatrix} U
		\begin{pmatrix}
			\Lambda_p & 0\\
			0 & 0
		\end{pmatrix}U^*
		\begin{pmatrix}
			I_M & 0\\
			0 & 0
		\end{pmatrix},
	\end{equation}
	where $\Lambda_p$ is the diagonal matrix containing the  nontrivial eigenvalues of $X^\top X. $ $\mathsf{D}_s$ includes the truncated Haar matrices here due to the block structure of the deterministic matrices. 
	
	Based on the above discussion, we see that the core is to analyze the matrix $\mathsf{D}_s.$ In fact, the above model is of the form (\ref{eq_defndatamatrixtype}) by setting
	\begin{equation}\label{eq_abtransform}
		A=
		\begin{pmatrix}
			I_M & 0\\
			0 & 0
		\end{pmatrix}, \ B=\begin{pmatrix}
			\Lambda_p & 0\\
			0 & 0
		\end{pmatrix}.
	\end{equation}
	For this particular case, even through the positive definite assumptions in Section \ref{sec_locallawresults} are slightly violated, the results still hold true near the upper edge.  These results, especially Theorem \ref{thm:main} can be used to establish the convergent rates for the variance efficiency (VE) of $\widehat{\beta}_s$, that is, the increase in parameter estimation error compared to using $\widehat{\beta}$ directly. For instance, together with Theorem 2.3 of \cite{ELnips}, we will be able to show that when $p$ is comparable to both $M$ and $N$ 
	\begin{equation}\label{eq_pointresults}
		\frac{\| \widehat{\beta}_s-\beta \|^2}{ \| \widehat{\beta}-\beta\|^2}=\frac{N-p}{M-p}+O_{\prec}\left(\frac{1}{\sqrt{N}}\right).
	\end{equation}
	(\ref{eq_pointresults}) can be used as a starting point to choose a numerically efficient value of $M$ for finite sample study. A rigorous discussion is out of the scope of this paper and we will pursue this direction in future works.
	
	Finally, we mention that the key matrices of many other problems are also in the form of (\ref{defn_freemultiplicationmodel}) or $A^{1/2}UBU^* A^{1/2}.$ For example, many of the nonlinear kernel-based data acquisition algorithms can be reduced to studying a random matrix of the form (\ref{defn_freemultiplicationmodel}) \cite{DW}. Consequently, our results can be potentially applied to check whether common signals  have been properly captured by two different sensors using a statistic similar to (\ref{eq_tform}) \cite{ding2021kernel}. Moreover, $A^{1/2}UBU^*A^{1/2}$ is also a natural model for spatiotemporal data analysis \cite{bun2017}, where $A$ and $B$ are respectively the spatial and
	temporal covariance matrices. In practice, a spiked model analogous to \cite{DYaos} is more reasonable for real applications. The results in this paper are key ingredients to study such a problem. We will pursue these directions in future works.

}

\section{General structure of the proof} \label{sec_proofroute}

\subsection{Partial randomness decomposition}
{As mentioned earlier, the partial randomness decomposition has been used as an important asset to handle the Haar random matrix in the additive model \cite{BAO3,BEC,bao2019}.} For our multiplicative model, we also need to use this tool. The partial randomness decomposition can be regarded as the counterpart of the Schur's complement, which plays  a central role in the proof of the local laws \cite{bao2015,ding2018,MR3704770,lee2016,yang2019} when $X$  in (\ref{eq_defndatamatrixtype}) has i.i.d. entries. In what follows, we focus on introducing this technique for Haar unitary matrix. We will also briefly discuss how the arguments apply to Haar orthogonal matrix.

Let $U$ be an $N \times N$ Haar unitary random matrix. For all $i\in\llbra1,N\rrbra$, define $\bsv_{i}\deq U\bse_{i}$ as the $i$-th column vector of $U$ and $\theta_{i}$ as the argument of $\bse_{i}\adj\bsv_{i}$. Following \cite{PM1}, we denote
\beq\label{eq_prd}
U^{\angi}\deq -\e{-\ii\theta_{i}}R_{i}U, \quad\text{where}\ R_{i}\deq I-\bsr_{i}\bsr_{i}\adj, \quad \bsr_{i}\deq \sqrt{2}\frac{\bse_{i}+\e{-\ii\theta_{i}}\bsv_{i}}{\norm{\bse_{i}+\e{-\ii\theta_{i}}\bsv_{i}}}  .
\eeq 
Since $\|\bsr_i \|^2=2$, we have that $R_i$ is a Householder reflection. Consequently, $R_i^*=R_i$ and $R_i^2=I.$ Furthermore, it is easy to see that $U^{\angi} \bm{e}_i=\bm{e}_i$ and $\bm{e}_i^* U^{\angi}=\bm{e}_i^*.$  This implies that $U^{\angi}$ is a unitary block-diagonal matrix. In other words, $U^{\angi}_{ii}=1$ and the $(i,i)$-matrix minor  of $U^{\angi}$ is Haar distributed  on $U(N-1) $ and $\bm{v}_i$ is uniformly distributed on the $N-1$ unit sphere. 
Denote
\beq\label{eq_hatubhatu}
\wt{B}^{\angi}\deq U^{\angi}B(U^{\angi})\adj.
\eeq
Since $\bsv_{i}$ is uniformly distributed on the unit sphere $\bbS^{N-1}_{\C}$, we can find a Gaussian vector $\wt{\bsg}_{i}\sim\caN_{\C}(0,N^{-1}I_{N})$ such that 
\beq\label{eq_defnvb}
\bsv_{i}= \frac{\wt{\bsg}_{i}}{\norm{\wt{\bsg}_{i}}}.
\eeq
Armed with the above Gaussian vector, we further define
\begin{align}\label{eq_prd2}
	&\bsg_{i}\deq \e{-\ii\theta_{i}}\wt{\bsg}_{i},	\ \ 
	\bsh_{i}\deq \frac{\bsg_{i}}{\norm{\bsg_{i}}}=\e{-\ii\theta_{i}}\bsv_{i}, \ \ \nonumber 
	\ell_{i}\deq \frac{\sqrt{2}}{\norm{\bse_{i}+\bsh_{i}}}, 	\\
	&	\mr{\bsg}_{i}\deq \bsg_{i}-g_{ii}\bse_{i},	 \ \
	\mr{\bsh}_{i}\deq \bsh_{i}-h_{ii}\bse_{i}.
\end{align}
It is easy to see from (\ref{eq_prd}) that 
\beq\label{eq_prd1}
\bm{r}_i=\ell_i(\bm{e}_i+\bm{h}_i), \	R_{i}\bse_{i}=-\bsh_{i}, \ R_{i}\bsh_{i}=-\bse_{i},
\eeq
which further implies that 
\beq\label{eq_hbwtr}
\begin{aligned}
	\bsh_{i}\adj \wt{B}^{\angi}R_{i}=-\bse_{i}\adj \wt{B},	\quad 
	\bse_{i}\adj\wt{B}^{\angi}R_{i}=	-\bsh_{i}\adj \wt{B}=-b_{i}\bsh_{i}\adj .
\end{aligned}
\eeq
The above equations provide several convenient identities for the Haar unitary matrix. Moreover, since $\wt{B}^{\angi}$ is independent of both $\bm{h}_i$ and $R_i,$ we can establish accurate large deviation estimates for quantities related to (\ref{eq_hbwtr}); see Section \ref{sec_largedeviation} for more details. Finally, 
for Haar orthogonal random matrix on $O(N)$, the only difference lies in the partial randomness decomposition. In fact, we can decompose an orthogonal matrix $U$ in the same way as in (\ref{eq_prd}), except that the factor $\e{-\ii\theta_{i}}$ in \eqref{eq_prd} should be replaced by $\mathrm{sgn}(\bse_{i}\adj \bsv_{i})$. We refer the readers to \cite[Appendix A]{BAO2} for more details. 
\subsection{Sketch of  the proof route}\label{subsec_sketch}
In this subsection, we summarize the main route of the proof. {Our proof basically follows \cite{BEC} and  we focus on explaining how to adapt their proof strategies to study the multiplicative model. For a proof route of the additive model, we refer the readers to \cite[Section 4.2]{BEC}. }


Recall (\ref{eq_defntrace}). Without loss of generality, till the end of the paper, we assume that both $A$ and $B$ are normalized such that $\tr A =\tr B=1.$ First, we introduce the random equivalents of the subordination functions $\Omega_A$ and $\Omega_B$ in terms of the resolvents. They are the starting points of the arguments and the counterparts of the additive model as in equation (5.2) of \cite{BEC}.
\begin{defn}[Approximate subordination functions]\label{defn_asf}
	For $z\in\CR$, we define
	\beqs
	\Omega_A^c \equiv \Omega_{A}^{c}(z) \deq\frac{z\tr \wt{B}G}{1+z\tr G}, \ \ \Omega_B^c \equiv
	\Omega_{B}^{c}(z) \deq\frac{z\tr AG}{1+z\tr G}=\frac{z\tr\caG\wt{A}}{1+z\tr\caG}.
	\eeqs
\end{defn}
As mentioned earlier in Section \ref{sec_introduction}, while the final goal is to prove (\ref{eq:main}), following the strategy of \cite{BEC}, we actually work with the approximate subordination functions in Definition \ref{defn_asf} and several auxiliary quantities.  To identity these  auxiliary quantities, we need several crucial decompositions. By  replacing $\Omega_B$ with $\Omega_B^c$ in (\ref{eq:main}), we observe that
\begin{align}\label{eq:Lambda}
	(zG_{ii}+1)-&\frac{a_{i}}{a_{i}-\Omega_{B}^{c}}
	=a_{i}(\wt{B}G)_{ii}-\frac{a_{i}}{a_{i}-\Omega_{B}^{c}} \nonumber	\\
	&=\frac{a_{i}}{(1+z\tr G)(a_{i}-\Omega_{B}^{c})}	
	\left((z\tr G+1)(a_{i}-\Omega_{B}^{c})(\wt{B}G)_{ii}-(z\tr G+1)\right) \nonumber	\\
	&=\frac{{a_{i}}z}{(1+z\tr G)(a_{i}-\Omega_{B}^{c})}\left(G_{ii}\tr(A\wt{B}G)-\tr (GA)(\wt{B}G)_{ii}\right),
\end{align}
where we used the elementary identity $(HG)_{ii}-zG_{ii}=a_{i}(\wt{B}G)_{ii}-zG_{ii}=1.$ 
In light of Proposition \ref{prop:stabN} and \eqref{eq:Lambda},  to prove (\ref{eq:main}), it suffices to control
\beq\label{eq_defnq}
Q_{i}\deq G_{ii}\tr(A\wt{B}G)-\tr(GA)(\wt{B}G)_{ii},
\eeq
and show that $\Omega_B$ and $\Omega_B^c$ are sufficiently close. {Note that $Q_i$ is the counterpart of equation (4.11) of \cite{BEC}.}

We first present detailed decomposition of $Q_i.$ In particular, following \cite[Section 4.2]{BEC}, we discuss how to decompose and explore the independence structure of $(\wt{B} G)_{ii}$ using the partial randomness decomposition. Using (\ref{eq_hatubhatu}) and (\ref{eq_prd2}), we introduce the notations 
\beq\label{eq_shorhandnotation}
S_{i}\deq \bsh_{i}\adj \wt{B}^{\angi}G\bse_{i},\quad \mr{S}_{i}\deq \mr{\bsh}_{i}\adj \wt{B}^{\angi}G\bse_{i},\quad T_{i}\deq \bsh_{i}\adj G\bse_{i}=\e{\ii\theta_{i}}\bse_{i}\adj U\adj G\bse_{i},\ \mr{T}_{i}\deq \mr{\bsh}_{i}\adj G\bse_{i},
\eeq
where we used $(e^{-\ii \theta_i})^*=e^{\ii \theta_i}.$ By the construction of $U^{\angi}$ in (\ref{eq_prd}) and (\ref{eq_prd1}), we find that 
\begin{equation*}
	(\wt{B}G)_{ii}=-\bsh_{i}\adj\wt{B}^{\angi}R_{i}G\bse_{i}.
\end{equation*} 
Using the definition of $R_i$ in (\ref{eq_prd}) and $\bm{r}_i=\ell_i(\bm{e}_i+\bm{h}_i)$, we have the following expansion
\begin{equation*}
	(\wt{B}G)_{ii}=-\bsh_{i}\adj\wt{B}^{\angi}G\bse_{i} 
	+\ell_{i}^{2}\bsh_{i}\adj\wt{B}^{\angi}(\bse_{i}+\bsh_{i})(\bse_{i}+\bsh_{i})\adj G\bse_{i},
\end{equation*}
Moreover,  utilizing the notations in (\ref{eq_shorhandnotation}), we can write that 
\begin{equation*}
	(\wt{B}G)_{ii}=-S_{i}+\ell_{i}^{2}(\bsh_{i}\adj\wt{B}^{\angi}\bse_{i}+\bsh_{i}\adj\wt{B}^{\angi}\bsh_{i})(G_{ii}+T_{i}). 
\end{equation*}
Further, since $R_i$ is a projection satisfying (\ref{eq_prd1}), we have that 
\begin{align}\label{eq:BGii-S}
	(\wt{B}G)_{ii}
	=&-S_{i}+\ell_{i}^{2}(-b_{i}\bsh_{i}\adj R_{i}\bsh_{i}+\bsh_{i}\adj\wt{B}^{\angi}\bsh_{i})(G_{ii}+T_{i}) \nonumber	\\
	=&-S_{i}+\ell_{i}^{2}(b_{i}h_{ii}+\bsh_{i}\adj\wt{B}^{\angi}\bsh_{i})(G_{ii}+T_{i}).
\end{align}
{We will see later in our proof (e.g. (\ref{eq:BGii})), the discussion boils down to 
	controlling $S_{i}$ and $T_{i}$, which are the counterparts of equation (4.10) of \cite{BEC}.}

In the actual proof, inspired by the arguments in \cite{BEC}, instead of working directly with $S_i$ and $T_i,$ we deal with the following quantities
\beq\label{eq_pk}
\begin{aligned}
	&P_{i}\deq Q_i+(G_{ii}+T_{i})\Upsilon,
	\\
	&K_{i}\deq	T_{i}+\tr(GA)(b_{i}T_{i}+(\wt{B}G)_{ii})-\tr(GA\wt{B})(G_{ii}+T_{i}),
\end{aligned}
\eeq
where $\Upsilon$ is defined as 
\beq\label{eq_defnupsilon}
\Upsilon\deq (\tr(GA\wt{B}))^{2}-\tr(GA)\tr(\wt{B}GA\wt{B})-\tr(GA\wt{B})+\tr(GA).
\eeq
{The analogs of the above quantities for the additive model have been defined in equations (4.12), (4.14) and (4.15) of \cite{BEC}.} On one hand, $P_i$ and $K_i$ are closely related to $S_i$ and $T_i$. On the other hand, they are easily to be handled with. In fact, using $GA\wt{B}=A\wt{B}G=zG+I$ and $\tr B=\tr \wt{B}=1$, we recognize that $\Upsilon$ is an average of $Q_i,$ that is, 
\begin{align} \label{eq_averageqdefinitionupsilon}
	\Upsilon&=\tr(A\wt{B}G)(\tr(zG+I)-1)-\tr(GA)(\tr(\wt{B}(zG+I))-\tr{\wt{B}}) \nonumber \\
	&=z\left(\tr (G)\tr(A\wt{B}G)-\tr(GA)\tr(\wt{B}G)\right) =\frac{z}{N}\sum_{i}Q_{i}.
\end{align}
The proof of the estimates of the above quantities relies on a two-level approach, which is commonly used in the proofs of local laws for random matrices. On the first level, we provide bounds for a fixed spectral parameter $z$  under the condition that $G_{ii}$, $\caG_{ii}$  and $T_i$ satisfy a weak a priori bound
(c.f., Assumption \ref{assu_ansz}). On the second level, we will verify the above priori bound and further prove they hold uniformly in $z.$

On the first level, the proof strategy contains three steps. In Step 1,  we establish recursive estimates for the high moments of $P_i$ and $K_i,$ that is, Proposition \ref{prop:entrysubor}, {which is the counterpart of Proposition 5.1 of \cite{BEC}.} The main idea is to employ the (Gaussian) integration by parts with respect to the coordinates of $\bsh_i$ since $\wt{B}^{\angi}$ is independent of $\bsh_i.$ This step concludes that all the quantities $P_i, K_i, T_i, Q_i$ and $\Upsilon$ can be bounded by $(N\eta)^{-1/2}.$ The actual proofs will be presented in Section \ref{sec_entrylaw}. In Step 2, we derive a rough bound on the averaged quantities. Especially, since $Q_i$ is the most fundamental quantity, we focus on the form $N^{-1} \sum d_i Q_i,$ where $d_i$'s are some generic weights. The proof also concerns the recursive moment estimates as in Step 1. This step yields that the averaged quantity can be bounded by $ \sqrt{\im m_H(z)} (N\eta)^{-1},$ which improves the bounds from Step 1.  The arguments will be given in Section \ref{sec_roughfa}. In Step 3, we prove that for some specific weights $d_i$ (c.f. (\ref{eq_optimalfaquantitiescoeff}) and (\ref{eq_optimalfaquantitiescoeffextra})), the averaged quantities can be bounded by $\im m_H(z) (N\eta)^{-1}.$ { Note that the weights we choose here are the counterparts of equation (7.12) of \cite{BEC}. }As a byproduct, we obtain a priori bound for   $|\Omega_B-\Omega_B^c|$ and $|\Omega_A-\Omega_A^c|.$ In fact, controlling the differences above can also be reduced to $\Upsilon$ due to the following decomposition
\begin{align}\label{eq:approx_subor}
	\Omega_{A}^{c}\Omega_{B}^{c}-zM_{\mu_{H}}(z)&=\frac{z^{2}}{(1+zm_{H}(z))^{2}}(\tr(GA)\tr(\wt{B}G)-\tr(G)\tr(A\wt{B}G)) \nonumber \\
	&	=-\frac{z}{(1+zm_{H}(z))^{2}}\Upsilon(z),
\end{align} 
where we used $A \wt{B}G=zG+I.$
All these will be discussed in Section \ref{subsec_stronglocallawfixed}.  

On the second level, the proof consists of two parts. In Part 1, we establish the \emph{weak local laws} by  verifying Assumption \ref{assu_ansz} and prove the   uniformity of the estimates
in $z$ is obtained by a continuity argument.  The weak local laws state that most of the aforementioned quantities, e.g., $P_i, K_i$ and (\ref{eq:Lambda}) can be bounded by $(N \eta)^{-1/2}$ uniformly in $z.$  Moreover, $|\Omega_B-\Omega_B^c|$ and $|\Omega_A-\Omega_A^c|$ can be uniformly bounded by $(N \eta)^{-1/3}.$ The formal arguments are given in Section \ref{subsec:weaklocallaw}.
In Part 2, using the weak local laws, we complete the proof of Theorem \ref{thm:main}, which is referred to as \emph{strong local laws}. The arguments can be found in  Section \ref{sec_proofofstronglocallaw}. Finally, we emphasize that even though the above strategy is sketched for the diagonal entries, the off-diagonal entries can be handled similarly. We discuss this aspect in details at the end of Section \ref{subsec:weaklocallaw}. 

Once Theorem \ref{thm:main} is proved, the other theorems can be justified based on it. First, Theorems \ref{thm_rigidity} can be proved by translating the closeness of the resolvent into the closeness of the eigenvalues and the quantiles of $\mu_{A} \boxtimes \mu_{B}$ using Helffer-Sj{\" o}strand formula, and Theorem \ref{thm_delocalization} can be proved by exploring the imaginary part of the resolvents. 

	\begin{rem}\label{rem:bulk2}
		Before concluding this section, we briefly discuss how the above strategy can be used to obtain the bulk local laws when the spectral parameter $z$ is in $\mathcal{D}_{\text{bulk}}$ defined in
		(\ref{eq_bulkspectraldomain}). Similar results have been proved for the additive model in \cite{BAO3}.
		
		In fact, in the bulk regime, the proof will be easier. The main reason is that when $z \in \mathcal{D}_{\text{bulk}},$ the key parameter $\kappa \equiv \kappa(z):=\min\{|z-E_-|, |z-E_+|\}$ satisfies $\kappa \sim 1$ so that $\sqrt{\kappa+\eta} \sim 1.$ Consequently, according to \cite{JHC}, in contrast to Proposition \ref{prop:stablimit}, the key quantities can be controlled more easily in the sense that $\im m_{\mu_{\alpha}\boxtimes\mu_{\beta}}(z)\sim 1$ and $\absv{\caS_{\alpha\beta}(z)}\sim 1$. Combining these updates with lines of the proof of Proposition \ref{prop:stabN}, we can update the results of Proposition \ref{prop:stabN} by inserting $\kappa \sim 1.$ 
		
		Next, we explain how the two-level approach applies and is easier. On the first level, the bulk regime only needs Steps 1 and 2. The reason is that since when $z \in \mathcal{D}_{\text{bulk}}$ we have $\im m_{\mu_A \boxtimes \mu_B}(z) \sim 1,$ Steps 1 and 2 will prove that $\im m_H(z) \sim 1$ and the averaged quantities will therefore be bounded by $(N \eta)^{-1}$ which is already optimal. On the second level, in contrast to the edge regime where we have to decompose the edge spectral domain according to different scales of $\sqrt{\kappa+\eta}$ as in Section \ref{sec_proofofstronglocallaw}, we can work on the whole bulk spectral domain directly as $\kappa\sim 1$.
	\end{rem}

\section{Entry-wise resolvent subordination}\label{sec_entrylaw}

In this section, we prove a subordination property for the resolvent entries, that is, Proposition \ref{prop:entrysubor}. 	In particular, we prove (\ref{eq:main}) and (\ref{eq:main1}) of Theorem \ref{thm:main} and other related quantities for fixed spectral parameter $z$ with a priori bound, that is, Assumption \ref{assu_ansz}. This completes Step 1 of the first level of our proof route as summarized in Section \ref{subsec_sketch}. Moreover, the proof of Proposition \ref{prop:entrysubor}, especially the recursive moment estimates in Lemma \ref{lem:PKrecmoment}, is a representative formal argument of our proof strategies in the sense that Steps 2 and 3 follow a similar discussion. {Analogous arguments have been made for the additive model in Section 5 of \cite{BEC}. }

We first introduce the assumptions. Denote    
\beq\label{eq_moregeneralerrorbound}
\Lambda_{di} \deq \Absv{zG_{ii}+1-\frac{a_{i}}{a_{i}-\Omega_{B}}},\quad \Lambda_{d}\deq \max_{i}\Lambda_{di},\quad \Lambda_{T}\deq \max_{i}\absv{T_{i}}.
\eeq
Similarly, we define $\Lambda_{di}^{c}$ and $\Lambda_{d}^{c}$ by replacing $\Omega_{B}$ with its approximate $\Omega_{B}^{c}$.  Moreover, denote $\wt{\Lambda}_{di}$, $\wt{\Lambda}_{d}$ and $\wt{\Lambda}_{T}$ as
\beqs
\wt{\Lambda}_{di}\deq \Absv{z\caG_{ii}+1-\frac{b_{i}}{b_{i}-\Omega_{A}}},\quad \wt{\Lambda}_{d}\deq \max_{i}\wt{\Lambda}_{di},\quad \wt{\Lambda}_{T}\deq \max_{i} \absv{\bse_{i}\adj U\caG \bse_{i}}.
\eeqs 
Furthermore, $\wt{\Lambda}_{di}^c$ and $\wt{\Lambda}_{d}^c $ are defined by replacing $\Omega_A$ with $\Omega_A^c.$
In this section, the statements and proofs are based on Assumption \ref{assu_ansz}, which provides a priori bound for the essential quantities. It will be verified in Section \ref{subsec:weaklocallaw}.

\begin{assu}\label{assu_ansz} 
	Recall (\ref{eq_fundementalset}). For the small constant $\gamma>0$ in (\ref{eq_eltalgamma}), fix $z\in\caD_{\tau}(\eta_{L},\eta_{U})$,  we suppose the following hold true;
	\beq\label{eq_locallaweqbound}
	\Lambda_{d}(z)\prec N^{-\gamma/4},\quad	\wt{\Lambda}_{d}(z)\prec N^{-\gamma/4},\quad
	\Lambda_{T}\prec 1,\quad	\wt{\Lambda}_{T}\prec 1.
	\eeq
\end{assu}

We now state the main result of this section.  Throughout the paper, we will consistently use the following control parameter
\beq\label{eq_controlparameter}
\Psi\equiv\Psi(z)\deq \sqrt{\frac{1}{N\eta}}, \ \ \Pi_{i}\equiv\Pi_{i}(z)\deq \sqrt{\frac{\im G_{ii}(z)+\im \caG_{ii}(z)}{N\eta}}.
\eeq
\begin{prop}\label{prop:entrysubor}
	Recall (\ref{eq_pk}). Suppose that the assumptions of Theorem \ref{thm:main} and Assumption \ref{assu_ansz} hold. Fix $z\in\caD_{\tau}(\eta_{L},\eta_{U})$. For all $i\in\llbra1, N\rrbra,$ we have that 
	\beq\label{eq:PKbound}
	\absv{P_{i}(z)}\prec\Psi(z),\quad \absv{K_{i}(z)}\prec \Psi(z).
	\eeq	
	Furthermore, we have
	\beq\label{eq_lambdaother}
	\Lambda_{d}^c(z)\prec\Psi(z),\quad \Lambda_{T}\prec\Psi(z),\quad \wt{\Lambda}_{d}^c\prec\Psi(z),\quad \wt{\Lambda}_{T}\prec\Psi(z),\quad \Upsilon\prec\Psi(z).
	\eeq
\end{prop}

The above proposition provides the estimates for the diagonal entries of the resolvents and 	{is an analog of  Proposition 5.1 of \cite{BEC} for the additive model.} The arguments of the off-diagonal entries are similar and we refer the readers to {the discussion at the end of} Section \ref{subsec:weaklocallaw} for more details.  

{In the rest of this section, we follow the proof strategy of \cite[Proposition 5.1]{BEC} to prove Proposition \ref{prop:entrysubor}.} The following resolvent identities will be frequently used in the proof.
\beq\label{eq:apxsubor}
\begin{aligned}
	(HG)_{ii}-zG_{ii}=a_{i}(\wt{B}G)_{ii}-zG_{ii}=1,\\
	(\caG\caH)_{ii}-z\caG_{ii}=b_{i}(\caG\wt{A})_{ii}-z\caG_{ii}=1.
\end{aligned}
\eeq

\begin{proof}[\bf Proof of Proposition \ref{prop:entrysubor}] We first prove (\ref{eq_lambdaother}) assuming (\ref{eq:PKbound}) holds. The arguments rely on the estimate
	\begin{equation}\label{eq_claimtibound}
		T_i \prec N^{-\gamma/4},
	\end{equation}
	where $\gamma>0$ is introduced in (\ref{eq_locallaweqbound}).
	
	Before proving (\ref{eq_lambdaother}), we pause to justify (\ref{eq_claimtibound}). By (\ref{eq:PKbound}) and (\ref{eq_pk}), we have 
	\begin{align}\label{eq:T1}
		T_{i}(1+b_{i}\tr(GA)-\tr(GA\wt{B}))
		=\tr (GA\wt{B})G_{ii}-\tr(GA)(\wt{B}G)_{ii}+\rO_{\prec}(\Psi). 
	\end{align}
	By (\ref{eq:apxsubor}), we have that 
	\beq\label{eq_gbgindeti}
	(\wt{B}G)_{ii}=\frac{zG_{ii}+1}{a_{i}}, \ 
	G_{ii}= \frac{1}{z}\left( a_i (\wt{B}G)_{ii}-1 \right).
	\eeq
	Based on (\ref{eq_gbgindeti}), on one hand, by (\ref{eq_locallaweqbound}), we find that 
	\begin{equation} \label{eq_wtbgiipointwise}
		(\wt{B}G)_{ii}=\frac{1}{a_{i}-\Omega_{B}}+\rO_{\prec}(N^{-\gamma/4}).
	\end{equation}
	On the other hand, using the fact that $\tr{(GA \wt{B})}=z m_H(z)+1$ and a relation similar to (\ref{eq_multiidentity}), together with (\ref{eq_locallaweqbound}), we conclude that 
	\begin{align}\label{eq_trgawtbpointwise}
		\tr(GA\wt{B})=\int\frac{x}{x-\Omega_{B}}\dd\mu_{A}(x)+\rO_{\prec}(N^{-\gamma/4})=zm_{\mu_{A}\boxtimes\mu_{B}}(z)+1+\rO_{\prec}(N^{-\gamma/4}),
	\end{align}
	and 
	\begin{align}\label{eq_trgapointwise}
		\tr(GA)=\frac{1}{N}\sum_{i}a_{i}G_{ii}&=\frac{\Omega_{B}}{z}\frac{1}{N}\sum_{i}\frac{a_{i}}{a_{i}-\Omega_{B}}+\rO_{\prec}(N^{-\gamma/4})\nonumber \\
		&=\frac{\Omega_{B}}{z}(zm_{\mu_{A}\boxtimes\mu_{B}}(z)+1)+\rO_{\prec}(N^{-\gamma/4}).
	\end{align}
	Combining (\ref{eq:T1}), (\ref{eq_wtbgiipointwise}), (\ref{eq_trgawtbpointwise}) and (\ref{eq_trgapointwise}), using (\ref{eq_locallaweqbound}), we have that
	\beq\label{eq:T2}
	T_{i}\left(1+(zm_{\mu_{A}\boxtimes\mu_{B}}(z)+1)\left(\frac{b_{i}\Omega_{B}}{z}-1\right)\right)=\rO_{\prec}(\Psi+N^{-\gamma/4}).
	\eeq
	Moreover, invoking (\ref{eq_mtrasindenity}) and (\ref{eq_suborsystem}),  we see that 
	\begin{align}\label{eq_t2simplify}
		1+(zm_{\mu_{A}\boxtimes\mu_{B}}(z)+1)\left(\frac{b_{i}\Omega_{B}}{z}-1\right)	
		&=(zm_{\mu_{A}\boxtimes\mu_{B}}(z)+1)\left(\frac{b_{i}\Omega_{B}}{z}-M_{\mu_{A}\boxtimes\mu_{B}}(z)\right) \nonumber	\\
		&=(zm_{\mu_{A}\boxtimes\mu_{B}}(z)+1)\frac{\Omega_{B}}{z}(b_{i}-\Omega_{A}).
	\end{align}
	By (\ref{eq_t2simplify}), a relation similar to (\ref{eq_multiidentity}), and (i) of Proposition \ref{prop:stabN}, we have proved the claim (\ref{eq_claimtibound}) using (\ref{eq:T2}).

	Then we prove (\ref{eq_lambdaother}). First, using (\ref{eq:PKbound}) and the definitions in (\ref{eq_pk}), we find that 
	\beq\label{eq_gammastepone}
	\frac{1}{N}\sum_{i}a_{i}P_{i}
	=\Upsilon\frac{1}{N}\sum_{i}a_{i}(G_{ii}+T_{i})\prec\Psi,
	\eeq
	where we used the fact that $\{a_i\}$ are bounded. 
	By (\ref{eq_gammastepone}), (\ref{eq_trgapointwise}), and (i) of Proposition \ref{prop:stabN}, we have proved that 
	\begin{equation}\label{eq_upsilonbound}
		\Upsilon \prec \Psi. 
	\end{equation}
	Second, using the definition of $P_i$ in (\ref{eq_pk}), the expansion (\ref{eq:Lambda}), and 
	(\ref{eq_upsilonbound}), we have proved that 
	$\Lambda_{d}^c \prec \Psi.$ 
	Third, by (\ref{eq:T1}) and a discussion similar to (\ref{eq_claimtibound}) with the bound $\Lambda_{d}^c \prec \Psi,$ it is easy to see that $\Lambda_T \prec \Psi(z).$ Finally, the proof for $\wt{\Lambda}_d^c$ and $\wt{\Lambda}_T$ follows from an argument similar to (\ref{eq:Lambda}) and (\ref{eq:apxsubor}). This completes the proof of (\ref{eq_lambdaother}).

	It remains to prove \eqref{eq:PKbound}, which is equivalent to the bounds for high moments of $P_{i}$ and $K_{i}$. More specifically, by Markov inequality, it suffices to prove for all positive integer $p\geq 2,$ that the following hold.
	\beq\label{eq:PKmomentbound}
	\expct{\absv{P_{i}}^{2p}}\prec\Psi^{2p}\AND 
	\expct{\absv{K_{i}}^{2p}}\prec\Psi^{2p}.
	\eeq
	The proof of (\ref{eq:PKmomentbound}) makes use of the recursive estimates, that is, Lemma \ref{lem:PKrecmoment} below.  
	Denote 
	\beq\label{eq_definitionpq}
	\frX_{i}^{(p,q)}\deq P_{i}^{p}\ol{P^q_i},\AND \frY_{i}^{(p,q)}\deq K_{i}^{p}\ol{K^q_i}.
	\eeq
	\begin{lem}\label{lem:PKrecmoment}
		For any fixed integer $p\geq 2$ and $i\in\llbra 1,N\rrbra$, we have that
		\begin{align}
			\expct{\frX_{i}^{(p,p)}}
			\leq \expct{\rO_{\prec}(\Psi)\frX_{i}^{(p-1,p)}+\rO_{\prec}(\Psi^{2})\frX_{i}^{(p-2,p)}+\rO_{\prec}(\Psi^{2})\frX_{i}^{(p-1,p-1)}},		
			\label{eq:PKrecmoment} \\
			\expct{\frY_{i}^{(p,p)}}\leq \expct{\rO_{\prec}(\Psi)\frY_{i}^{(p-1,p)}+\rO_{\prec}(\Psi^{2})\frY_{i}^{(p-2,p)}+\rO_{\prec}(\Psi^{2})\frY_{i}^{(p-1,p-1)}}. \label{eq:etakrecursive}
		\end{align}
	\end{lem}
	We next explain how Lemma \ref{lem:PKrecmoment} implies  \eqref{eq:PKmomentbound} and will give the proof of Lemma \ref{lem:PKrecmoment} at the end of this section. 
	Recall that for any positive numbers $\mathsf{u}, \mathsf{v}>0$ 
	we have 
	\begin{equation}\label{eq_young}
		\mathsf{u} \mathsf{v} \leq \frac{\mathsf{u}^{m}}{m}+\frac{\mathsf{v}^{n}}{n}, \ \text{where} \ m,n>1 \ \text{are real numbers with} \  \frac{1}{m}+\frac{1}{n}=1.
	\end{equation} 
	For $k=1,2$, any arbitrary small constant $\epsilon>0$ and any random variable $\frN=\rO_{\prec}(\Psi^{k})$ satisfying $\expct{\absv{\frN}^{q}}\prec \Psi^{qk}$, we have that
	\begin{align}
		\expct{\absv{\frN P_{i}^{2p-k}}} & =\expct{\absv{N^{\epsilon}\frN}\absv{N^{-\frac{\epsilon}{2p-k}}P_{i}}^{2p-k}} \nonumber \\
		&	\leq \frac{kN^{\frac{2p\epsilon}{k}}}{2p}\expct{\absv{\frN}^{\frac{2p}{k}}}+\frac{(2p-k)N^{-\frac{2p\epsilon}{(2p-k)^2}}}{2p}\expct{\absv{P_{i}}^{(2p-k)\frac{2p}{2p-k}}} \nonumber \\
		&	\leq  \frac{kN^{(\frac{2p}{k}+1)\epsilon}}{2p}\Psi^{2p}+\frac{(2p-k)N^{-\frac{2p\epsilon}{(2p-k)^2}}}{2p}\expct{\absv{P_{i}}^{2p}}, \label{eq_arbitraydiscussion}
	\end{align}
	where in the first inequality we used (\ref{eq_young}) with $m=2p/k$ and $n=2p/(2p-k),$ and in the second inequality we used  $\expct{\absv{\frN}^{q}}\prec \Psi^{qk}.$ Together with \eqref{eq:PKrecmoment}, it yields that
	\beqs
	\expct{\absv{P_{i}}^{2p}}\leq \frac{3}{2p}N^{(2p+1)\epsilon}\Psi^{2p}+\frac{3(2p-1)}{2p}N^{-\frac{2p\epsilon}{(2p-1)^2}}\expct{\absv{P_{i}}^{2p}}.
	\eeqs 
	Since $\epsilon>0$ is arbitrarily small, we can conclude the first part of \eqref{eq:PKmomentbound}. The second part can be proved similarly and we omit the details here. This completes the proof of (\ref{eq:PKbound}) and hence the proof of Proposition \ref{prop:entrysubor}. 
\end{proof}

The rest of this subsection is devoted to the proof of Lemma \ref{lem:PKrecmoment}. {This type of estimates have been used in Lemma 5.2 of \cite{BEC} to study the analogous quantities for the additive model and our arguments basically follow the proof therein. In particular, similar to equation (5.34) of \cite{BEC}, integration by parts will be mainly applied to the term $\mr{S}_{i}$ in (\ref{eq_rewritemrs}) to generate several hidden terms which will cancel many other existing larger order terms.} Throughout the proof, we will need some derivative formulas and large deviation estimates as our technical inputs. These can be found in Lemmas \ref{lem_derivative}, \ref{lem:LDE}, \ref{lem:DeltaG} and \ref{lem:recmomerror}
\begin{proof}[\bf Proof of Lemma \ref{lem:PKrecmoment}] We start with the proof of (\ref{eq:PKrecmoment}). Recall (\ref{eq_shorhandnotation}). Since $h_{ii}\bse_{i}\adj \wt{B}^{\angi}G\bse_{i}=b_{i}h_{ii}G_{ii}$, we can rewrite \eqref{eq:BGii-S} as
	\beq\label{eq:BGii}
	(\wt{B}G)_{ii}=-S_{i}+\ell_{i}^{2}(b_{i}h_{ii}+\bsh_{i}\adj\wt{B}^{\angi}\bsh_{i})(G_{ii}+T_{i})=-\mr{S}_{i}+G_{ii}+T_{i}+\mathsf{e}_{i1},
	\eeq
	where we denoted
	\beq\label{eq_epsilon1}
	\mathsf{e}_{i1}
	\deq (\ell_{i}^{2}-1)b_{i}h_{ii}G_{ii}+(\ell_{i}^{2}\bsh_{i}\adj\wt{B}^{\angi}\bsh_{i}-1)(G_{ii}+T_{i})+\ell_{i}^{2}b_{i}h_{ii}T_{i}.
	\eeq
	Recall $\wt{\bsg} \sim \mathcal{N}_{\mathbb{C}}(0, N^{-1}I_N)$. By Lemma \ref{lem:LDE}, we see that 
	\begin{equation}\label{eq_hiicontrol}
		h_{ii}=\norm{\wt{\bsg}_{i}}^{-1}\absv{\bse_{i}\adj \wt{\bsg}_{i}}\prec N^{-1/2}.
	\end{equation}
	Consequently, using the definitions in (\ref{eq_prd2}), we obtain that 
	\beq\label{eq_licontrol}
	\ell_{i}^{2}=\frac{2}{\norm{\bse_{i}+\bsh_{i}}^{2}}=\frac{1}{1+\bse_{i}\adj \bsh_{i}}=1+\rO_{\prec}(N^{-1/2}).
	\eeq
	Moreover, by (\ref{eq_prd}), (\ref{eq_prd1}), (\ref{eq_defnvb}), and Lemma \ref{lem:LDE}, we have 
	\begin{align}\label{eq_controlhibhi}
		\bsh_{i}\adj\wt{B}^{\angi}\bsh_{i} =\bsh_{i}\adj R_{i}\wt{B}R_{i}\bsh_{i}=\bse_{i}\adj\wt{B}\bse_{i}=\frac{1}{\norm{\wt{\bsg}_{i}}^{2}}\wt{\bsg}_{i}\adj B\wt{\bsg}_{i}=1+\rO_{\prec}(N^{-1/2}),  
	\end{align}
	where we recall that $B$ is normalized such that $\operatorname{tr} B=1.$ Using the definition (\ref{eq_epsilon1}), by (\ref{eq_hiicontrol}), (\ref{eq_licontrol}), (\ref{eq_controlhibhi}), and (\ref{eq_locallaweqbound}), we conclude that 
	\begin{equation}\label{eq_boundepsilon1}
		\absv{\mathsf{e}_{i1}}\prec N^{-1/2}.
	\end{equation}
	Therefore, by (\ref{eq_pk}), \eqref{eq:BGii}, and (\ref{eq_boundepsilon1}), we have shown that
	\begin{align}\label{eq:recmomP}
		\expct{\frX_{i}^{(p,p)}}
		&=\expct{\left(G_{ii}\tr(A\wt{B}G)+\tr(GA)(\mr{S}_{i})+(G_{ii}+T_{i})(\Upsilon-\tr(GA))\right)\frX_{i}^{(p-1,p)}}\nonumber \\
		&+\expct{ \mathsf{e}_{i1} \tr{(GA)}\frX_{i}^{(p-1,p)}}. 
	\end{align}

	Next, we control all the terms on the RHS of (\ref{eq:recmomP}). We mainly focus on the term involving $\tr(GA)(\mr{S}_{i})$. As we will see later, by exploring the hidden terms using integration by parts, the term involving $\tr(GA)(\mr{S}_{i})$ will generate several terms which would cancel the rest of the terms on the RHS of (\ref{eq:recmomP}) algebraically. Note that
	\beq\label{eq_rewritemrs}
	\mr{S}_{i}=\mr{\bsh}_{i}\adj \wt{B}^{\angi}G\bse_{i} =\sum_{k}\mr{\bsh}_{i}\adj \bse_{k}\bse_{k}\adj \wt{B}^{\angi}G\bse_{i}
	=\sum_{k}^{(i)}\ol{g}_{ik}\frac{1}{\norm{\bsg_{i}}}\bse_{k}\adj \wt{B}^{\angi}G\bse_{i},
	\eeq
	where we use the shorthand notation $\sum_{k}^{(i)}$ to represent the sum over $\llbra 1,N\rrbra \backslash \{i\}.$
	Our calculation relies on the following  integration by parts formula for $g \sim \mathcal{N}(0,\sigma^2)$ (see equation (5.33) of \cite{BEC}) 
	\begin{equation}\label{eq_formulaintergrationbyparts}
		\int_{\mathbb{C}} \bar{g} f(g, \bar{g}) e^{-\frac{|g|^2}{\sigma^2}} \mathrm{d}^2 g= \sigma^2 \int_{\mathbb{C}} \partial_g f(g, \bar{g}) e^{-\frac{|g|^2}{\sigma^2}}  \mathrm{d}^2 g,
	\end{equation}
	where $f: \mathbb{C}^2 \rightarrow \mathbb{C}$ is a differentiable function. By (\ref{eq_rewritemrs}) and (\ref{eq_formulaintergrationbyparts}), we have
	\begin{align}
		\expct{\mr{S}_{i}\tr (GA)\frX^{(p-1,p)}}
		= \sum_{k}^{(i)}\expct{\ol{g}_{ik}\frac{1}{\norm{\bsg_{i}}}\bse_{k}\adj \wt{B}^{\angi}G\bse_{i}\tr(GA)\frX^{(p-1,p)}} \quad \quad \quad \quad \quad \quad \quad  \quad \quad \quad \quad \quad \quad  \nonumber	 \\
		=\frac{1}{N}\sum_{k}^{(i)}\expct{\frac{\partial\norm{\bsg_{i}}^{-1}}{\partial g_{ik}}\bse_{k}\adj \wt{B}^{\angi}G\bse_{i}\tr(GA)\frX^{(p-1,p)}}
		+\frac{1}{N}\sum_{k}^{(i)}\expct{\frac{1}{\norm{\bsg_{i}}}\frac{\partial (\bse_{k}\adj\wt{B}^{\angi}G\bse_{i})}{\partial g_{ik}}\tr (GA)\frX^{(p-1,p)}} \nonumber	\\
		+ \frac{1}{N}\sum_{k}^{(i)}\expct{\frac{\bse_{k}\adj\wt{B}^{\angi}G\bse_{i}}{\norm{\bsg_{i}}}\frac{\partial\tr(GA)}{\partial g_{ik}}\frX^{(p-1,p)}}	
		+\frac{p-1}{N}\expct{\frac{1}{\norm{\bsg_{i}}}\bse_{k}\adj\wt{B}^{\angi}G\bse_{i}\tr(GA)\frac{\partial P_{i}}{\partial g_{ik}}\frX^{(p-2,p)}} \quad \quad \quad \nonumber	\\
		+ \hfill \frac{p}{N}\expct{\frac{1}{\norm{\bsg_{i}}}  \bse_{k}\wt{B}^{\angi}G\bse_{i}\tr(GA)\frac{\partial \ol{P}_{i}}{g_{ik}}\frX^{(p-1,p-1)}}, \quad \quad \quad  \label{eq_cumulantfirst}
	\end{align}
	where we recall that $\frX^{(p-1,p)}=P_{i}^{p-1}\ol{P}_{i}^{p}$ as defined in (\ref{eq_definitionpq}). We point out that the second term on the RHS of (\ref{eq_cumulantfirst}) will provide some hidden terms for cancellation. Further, since $\wt{B}^{\angi}$ is independent of $\bm{v}_i,$ we have that
	\begin{equation*}
		\frac{\partial(\bm{e}_k^* \wt{B}^{\angi} G \bm{e}_i)}{\partial g_{ik}}=\bm{e}_k^* \wt{B}^{\angi} \frac{\partial G}{\partial g_{ik}} \bm{e}_i.
	\end{equation*}  
	
	We start with  the analysis of the Household reflection defined in (\ref{eq_prd}). Recall $\bm{r}_i=\ell_i(\bm{e}_i+\bm{h}_i)$ and $\ell_i$ defined in (\ref{eq_prd2}). By (\ref{eq_householdderivative}), (\ref{eq_hiderivative}), and (\ref{eq_partialli2}), we have that 
	\begin{multline*}
		\frac{\partial R_{i}}{\partial g_{ik}}=-\ell_{i}^{4}\norm{\bsg_{i}}^{-1}\ol{h}_{ik}h_{ii}(\bse_{i}+\bsh_{i})(\bse_{i}+\bsh_{i})\adj	\\
		-\ell_{i}^{2}\norm{\bsg_{i}}^{-1}\left(\bse_{k}\bse_{i}\adj-\ol{h}_{ik}(\bsh_{i}\bse_{i}\adj+\bse_{i}\bsh_{i}\adj)
		+\bse_{k}\bsh_{i}\adj-\ol{h}_{ik}\bsh_{i}\bsh_{i}\adj -\ol{h}_{ik}\bsh_{i}\bsh_{i}\adj\right).
	\end{multline*}
	We can further rewrite the above equation as 
	\beq\label{eq_househouldfinalexpression}
	\frac{\partial R_{i}}{\partial g_{ik}}=-\frac{\ell_{i}^{2}}{\norm{\bsg_{i}}}\bse_{k}(\bse_{i}\adj +\bsh_{i}\adj)+\Delta_{R}(i,k),
	\eeq
	where we defined
	\begin{align}\label{eq_deltarg}
		\Delta_{R}(i,k)
		\deq & -\frac{\ell_{i}^{4}}{\norm{\bsg_{i}}^{2}}\bar{h}_{ik}h_{ii}(\bse_{i}+\bsh_{i})(\bse_{i}+\bsh_{i})\adj \nonumber \\
		&+\ell_{i}^{2}\norm{\bsg_{i}}^{-1}\ol{h}_{ik}(\bsh_{i}\bse_{i}\adj+\bse_{i}\bsh_{i}\adj +2\bsh_{i}\bsh_{i}\adj).
	\end{align}
	By (\ref{eq_househouldfinalexpression}) and the fact that $\wt{B}^{\angi}$ is independent of $g_{ik}$, we obtain that
	\beq\label{eq:Gder1}
	\frac{\partial G}{\partial g_{ik}}=\frac{\ell_{i}^{2}}{\norm{\bsg_{i}}}GA\left(\bse_{k}(\bse_{i}\adj+\bsh_{i}\adj)\wt{B}^{\angi}R_{i}+R_{i}\wt{B}^{\angi}\bse_{k}(\bse_{i}\adj+\bsh_{i}\adj)\right)G+\Delta_{G}(i,k),
	\eeq
	where
	\beq\label{eq_defndeltag}
	\Delta_{G}(i,k)\deq -GA\left(\Delta_{R}(i,k)\wt{B}^{\angi}R_{i}+R_{i}\wt{B}^{\angi}\Delta_{R}(i,k)\right)G.
	\eeq

	We see from  (\ref{eq:Gder1}) and \eqref{eq:DeltaG1} that 
	\begin{align}\label{eq:1st_der_expa}
		\frac{1}{N}\sum_{k}^{(i)}\bse_{k}\adj\wt{B}^{\angi}\frac{\partial G}{\partial g_{ik}}\bse_{i}
		=\frac{\ell_{i}^{2}}{\norm{\bsg_{i}}}\frac{1}{N}\sum_{k}^{(i)}\bse_{k}\adj \wt{B}^{\angi}GA\left(\bse_{k}(\bse_{i}\adj +\bsh_{i}\adj)\wt{B}^{\angi}R_{i}+R_{i}\wt{B}^{\angi}\bse_{k}(\bse_{i}\adj+\bsh_{i}\adj)\right)G\bse_{i} +\rO_{\prec}(\Pi_{i}^{2})	\nonumber \\
		=\frac{\ell_{i}^{2}}{\norm{\bsg_{i}}}\frac{1}{N}\sum_{k}^{(i)} \left[ a_{k}\bse_{k}\adj\wt{B}^{\angi}G\bse_{k}(-\bsh_{i}\adj\wt{B}-\bse_{i}\adj\wt{B})G\bse_{i}+\bse_{k}\adj \wt{B}^{\angi}GAR_{i}\wt{B}^{\angi}\bse_{k}(G_{ii}+\bsh_{i}\adj G\bse_{i})\right]+\rO_{\prec}(\Pi_{i}^{2})  \nonumber	\\
		=\frac{\ell_{i}^{2}}{\norm{\bsg_{i}}}\frac{1}{N}\sum_{k}^{(i)} a_{k}(\wt{B}^{\angi}G)_{kk}(-b_{i}T_{i}-(\wt{B}G)_{ii}) \quad \quad \quad \quad \quad \quad \quad \quad \quad \quad \quad \quad \quad \quad  \quad \quad \quad \quad \quad \quad \quad \quad \nonumber \\
		+\frac{\ell_{i}^{2}}{\norm{\bsg_{i}}}\frac{1}{N}\sum_{k}^{(i)} (\bse_{k}\adj \wt{B}^{\angi}GAR_{i}\wt{B}^{\angi}\bse_{k})(G_{ii}+T_{i})+\rO_{\prec}(\Pi_{i}^{2}),    \quad \quad \quad \quad \quad \quad \quad \quad    
	\end{align}
	where in the second equality we used (\ref{eq_hbwtr}) and in the third equality we used (\ref{eq_hbwtr}) and (\ref{eq_shorhandnotation}). Moreover, by (\ref{eq_locallaweqbound}) and the assumption that $\{a_i\}$ and $\{b_i\}$ are bounded, we readily see that
	\begin{equation}\label{eq_firstapproximatetrace}
		\tr(A\wt{B}^{\angi}G)-\frac{1}{N}\sum_{k}^{(i)}a_{k}(\wt{B}^{\angi}G)_{kk} 
		=\frac{1}{N}a_{i}b_{i}G_{ii}\prec\frac{1}{N},
	\end{equation}
	and
	\begin{align}\label{eq_firstapproximatetrace2}
		\tr(\wt{B}^{\angi}GAR_{i}\wt{B}^{\angi})-\frac{1}{N}\sum_{k}^{(i)}\bse_{k}\adj\wt{B}^{\angi}GAR_{i}\wt{B}^{\angi}\bse_{k}
		=-\frac{b_{i}}{N}\bse_{i}\adj GA\wt{B}\bsh_{i}  \prec\frac{1}{N}.
	\end{align}
	We claim that  we can replace $\wt{B}^{\angi}$ by $\wt{B}$ in (\ref{eq_firstapproximatetrace}) and (\ref{eq_firstapproximatetrace2}) without changing the error bound in (\ref{eq:1st_der_expa}). In fact, by the definition of $\wt{B}^{\angi}$, we have that 
	\begin{align}\label{eq_removeindexcontrol}
		\tr(A\wt{B}G)-\tr(A\wt{B}^{\angi}G)	&
		=\tr(A\wt{B}G)-\tr(AR_{i}\wt{B}R_{i}G) \nonumber	\\
		&=\tr(A\bsr_{i}\bsr_{i}\adj\wt{B}G)+\tr(A\wt{B}\bsr_{i}\bsr_{i}\adj G)-\tr(A\bsr_{i}\bsr_{i}\adj \wt{B}\bsr_{i}\bsr_{i}\adj G) \nonumber	\\
		&=\frac{1}{N}\bsr_{i}\adj\wt{B}GA\bsr_{i}+\frac{1}{N}\bsr_{i}\adj GA\wt{B}\bsr_{i}-\frac{1}{N}\bsr_{i}\adj \wt{B}\bsr_{i}\bsr_{i}\adj GA\bsr_{i}.
	\end{align}

	Recall that $\bsr_i=\ell_i(\bm{e}_i+\bm{h}_i).$ Then we have 
	\begin{align}\label{eq_casestudyerrorone}
		\left|\frac{1}{N}\bsr_{i}\adj\wt{B}GA\bsr_{i}\right|
		& \lesssim \frac{1}{N}\left(\norm{\sqrt{A}G\adj \wt{B}\bse_{i}} \norm{A^{1/2}}+\norm{G\adj \wt{B}\bsh_{i}}\right) \nonumber \\
		& \lesssim \frac{1}{N}\left(\bse_{i}\wt{B} G AG\adj \wt{B}\bse_{i}+b_{i}^{2}\bsh_{i}\adj G G\adj\bsh_{i}\right)^{1/2},
	\end{align}
	where in the second inequality we used the fact that $\| A \|$ is bounded. Moreover, using (\ref{eq_gggconnetction}), (\ref{eq_connectiongreenfunction}) and (\ref{defn_eq_matrices}), it is easy to see that
	\beq\label{eq:a_i}
	\bse_{i}\adj \wt{B}GAG\adj \wt{B}\bse_{i}
	=\frac{\bse_{i}\adj\sqrt{A}\wt{B}\sqrt{A}\wt{G}\wt{G}\adj \sqrt{A}\wt{B}\sqrt{A}\bse_{i}}	{a_{i}}
	=\frac{\bse_{i}\adj \wt{H}\wt{G}\wt{G}\adj \wt{H}\bse_{i}}{a_{i}}
	\leq \norm{\wt{H}}^{2}\frac{\im \wt{G}_{ii}}{a_{i}\eta},
	\eeq
	where in the last inequality we used the fact that $\wt{G}$ is Hermitian and the Ward identity 
	\begin{equation*}
		\sum_{j=1}^N |\wt{G}_{ij}|^2=(\wt{G} \wt{G}^*)_{ii} =\frac{\im \wt{G}_{ii}}{\eta}.
	\end{equation*}
	Similarly, we can show that 
	\beq\label{eq_part2ai}
	b_{i}^{2}\bsh_{i}\adj GG\adj \bsh_{i}=b_{i}^{2}\bse_{i}\caG \caG\adj \bse_{i} =b_{i}\bse_{i}\wt{\caG}B\wt{\caG}\adj \bse_{i}\leq b_{i}\norm{B}\frac{\im \wt{\caG}_{ii}}\eta.
	\eeq
	Since $A, B, \wt{H}$ are bounded, by (\ref{eq_casestudyerrorone}), (\ref{eq:a_i}), and (\ref{eq_part2ai}), we see that 
	\begin{equation*}
		\left|\frac{1}{N} \bsr_i^* \wt{B} GA \bsr_i \right| \lesssim \frac{1}{N}\left( \frac{\im \wt{G}_{ii}}{\eta}+\frac{\im \wt{\caG}_{ii}}{\eta} \right)^{1/2}.
	\end{equation*}
	By an analogous discussion, we can control the other two terms of the RHS of (\ref{eq_removeindexcontrol}) as
	\begin{equation*}
		\left|\frac{1}{N} \bsr_i^* GA \wt{B} \bsr_i \right| \lesssim \frac{1}{N}\left( \frac{\im \wt{G}_{ii}}{\eta}+\frac{\im \wt{\caG}_{ii}}{\eta} \right)^{1/2}, \  \left|\frac{1}{N} \bsr_i^* \wt{B} \bsr_i \bsr_i^* GA \bsr_i \right| \lesssim \frac{1}{N}\left( \frac{\im \wt{G}_{ii}}{\eta}+\frac{\im \wt{\caG}_{ii}}{\eta} \right)^{1/2}.
	\end{equation*}

	Furthermore, from the spectral decomposition of $\wt{H}$ and $\wt{\caH}$, it is clear that that $\im \wt{G}_{ii}/\eta\geq c$ and $\im \wt{\caG}_{ii}/\eta\geq c$ for some fixed constant $c>0.$ This shows that for some constant $C>0,$ 
	\beqs
	\frac{1}{N}\left(\frac{\im\wt{G}_{ii}+\im\wt{\caG}_{ii}}{\eta}\right)^{1/2}\leq \frac{C}{N}\frac{\im\wt{G}_{ii}+\im\wt{\caG}_{ii}}{\eta}=C\Pi_{i}^{2}.
	\eeqs
	Together with (\ref{eq_removeindexcontrol}), we arrive at 
	\begin{equation}\label{eq_firstindexreducecomplete}
		\tr(A\wt{B}^{\angi}G)=\tr(A\wt{B}G)+\rO_{\prec}(\Pi_{i}^2). 
	\end{equation}
	By a discussion similar to (\ref{eq_firstindexreducecomplete}), we can get
	\begin{equation}\label{eq_secondindexreducecomplete}
		\tr(\wt{B}^{\angi}GAR_{i}\wt{B}^{\angi})=\tr{(\wt{B}GA \wt{B})}+\rO_{\prec}(\Pi_{i}^2). 
	\end{equation}
	
	Therefore, by (\ref{eq:1st_der_expa}), (\ref{eq_firstindexreducecomplete}), and (\ref{eq_secondindexreducecomplete}), we conclude that
	\begin{align}\label{eq:Sexpand}
		\frac{1}{N}\sum_{k}^{(i)}\bse_{k}\adj\wt{B}^{\angi} & \frac{\partial G}{\partial g_{ik}}\bse_{i} \nonumber \\
		& =\frac{\ell_{i}^{2}}{\norm{\bsg_{i}}}\left(\tr(A\wt{B}G)(-b_{i}T_{i}-(\wt{B}G)_{ii})+\tr(\wt{B}GA\wt{B})(G_{ii}+T_{i})\right)+\rO_{\prec}(\Pi_{i}^{2}).
	\end{align}

	Note that compared to the expansion (\ref{eq:recmomP}), the coefficient in front of $\tr{(A \wt{B}G)}$ is still different. We need further explore the hidden relation. By a discussion similar to (\ref{eq:Sexpand}), we have that 
	\begin{align}\label{eq:Texpand}
		\frac{1}{N}\sum_{k}^{(i)}\bse_{k}\adj & \frac{\partial G}{\partial g_{ik}}\bse_{i} \nonumber \\
		&=\frac{\ell_{i}^{2}}{\norm{\bsg_{i}}}\left(\tr(GA)(-b_{i}T_{i}-(\wt{B}G)_{ii})+\tr(GA\wt{B})(G_{ii}+T_{i})\right)+\rO_{\prec}(\Pi_{i}^{2}).
	\end{align}
	In light of (\ref{eq:Sexpand}), (\ref{eq:Texpand})
	and (\ref{eq:recmomP}), it suffices to control 
	$$\tr(GA)\frac{1}{N}\sum_{k}^{(i)}\bse_{k}\adj\wt{B}^{\angi}\frac{\partial G}{\partial g_{ik}}\bse_{i}-\tr(A\wt{B}G)\frac{1}{N}\sum_{k}^{(i)}\bse_{k}\adj \frac{\partial G}{\partial g_{ik}}\bse_{i}.$$
	Combining (\ref{eq:Sexpand}) and (\ref{eq:Texpand}), we have that
	\begin{align}\label{eq_combineminus}
		\tr(GA)& \frac{1}{N}\sum_{k}^{(i)}\bse_{k}\adj\wt{B}^{\angi}\frac{\partial G}{\partial g_{ik}}\bse_{i}-\tr(A\wt{B}G)\frac{1}{N}\sum_{k}^{(i)}\bse_{k}\adj \frac{\partial G}{\partial g_{ik}}\bse_{i}	\nonumber \\
		&=\frac{\ell_{i}^{2}}{\norm{\bsg_{i}}}(G_{ii}+T_{i})(\tr(GA)\tr(\wt{B}GA\wt{B})-\tr(GA\wt{B})\tr(GA\wt{B}))+\rO_{\prec}(\Pi_{i}^{2})	\nonumber \\
		&=\frac{\ell_{i}^{2}}{\norm{\bsg_{i}}}(G_{ii}+T_{i})(-\Upsilon-\tr(A\wt{B}G)+\tr(GA))+\rO_{\prec}(\Pi_{i}^{2}),
	\end{align}
	where in the second equality we employed the definition of $\Upsilon$ in (\ref{eq_defnupsilon}). Denote
	\begin{align}\label{eq_defnmathsfe2}
		\mathsf{e}_{i2}\deq \left(\frac{\ell_{i}^{2}}{\norm{\bsg_{i}}}-\norm{\bsg_{i}}\right)(&-G_{ii}\tr(A\wt{B}G)-(G_{ii}+T_{i})(\Upsilon-\tr(GA))) \nonumber \\
		&+\tr(A\wt{B}G)\left(\norm{\bsg_{i}}\mr{T}_{i}-\frac{\ell_{i}^{2}}{\norm{\bsg_{i}}}T_{i}\right).
	\end{align} 
	By a discussion similar to (\ref{eq_boundepsilon1}), we can conclude that
	\begin{equation}\label{eq_controlepsilon2}
		|\mathsf{e}_{i2}| \prec N^{-1/2}. 
	\end{equation}
	Moreover, by a simple algebraic calculation using (\ref{eq_combineminus}) and (\ref{eq_defnmathsfe2}), we find that 
	\begin{align}\label{eq_finalexpansion}
		\tr(GA)\frac{1}{N}\sum_{k}^{(i)}\bse_{k}\adj\wt{B}^{\angi}\frac{\partial G}{\partial g_{ik}}\bse_{i}=&\norm{\bsg_{i}}(-G_{ii}\tr(A\wt{B}G)-(G_{ii}+T_{i})(\Upsilon-\tr(GA))) \nonumber \\
		+&\tr(A\wt{B}G)\left(\frac{1}{N}\sum_{k}^{(i)}\bse_{k}\adj\frac{\partial G}{\partial g_{ik}}\bse_{i}-\norm{\bsg_{i}}\mr{T}_{i}\right)+\mathsf{e}_{i2}+\rO_{\prec}(\Pi_{i}^{2}).
	\end{align}
	With (\ref{eq_finalexpansion}) and (\ref{eq_cumulantfirst}), we can now come back to discussing \eqref{eq:recmomP}. More specifically, inserting (\ref{eq_finalexpansion}) into (\ref{eq_cumulantfirst}) and then (\ref{eq:recmomP}), we have that  
	\begin{align}\label{eq:recmomP1}
		&\expct{\frX_{i}^{(p,p)}}
		=\expct{\left(\frac{1}{N}\sum_{k}^{(i)}\frac{1}{\norm{\bsg_{i}}}\bse_{k}\adj\frac{\partial G}{\partial g_{ik}}\bse_{i}-\mr{T}_{i}\right)\tr(A\wt{B}G)\frX_{i}^{(p-1,p)}} \quad \quad \quad \quad \quad \quad \quad \quad \quad \quad \quad \quad \quad \quad \quad \quad \quad \quad \quad \quad \nonumber \\
		+&\frac{1}{N}\sum_{k}^{(i)}\expct{\frac{\partial\norm{\bsg_{i}}^{-1}}{\partial g_{ik}}\bse_{k}\adj\wt{B}^{\angi}G\bse_{i}\tr(GA)\frX_{i}^{(p-1,p)}} \nonumber +\frac{1}{N}\sum_{k}^{(i)}\expct{\frac{\bse_{k}\adj\wt{B}^{\angi}G\bse_{i}}{\norm{\bsg_{i}}}\tr\left(\frac{\partial G}{\partial g_{ik}}A\right)\frX_{i}^{(p-1,p)}} \quad \quad \quad \quad \quad \quad \quad \quad \\
		+&\frac{p-1}{N}\sum_{k}^{(i)}\expct{\frac{1}{\norm{\bsg_{i}}}\bse_{k}\adj\wt{B}^{\angi}G\bse_{i}\tr(GA)\frac{\partial P_{i}}{\partial g_{ik}}\frX_{i}^{(p-2,p)}} +\frac{p}{N}\sum_{k}^{(i)}\expct{\frac{1}{\norm{\bsg_{i}}}\bse_{k}\adj\wt{B}^{\angi}G\bse_{i}\tr(GA)\frac{\partial \ol{P}_{i}}{\partial g_{ik}}\frX_{i}^{(p-1,p-1)}} \quad \quad \quad \quad  \nonumber \\
		+&\expct{\left(\mathsf{e}_{i1}\tr(GA)+\frac{1}{\norm{\bsg_{i}}}\mathsf{e}_{i2}+\rO_{\prec}(\Pi_{i}^{2})\right)\frX_{i}^{(p-1,p)}}.
	\end{align}
	We do one more expansion for the first term of the above equation. Recall the definitions in (\ref{eq_shorhandnotation}). Applying the technique of integration by parts, that is, (\ref{eq_formulaintergrationbyparts}), we get that 
	\begin{align}\label{eq:recmomP2}
		\expct{\left(\mr{T}_{i}-\frac{1}{N}\sum_{k}^{(i)}\frac{1}{\norm{\bsg_{i}}}\bse_{k}\adj\frac{\partial G}{\partial g_{ik}}\bse_{i}\right)\tr(A\wt{B}G)\frX_{i}^{(p-1,p)}} \quad \quad \quad \quad \quad \quad  \quad \quad \quad \quad 
		\\
		=\frac{1}{N}\sum_{k}^{(i)}\expct{\frac{\partial \norm{\bsg_{i}}^{-1}}{\partial g_{ik}}G_{ki}\tr(A\wt{B}G)\frX_{i}^{(p-1,p)}}	
		+\frac{1}{N}\sum_{k}^{(i)}\expct{\frac{1}{\norm{\bsg_{i}}}G_{ki}\tr(A\wt{B}\frac{\partial G}{\partial g_{ik}})\frX_{i}^{(p-1,p)}} \quad \quad \quad   \nonumber \\
		+\frac{p-1}{N}\sum_{k}^{(i)}\expct{\frac{1}{\norm{\bsg_{i}}}G_{ki}\tr(A\wt{B}G)\frac{\partial P_{i}}{\partial g_{ik}}\frX_{i}^{(p-2,p)}}	
		+\frac{p}{N}\sum_{k}^{(i)}\expct{\frac{1}{\norm{\bsg_{i}}}G_{ki}\tr(A\wt{B}G)\frac{\partial \ol{P}_{i}}{\partial g_{ik}}\frX_{i}^{(p-1,p-1)}}. \nonumber
	\end{align}
	Combining (\ref{eq:recmomP2}) and (\ref{eq:recmomP1}), we can rewrite 
	\begin{equation}\label{eq_pkfinalrepresentation}
		\expct{\frX_{i}^{(p,p)}}=\expct{ \mathfrak{C}_1\frX_{i}^{(p-1,p)}}
		+\expct{\mathfrak{C}_2\frX_{i}^{(p-2,p)}}
		+\expct{\mathfrak{C}_3\frX_{i}^{(p-1,p-1)}},
	\end{equation}
	where the coefficients $\mathfrak{C}_k,k=1,2,3$ are defined as 
	\begin{align}\label{eq_defnmathfrackc1}
		\mathfrak{C}_1:=\frac{1}{N} \sum_{k}^{(i)} \Big(\frac{\partial \norm{\bsg_{i}}^{-1}}{\partial g_{ik}}G_{ki}\tr(A\wt{B}G)+\frac{1}{\norm{\bsg_{i}}}G_{ki}\tr(A\wt{B}\frac{\partial G}{\partial g_{ik}})+\frac{\partial\norm{\bsg_{i}}^{-1}}{\partial g_{ik}}\bse_{k}\adj\wt{B}^{\angi}G\bse_{i}\tr(GA) \nonumber  \\
		+\frac{\bse_{k}\adj\wt{B}^{\angi}G\bse_{i}}{\norm{\bsg_{i}}}\tr\left(\frac{\partial G}{\partial g_{ik}}A\right)+\left(\mathsf{e}_{i1}\tr(GA)+\frac{1}{\norm{\bsg_{i}}}\mathsf{e}_{i2}+\rO_{\prec}(\Pi_{i}^{2})\right)
		\Big),  \\
		\mathfrak{C}_2:=\frac{p-1}{N} \sum_{k}^{(i)} \Big(\frac{1}{\norm{\bsg_{i}}}\bse_{k}\adj\wt{B}^{\angi}G\bse_{i}\tr(GA)\frac{\partial P_{i}}{\partial g_{ik}}+\frac{1}{\norm{\bsg_{i}}}G_{ki}\tr(A\wt{B}G)\frac{\partial P_{i}}{\partial g_{ik}} \Big),  \label{eq_defnmathfrackc2} \\
		\mathfrak{C}_3:=\frac{p}{N} \sum_{k}^{(i)} \Big(\frac{1}{\norm{\bsg_{i}}}\bse_{k}\adj\wt{B}^{\angi}G\bse_{i}\tr(GA)\frac{\partial \ol{P}_{i}}{\partial g_{ik}}+ \frac{1}{\norm{\bsg_{i}}}G_{ki}\tr(A\wt{B}G)\frac{\partial \ol{P}_{i}}{\partial g_{ik}}\Big). \label{eq_defnmathfrackc3}
	\end{align}
	To conclude the proof of (\ref{eq:PKrecmoment}), it suffices to control the coefficients $\mathfrak{C}_k, k=1,2,3.$ For $\mathfrak{C}_1,$ by Lemma \ref{lem:recmomerror}, we find that
	\begin{equation}\label{eq:controlc1}
		\mathfrak{C}_1 \prec N^{-1/2}+\Pi_i^2. 
	\end{equation}
	For $\mathfrak{C}_2,$ by (\ref{eq_partialpi}) and Lemma \ref{lem:recmomerror}, we find that 
	\begin{equation}\label{eq:controlc2}
		\mathfrak{C}_2 \prec \Pi_i^2.
	\end{equation} 
	Similarly, we can show that 
	\begin{equation}\label{eq:controlc3}
		\mathfrak{C}_3 \prec \Pi_i^2. 
	\end{equation}
	Using (\ref{eq_controlparameter}), we complete the proof of (\ref{eq:PKrecmoment}) using (\ref{eq:controlc1}), (\ref{eq:controlc2}), (\ref{eq:controlc3}), and (\ref{eq_pkfinalrepresentation}). 
	
	Finally, due to similarity, we only briefly discuss the proof of (\ref{eq:etakrecursive}). Using the definition of $K_i$ in (\ref{eq_pk}) and the fact that $T_{i}-\mr{T}_{i}=h_{ii}G_{ii}\prec N^{-1/2},$ we find that  
	\begin{align}
		\expct{\frY_{i}^{(p,p)}}=\expct{\left(\mr{T}_{i}+\tr(GA)(b_{i}T_{i}+(\wt{B}G)_{ii})-\tr(GA\wt{B}G)(G_{ii}+T_{i})\right)\frY_{i}^{(p-1,p)}}+\expct{\rO_{\prec}(N^{-1/2})\frY_{i}^{(p-1,p)}} \nonumber \\
		=\sum_{k}^{(i)}\expct{\frac{\ol{g}_{ik}}{\norm{\bsg_{i}}}\bse_{k}\adj G\bse_{i}\frY_{i}^{(p-1,p)}}+\expct{\left(\tr(GA)(b_{i}T_{i}+(\wt{B}G)_{ii})-\tr(GA\wt{B}G)(G_{ii}+T_{i})\right)\frY_{i}^{(p-1,p)}} \nonumber \\
		+\expct{\rO_{\prec}(N^{-1/2})\frY_{i}^{(p-1,p)}}, \quad \quad \quad \quad \quad \quad \quad \quad \quad \quad  \quad \quad \quad \quad \quad \ \label{eq_closeendetaexpansion}
	\end{align}
	where in the second equality we used the definition  in (\ref{eq_shorhandnotation}). Applying (\ref{eq_formulaintergrationbyparts}) to the first term of the RHS of the above equation, we obtain
	\begin{align}
		\sum_{k}^{(i)}\expct{\frac{\ol{g}_{ik}}{\norm{\bsg_{i}}}\bse_{k}\adj G\bse_{i}\frY_{i}^{(p-1,p)}}
		&=\frac{1}{N}\sum_{k}^{(i)}\expct{\frac{\partial \norm{\bsg_{i}}^{-1}}{\partial g_{ik}}\bse_{k}\adj G\bse_{i}\frY_{i}^{(p-1,p)}}
		+\frac{1}{N}\sum_{k}^{(i)}\expct{\frac{1}{\norm{\bsg_{i}}}\bse_{k}\adj\frac{\partial G}{\partial g_{ik}}\bse_{i}\frY_{i}^{(p-1,p)}} \nonumber \\
		& +\frac{p-1}{N}\sum_{k}^{(i)}\expct{\frac{1}{\norm{\bsg_{i}}}\bse_{k}\adj G\bse_{i}\frac{\partial K_{i}}{\partial g_{ik}}\frY_{i}^{(p-2,p)}} \nonumber	\\
		&+\frac{p}{N}\sum_{k}^{(i)}\expct{\frac{1}{\norm{\bsg_{i}}}\bse_{k}\adj G\bse_{i}\frac{\partial \ol{K}_{i}}{\partial g_{ik}}\frY_{i}^{(p-1,p-1)}}. \label{eq_firstterm}
	\end{align}
	Inserting \eqref{eq:Texpand} into (\ref{eq_firstterm}) and then (\ref{eq_closeendetaexpansion}), by a discussion similar to the cancellation in (\ref{eq:recmomP1}) and error controls in (\ref{eq_boundepsilon1}) and (\ref{eq_controlepsilon2}), we conclude that   
	\begin{align*}
		\expct{\frY_{i}^{(p,p)}}
		&=\frac{1}{N}\sum_{k}^{(i)}\expct{\frac{\partial\norm{\bsg_{i}}^{-1}}{\partial g_{ik}}\bse_{k}\adj G\bse_{i}\frY_{i}^{(p-1,p)}}+\frac{p-1}{N}\sum_{k}^{(i)}\expct{\frac{1}{\norm{\bsg_{i}}}\bse_{k}\adj G\bse_{i}\frac{\partial K_{i}}{\partial g_{ik}}\frY_{i}^{(p-2,p)}} \\
		&+\frac{p}{N}\sum_{k}^{(i)}\expct{\frac{1}{\norm{\bsg_{i}}}\bse_{k}\adj G\bse_{i}\frac{\partial \ol{K}_{i}}{\partial g_{ik}}\frY_{i}^{(p-1,p-1)}}+\expct{\rO_{\prec}(N^{-1/2})\frY_{i}^{(p-1,p)}}.
	\end{align*}
	Using (\ref{eq_partialKderivative}), Lemma \ref{lem:recmomerror} and a discussion similar to (\ref{eq:controlc1}), (\ref{eq:controlc2}), and (\ref{eq:controlc3}), we can finish the proof of (\ref{eq:etakrecursive}).
\end{proof}

\section*{Acknowledgements}

The authors would like to thank the Editor, Associate Editor and an anonymous referee for their many critical suggestions which have significantly improved the paper. We also want to thank Zhigang Bao and Ji Oon Lee for many helpful discussions and comments. The first author is partially supported by NSF-DMS 2113489 and grateful for the AMS-SIMONS travel grant (2020-2023).  The second author is supported by ERC Advanced Grant "RMTBeyond" No.~101020331. 

\appendix

\section*{Appendix}
The appendix is organized as follows. In Section \ref{proof_prop32}, we prove the properties of subordination function, i.e., Proposition  \ref{prop:stabN}. Following the proof strategy as described in Section \ref{sec_proofroute} of the main manuscript, using the estimates established in Section \ref{sec_entrylaw} of the main manuscript, we prove the fluctuation averaging estimates for the essential quantities in Section \ref{sec_faall} and provide the proof for the local law Theorem \ref{thm:main} in Section \ref{sec_finalsection}. Once the local laws are proved, in Section \ref{sec_proofofspikedmodel}, we prove Theorems \ref{thm_rigidity} and \ref{thm_delocalization}. We collect and prove some auxiliary lemmas in Appendix \ref{appendix_deriavtive}.

\section{Proof of Proposition \ref{prop:stabN}}\label{proof_prop32}
The proof is divided into two parts. In Section \ref{sec:freelimit}, we prove an analogue of Proposition \ref{prop:stabN} for {the $N$-independent quantities} $\Omega_{\alpha}$, $\Omega_{\beta}$ and $\mu_{\alpha}\boxtimes\mu_{\beta}$ in Lemma \ref{lem:stabbound} and Proposition \ref{prop:stablimit}. 
Then in Section \ref{sec:freeN}, based on the bounds  $\absv{\Omega_{A}(z)-\Omega_{\alpha}(z)}=\rO_{\prec}(N^{-1/2})$ and $\absv{\Omega_{B}(z)-\Omega_{\beta}(z)}=\rO_{\prec}(N^{-1/2}),$ we extend the results of Section \ref{sec:freelimit} to $\Omega_A$ and $\Omega_B,$ and complete the proof of Proposition \ref{prop:stabN}. In Section \ref{sec:omegabound}, we prove the upper bounds of $\absv{\Omega_{A}(z)-\Omega_{\alpha}(z)}$ and $\absv{\Omega_{B}(z)-\Omega_{\beta}(z)}.$ Proofs therein follow the ideas of \cite{Bao-Erdos-Schnelli2016,BEC}, in the sense that we first prove the bound when $\im z$ is large and then use a bootstrapping argument to expand the domain of $z$.

\subsection{Some preliminary results: free convolution of $\mu_{\alpha} \boxtimes \mu_{\beta}$}\label{sec:freelimit}
In this subsection, we collect the results concerning the measure $\mu_{\alpha} \boxtimes \mu_{\beta}$ and its corresponding subordination functions. Most of the results have been proved in \cite[Section 5]{JHC}. 

\begin{lem}[Lemma 5.2 and Proposition 5.6 of \cite{JHC}]\label{lem:stabbound}
	Suppose that $\mu_{\alpha}$ and $\mu_{\beta}$ satisfy Assumption \ref{assu_limit}.
	Then the following statements hold: \\
	(i). For any compact  $\caD \subset \C_{+}\cup(0,\infty)$, there exists some constant $C>1$ such that for all $z\in\caD$,
	\begin{align*}
		C^{-1}\leq \absv{\Omega_{\alpha}(z)}\leq C, \ \  C^{-1}\leq \absv{\Omega_{\beta}(z)}\leq C.
	\end{align*}
	(ii). There exists some constant $\varsigma_{0}>0$ such that for all $z\in\C_{+}$,
	\begin{align} \label{eq:stabbound}
		\dist(\Omega_{\alpha}(z),\supp\mu_{\beta})>\varsigma_{0},  \ \  \dist(\Omega_{\beta}(z),\supp\mu_{\alpha})>\varsigma_{0}.
	\end{align}
	(iii). {Recall (\ref{eq_edge}).} We have that 
	\begin{align} \label{eq_boundedomega}
		&0<\Omega_{\alpha}(E_{-})<E_{-}^{\beta}<E_{+}^{\beta}<\Omega_{\alpha}(E_{+}),	
		&0<\Omega_{\beta}(E_{-})<E_{-}^{\alpha}<E_{+}^{\alpha}<\Omega_{\beta}(E_{+}).
	\end{align}
\end{lem}
The lower bound in the second statement of Lemma \ref{lem:stabbound} is the so-called \emph{stability bound}, as it implies the stability of the systems \eqref{eq_suborsystem} and \eqref{eq_suborsystemPhi}. 

Recall (\ref{eq_defn_salphabeta}). One consequence of the stability bound is that the edges of $\mu_{\alpha} \boxtimes \mu_{\beta}$ in \eqref{eq_edge} are completely characterized by the equation $\caS_{\alpha\beta}(z)=0$. Specifically, in \cite[Proposition 6.15]{JHC}, the author has proved that there are exactly two real positive solutions $z=E_{\pm}$ to the equation $\caS_{\alpha\beta}(z)=0$ and that $\{x\in\R_{+}:\rho(x)>0\}=(E_{-},E_{+})$. Using this as an input, the author in \cite{JHC} also shows that the subordination functions admit the following Pick representations. Analogous results hold for the $L$-transforms of $\mu_{\alpha}$ and $\mu_{\beta}$. 
\begin{lem}[Lemma 5.8 of \cite{JHC}]\label{lem:reprM}
	Suppose that $\mu_{\alpha}$ and $\mu_{\beta}$ satisfy Assumption \ref{assu_limit}. Then there exist unique finite measures $\wh{\mu}_{\alpha}$, $\wh{\mu}_{\beta}$, $\wt{\mu}_{\alpha}$, and $\wt{\mu}_{\beta}$ on $\R_{+}$ such that the following hold:
	\begin{align}
		&L_{\mu_{\alpha}}(z)=1+m_{\wh{\mu}_{\alpha}}(z),&
		&\wh{\mu}_{\alpha}(\R_{+})=\Var{\mu_{\alpha}},& &\supp\wh{\mu}_{\alpha}=\supp\mu_{\alpha},\label{eq:reprM} \\
		&\frac{\Omega_{\alpha}(z)}{z}=1+m_{\wt{\mu}_{\alpha}}(z), & &\wt{\mu}_{\alpha}(\R_{+})=\Var{\mu_{\alpha}}, &
		&\supp \wt{\mu}_{\alpha}=\supp \rho, \nonumber
	\end{align}
	where $\Var{\mu_{\alpha}}$ is the variance of a random variable with law $\mu_{\alpha}$. Similar results hold true if we replace $\alpha$ with $\beta.$ 
\end{lem} 

Based on Lemma \ref{lem:stabbound}, it was proved in \cite{JHC} that the equation $\mathcal{S}_{\alpha \beta}(z)=0$ is  \emph{locally quadratic} for $z$ around the edges $E_{+}$ and $E_{-}$ and consequently, $\Omega_{\alpha}$ and $\Omega_{\beta}$ have square root behavior. Specifically, denote
\beq\label{eq_defnztilde}
\wt{z}_{+}(\Omega)\deq \Omega\frac{M_{\mu_{\alpha}}^{-1} ( M_{\mu_{\beta}}(\Omega) )}{M_{\mu_{\beta}}(\Omega)}.
\eeq
From the results of \cite[Section 6.3]{JHC}, we find that there exists some neighborhood $U$ around $E_+$ such that when $z \in U,$ $\wt{z}_{+}(\Omega_{\alpha}(z))=z$  and
\begin{align}\label{eq:suborTaylor}
	z-E_{+}=\wt{z}_{+}(\Omega_{\alpha}(z))-\wt{z}_{+}(\Omega_{\alpha}(E_{+}))& = \frac{1}{2}\wt{z}_{+}''(\Omega_{\alpha}(E_{+}))(\Omega_{\alpha}(z)-\Omega_{\alpha}(z))^{2} \\
	&	+\rO(\absv{\Omega_{\alpha}(z)-\Omega_{\alpha}(E_{+})}^{3}). \nonumber
\end{align}
Moreover, we have that
$\wt{z}_{+}'(\Omega_{\alpha}(E_{+}))=0$ and $\wt{z}_{+}''(\Omega_{\alpha}(E_{+}))>0.$ Combining \eqref{eq:suborTaylor} with the fact that $\{x\in\R_{+}:\rho(x)>0\}=(E_{-},E_{+})$, we obtain the following lemma. Its proof can be found in \cite[Proposition 5.10]{JHC}.
\begin{lem}\label{lem:suborsqrt}
	Suppose Assumption \ref{assu_limit} holds.	 Then there exist positive constants $\gamma^{\alpha}_{\pm}$, $\gamma^{\beta}_{\pm}$, $0<\tau<1$ and $\eta_{0}$ such that
	\beq \label{eq_squareroot}
	\Omega_{\alpha}(z)=\Omega_{\alpha}(E_{\pm})+\gamma^{\alpha}_{\pm}\sqrt{\pm(z-E_{\pm})}+\rO(\absv{z-E_{\pm}}^{3/2}),
	\eeq
	holds uniformly in $z\in\{z\in\C_{+}: E_+-\tau \leq
	E\leq \tau^{-1},\ 0 \leq \eta<\eta_{0}\}$, where $\sqrt{-1}=\ii.$ The same asymptotics holds with $\alpha$ replaced by $\beta$. Also, for any compact interval $[a,b]\subset(E_{-},E_{+})$, there exists a constant $c>0$ such that for all $x\in[a,b]$
	\begin{align}\label{eq_imaginarypart}
		\im\Omega_{\alpha}(x)>c,\ \ \im\Omega_{\beta}(x)>c.
	\end{align}
\end{lem}

Furthermore, the density of the $\mu_{\alpha} \boxtimes \mu_{\beta}$ also has the square root behavior. 
\begin{lem}[Theorem 3.3 of \cite{JHC}] Suppose that (i) and (ii) of Assumption \ref{assu_limit} hold. Then there exists some constant $C>1$ such that for all $x\in[E_{-},E_{+}]$,
	\begin{equation*}
		C^{-1}\leq \frac{\rho(x)}{\sqrt{(E_{+}-x)(x-E_{-})}} \leq C.
	\end{equation*} 
\end{lem}

To characterize the behavior of $\caS_{\alpha \beta}, \mathcal{T}_{\alpha}$ and $\mathcal{T}_{\beta},$ we need the following lemma. Its proof is analogous to {the counterpart of the additive model in} \cite[Corollary 3.10]{BEC} and {we omit the details}.
\begin{lem}\label{lem:suborder}
	Suppose Assumption \ref{assu_limit} holds. Then	there exist positive constants $\tau$ and $\eta_{0}$ such that the following hold uniformly in $z \in \mathcal{D}_\tau(0,\eta_0)$:
	\begin{align}
		L_{\mu_{\alpha}}'(\Omega_{\beta}(z))\sim  M_{\mu_{\alpha}}'(\Omega_{\beta}(z)) \sim 1, &&
		L_{\mu_{\alpha}}''(\Omega_{\beta}(z))\sim  M_{\mu_{\alpha}}''(\Omega_{\beta}(z)) \sim 1,\label{eq:LMder}\\
		\absv{\Omega_{\alpha}'(z)}\sim\frac{1}{\sqrt{\absv{z-E_{+}}}}, &&
		\absv{\Omega_{\alpha}''(z)}\sim\frac{1}{\absv{z-E_{+}}^{3/2}}\label{eq:suborder}.
	\end{align}
	Similar results hold true if we replace $\alpha$ with $\beta.$
\end{lem}

Finally, we investigate the properties of $\caS_{\alpha\beta}$, $\caT_{\alpha}$, and $\caT_{\beta}$ in the following proposition. {We omit its proof, since it can be easily proved following the proof of \cite[Lemma 6.6]{JHC} provided Lemmas \ref{lem:suborsqrt} and \ref{lem:suborder} hold true. 
	An analogous result for free additive convolution has been established in \cite[Corollary 3.11]{BEC}.}
\begin{prop}\label{prop:stablimit}
	Suppose Assumption \ref{assu_limit} holds. Then there exist constants $\tau,\eta_{0}>0$ such that the following hold:
	\begin{itemize}
		\item[(i)] 
		For $z=E+\ii\eta\in\caD_{\tau}(0,\eta_{0})$ uniformly, we have  
		\begin{align*}
			\im M_{\mu_{\alpha}\boxtimes\mu_{\beta}}(z)\sim \im m_{\mu_{\alpha}\boxtimes\mu_{\beta}}(z) \sim\im \Omega_{\alpha}(z)\sim\im\Omega_{\beta}(z)\sim
			\begin{cases}
				\sqrt{\kappa+\eta} &\text{if }E\in[E_{-},E_{+}],	\\
				\dfrac{\eta}{\sqrt{\kappa+\eta}} &\text{if }E\notin[E_{-},E_{+}],
			\end{cases}
		\end{align*}

		\item[(ii)] {For all fixed $\eta_{0} \leq \eta_{U}$}, there exists a constant $C>0$ such that
		\begin{align}\label{eq_stalphabetabound}
			\caS_{\alpha\beta}(z) &\sim\sqrt{\kappa+\eta}, &|\caT_{\alpha}(z)| &\leq C,& |\caT_{\beta}(z)| &\leq C
		\end{align}
		hold uniformly in $z\in\caD_{\tau}(0,\eta_{0})$.
		\item[(iii)] There exists a constant $\delta>0$ such that 
		\begin{align*}
			&\caT_{\alpha}(z)\sim 1, & &\caT_{\beta}(z)\sim 1
		\end{align*}
		hold uniformly in $z\in\{z\in\C_{+}:\absv{z-E_{+}}<\delta\}$.
	\end{itemize}
\end{prop}

\subsection{Free convolution of $\mu_{A} \boxtimes \mu_{B}$: proof of Proposition \ref{prop:stabN}}\label{sec:freeN}
In this subsection, we prove Proposition \ref{prop:stabN}. More specifically, we will prove our results on a slightly larger domain $\caD_{0},$ where $\caD_{0}\deq \caD_{I}\cup\caD_{O}$ and   
\begin{align}
	\caD_{I}\deq &	\{z\in\C_{+}: \re z\in[E_{+}-\tau ,E_{+}+N^{-1+\xi \epsilon}],\,\im z\in[ N^{-1+\xi \epsilon},\eta_{1}]\}, \nonumber	\\
	\caD_{O}\deq &	\{z\in\C_{+}: \re z\in[E_{+}+N^{-1+\xi \epsilon},\tau^{-1}],\im z\in(0,\eta_{1})\}, \label{eq_defineoutsideset}
\end{align}
where $\xi>1$ is a fixed constant, $\eta_{1}>\eta_{U}$  will be given later in the proof (see Lemma \ref{lem:OmegaBound1} below). Following the strategy of \cite{BEC}, our proof relies on the control of the difference between $(\Omega_{\alpha},\Omega_{\beta})$ and $(\Omega_{A},\Omega_{B})$. 
\begin{lem}\label{lem:OmegaBound}
	Let $\epsilon>0$ and $\mu_{\alpha}$, $\mu_{\beta}$, $\mu_{A}$, and $\mu_{B}$ satisfy Assumptions \ref{assu_limit} and \ref{assu_esd}. Then there exist { some constants $C, C_1>0$} such that for sufficiently large $N$ the following statements hold: 
	\begin{itemize}
		\item[(i)] For all $z\in\caD_{0}$, we have 
		\beq \label{eq:OmegaBound}
		\absv{\Omega_{A}(z)-\Omega_{\alpha}(z)}+\absv{\Omega_{B}(z)-\Omega_{\beta}(z)}
		\leq C\frac{N^{-1+\epsilon}}{\sqrt{\absv{z-E_{+}}}}\leq C_1 N^{-1/2+\epsilon}.
		\eeq 
		\item[(ii)] For all $z\in\caD_{O},$ we have  
		\begin{align}\label{eq:OmegaImBound}
			\absv{\im\Omega_{A}(z)-\im\Omega_{\alpha}(z)}& +\absv{\im\Omega_{B}(z)-\im\Omega_{\beta}(z)} \nonumber \\
			& \leq C\frac{N^{-1+\epsilon}\im(\Omega_{\alpha}(z)+\Omega_{\beta}(z))+\im z}{\sqrt{\absv{z-E_{+}}}}.
		\end{align}
	\end{itemize}
\end{lem}

We first present the proof of Proposition \ref{prop:stabN} using Lemma \ref{lem:OmegaBound} and postpone its proof to Section \ref{sec:omegabound}. Along the proof of Proposition \ref{prop:stabN}, we require two additional lemmas. The first result is an analogue of Lemma \ref{lem:reprM} for the ESDs  $\mu_{A}$ and $\mu_{B}$ {which can be proved similarly as Lemma \ref{lem:reprM}.} We omit the details. 
\begin{lem}\label{lem:reprMemp}
	Suppose that $\mu_{A}$ and $\mu_{B}$ satisfy Assumptions \ref{assu_limit} and \ref{assu_esd}. Then there exist unique probability measures $\wh{\mu}_{A}$, $\wh{\mu}_{B}$, $\wt{\mu}_{A}$, and $\wt{\mu}_{B}$ on $\R_{+}$ such that the following hold:
	\begin{align} \label{eq_newsupportbound}
		L_{\mu_{A}}(z)&=1+m_{\wh{\mu}_{A}}(z),	\
		\wh{\mu}_{A}(\R_{+})=\Var{\mu_{A}},	 \
		\supp\wh{\mu}_{A}\subset [a_{N},a_{1}], \nonumber \\
		&	\frac{\Omega_{A}(z)}{z}=1+m_{\wt{\mu}_{A}}(z), \ 
		\wt{\mu}_{A}(\R_{+})=\Var{\mu_{A}}, \nonumber \\
		&	\supp\wt{\mu}_{A}\subset[\inf \supp \mu_{A}\boxtimes\mu_{B},\sup\supp \mu_{A}\boxtimes\mu_{B}].
	\end{align}
	Similar results hold true if we replace $A$ with $B.$ 
\end{lem}
We remark that by (\ref{eq_newsupportbound}), we have 
\begin{align} 
	\im\Omega_{A}(z)&=\im (z+zm_{\mu_{\wt{A}}}(z))=\eta+\eta\int\frac{x}{\absv{x-z}^{2}}\dd\wt{\mu}_{A}(x)\geq\eta, \label{eq_omegabbound1} \\ 
	\im\frac{\Omega_{A}(z)}{z}&=\eta\int\frac{1}{\absv{x-z}^{2}}\dd\wt{\mu}_{A}\geq \Var{\mu_{A}}\frac{\eta}{2(\absv{z}^{2}+\norm{A}^{2}\norm{B}^{2})} \label{eq_omegabbound},
\end{align}
where in the second step of (\ref{eq_omegabbound}) we used  (\ref{eq:priorisupp}). 

We next introduce the second lemma. We will frequently refer to this lemma whenever we need to bound the differences between different Stieltjes transforms in terms of $\caL(\mu_{\alpha},\mu_{A})+\caL(\mu_{\beta},\mu_{B}).$ {It can be easily verified by some elementary calculations} and we omit the details. 
\begin{lem}\label{lem:rbound}
	Let $\delta>0$ be fixed and $f:\R\to\C$ be a continuous function with continuous derivative on $(E_{-}^{\alpha}-\delta,E_{+}^{\alpha}+\delta)$. Suppose that there exists $N_0 \in \mathbb{N}$ independent of $f$ such that $\bsd\leq \delta/2$ and $[a_{N},a_{1}]\in[E_{-}^{\alpha}-\delta/2,E_{+}^{\alpha}+\delta/2]$ for all $N\geq N_{0}$. Then there exists some constant $C>0$, independent of $f$, such that we have
	\beqs
	\Absv{\int_{\R_{+}} f(x)\dd\mu_{\alpha}(x)-\int_{\R_{+}} f(x)\dd\mu_{A}(x)}\leq C \norm{f'}_{\mathrm{Lip},\delta}\bsd,
	\eeqs
	for all $N\geq N_{0}$, where we denoted
	\beqs
	\norm{f'}_{\mathrm{Lip},\delta}\deq\sup_{{x\in[E_{-}^{\alpha}-\delta,E_{+}^{\alpha}+\delta]}}\absv{f'(x)}+\sup\left\{\Absv{\frac{f'(x)-f'(y)}{x-y}}:x,y\in[E_{-}^{\alpha}-\delta,E_{+}^{\alpha}+\delta], x\neq y\right\}.
	\eeqs
	The same result holds if we replace $\alpha$ and $A$ by $\beta$ and $B$.
\end{lem}

Now we are ready to present the proof of Proposition \ref{prop:stabN}. {Since its proof follows closely to the counterpart of the additive model in \cite[Proposition 3.1]{BEC},  we only sketch the key points.}
\begin{proof}[\bf Proof of Proposition \ref{prop:stabN}]
	Throughout the proof we choose $\eta_1 \geq \eta_U$ such that
	$\caD_{\tau}(\eta_{L},\eta_{U})\subset\caD_{0}.$ It is easy to see that (i) follows directly from Lemmas \ref{lem:stabbound} and \ref{lem:OmegaBound}.
	
	Next, we proceed to prove (ii). By (\ref{eq:OmegaImBound}), it is easy to verify that
	\beq\label{eq_derivativecontroldifference}
	\absv{\Omega_{\beta}(z)-\Omega_{B}(z)}\leq C\frac{N^{-1+\epsilon}}{\sqrt{\kappa+\eta}} 
	\ll\left\{
	\begin{array}{cl}
		\sqrt{\kappa+\eta}		&	\text{ if }E\in[E_{-},E_{+}],\\
		\frac{\eta}{\sqrt{\kappa+\eta}}	&	\text{ if }E>E_{+},
	\end{array}
	\right.
	\eeq
	holds uniformly in $z\in\caD_{0}$ when $N$ is sufficiently large. Then, combining \eqref{eq_multiidentity}, (i) of Proposition \ref{prop:stabN},  Lemma \ref{lem:stabbound}, (i) of Proposition \ref{prop:stablimit}, and \eqref{eq_derivativecontroldifference}, we have the following chain of equivalences for $z\in\caD_{0}$:
	\begin{equation*}
		\im m_{\mu_{A}\boxtimes\mu_{B}}(z)	\sim\im zm_{\mu_{A}\boxtimes\mu_{B}}(z)	\sim\im\Omega_{B}(z)	\sim\im\Omega_{\beta}(z)\sim\im zm_{\rho}(z)	\sim\im m_{\rho}(z).
	\end{equation*}
	This finishes the proof of (ii). 
	
	The proofs of (iii) and (iv) are analogous to those of \cite[Proposition 3.1 (iii) and (iv)]{BEC}, and we only discuss the necessary modifications. Firstly, we need $\absv{\Omega_{\alpha}(z)}\sim 1$ and $\absv{\Omega_{\beta}(z)}\sim 1$ from Lemma \ref{lem:stabbound} as additional inputs, since $\mathtt{L}$-transforms in \eqref{eq_mtrasindenity} has a factor of $z^{-1}$. Secondly, we apply Lemma \ref{lem:rbound} to the functions $x\mapsto x(x-\Omega_{\alpha}(z))^{-1}$ and $x^{2}(x-\Omega_{\alpha}(z))^{-2}$, in contrast to $x\mapsto (x-\omega_{\alpha}(z))^{-1}$ and $(x-\omega_{\alpha}(z))^{-2}$ in \cite{BEC}.
	Lastly, we need the fact that $\absv{z}\sim 1$ for the proof of (iv) due to a factor of $z^{-1}$ in the identity
	\beqs
	\begin{pmatrix}
		\Omega_{A}'(z) \\
		\Omega_{B}'(z) 
	\end{pmatrix}
	=-\frac{1}{z\caS_{AB}(z)}
	\begin{pmatrix}
		zL'_{\mu_{A}}(\Omega_{B}(z)) & 1 \\
		1 & zL'_{\mu_{B}}(\Omega_{A}(z))
	\end{pmatrix}
	\begin{pmatrix}
		\Omega_{B}(z) \\
		\Omega_{A}(z)
	\end{pmatrix}.
	\eeqs
	Once these modifications are applied, (iii) and (iv) of Proposition \ref{prop:stabN} can be proved following \cite[Propsition 3.1]{BEC}.
\end{proof}

\subsection{Closeness of subordination functions: proof of Lemma \ref{lem:OmegaBound}}\label{sec:omegabound}
In this subsection, we prove Lemma \ref{lem:OmegaBound}. By the construction of (\ref{eq_suborsystemPhi}), we find that $\Omega_A(z)$ and $\Omega_B(z)$ are determined by the equation $$\Phi_{AB}(\Omega_A, \Omega_B,z)=0.$$ Equivalently, $\Omega_A$ and $\Omega_B$ are solutions of the equations 
\begin{align}\label{eq:SuborPerturb}
	\Phi_{\alpha}(\omega_{1},\omega_{2},z)=r_{1}(z),	\
	\Phi_{\beta}(\omega_{1},\omega_{2},z)=r_{2}(z),
\end{align}
where $\Phi_{\alpha}$ and $\Phi_{\beta}$ are defined in \eqref{eq:def_Phi_ab} and
\begin{align}\label{eq_defnra}
	r_{1}(z)&\deq \frac{M_{\mu_{\alpha}}(\omega_2)-M_{\mu_{A}}(\omega_2)}{\omega_2}, &
	r_{2}(z)&\deq \frac{M_{\mu_{\beta}}(\omega_1)-M_{\mu_{B}}(\omega_1)}{\omega_1}.
\end{align}
In what follows, we use the notations $r_A(z) \equiv r_1(\Omega_B(z))$ and $r_B(z) \equiv r_2(\Omega_A(z)).$ Ideally, both $r_A(z)$ and $r_B(z)$ should be small and we aim to bound $\absv{\Omega_{A}-\Omega_{\alpha}(z)}$ and $\absv{\Omega_{B}(z)-\Omega_{\beta}(z)}$ in terms of the norm of $r(z)\deq (r_{A}(z),r_{B}(z))$. 

{We prepare two technical lemmas: Lemmas \ref{lem:OmegaBound1} and \ref{lem:OmegaBound2}. The first one provides a control for the subordination functions. Its proof directly follows from Lemmas \ref{lem:Kantorovich_appl} and \ref{lem:reprMemp} and we omit the details here.}
\begin{lem}\label{lem:OmegaBound1}
	There exists a constant $\eta_{1}>0$ such that for all $z\in\C_{+}$ with $\re z\in [E_{+}-\tau,\tau^{-1}]$ and $\im z=\eta_{1}$ the following hold:
	\begin{align*}
		\absv{\Omega_{\alpha}(z)-\Omega_{A}(z)}&\leq 2\norm{r(z)},	&
		\absv{\Omega_{\beta}(z)-\Omega_{B}(z)}&\leq 2\norm{r(z)},
	\end{align*}
	for all sufficiently large $N$.
\end{lem}

Denote $K_{1}\deq4K_{4}$ and $K_{2}\deq (4K_{3}K_{4})^{-1}$, where
\begin{equation*}
	K_{3}\deq 27\varsigma_{0}^{-3}\max(\wh{\mu}_{\alpha}(\R_{+}),\wh{\mu}_{\beta}(\R_{+})),  
\end{equation*}
with the constant $\varsigma_{0}$ in \eqref{eq:stabbound} and
\begin{equation*}
	K_{4}\deq \sup\left\{\absv{z}+\absv{z^{2}L_{\mu_{\alpha}}'(\Omega_{\beta}(z))},\absv{z}+\absv{z^{2}L_{\mu_{\beta}}'(\Omega_{\alpha}(z))}:z\in\caD_{\tau}(0,\eta_{1})\right\},
\end{equation*}
where $\eta_{1}$ is introduced in Lemma \ref{lem:OmegaBound1}. It is easy to see that both $K_3$ and $K_4$ are positive finite numbers. We next state the second lemma. Recall (\ref{eq_defn_talpha}). 
\begin{lem}\label{lem:OmegaBound2}
	For sufficiently large  $N$ and $z_{0}\in\caD_{\tau}(0,\eta_{1}),$ where $\eta_{1}$ is introduced in Lemma \ref{lem:OmegaBound1}, suppose that there exists some $q>0$ such that the following hold:
	\begin{itemize}
		\item[(i).] $\displaystyle q\leq \frac{1}{3}\varsigma_{0};$
		\item[(ii).]
		$\displaystyle
		\absv{\Omega_{A}(z_{0})-\Omega_{\alpha}(z_{0})}\leq q,	\quad
		\absv{\Omega_{B}(z_{0})-\Omega_{\beta}(z_{0})}\leq q;
		$
		\item[(iii).] $\displaystyle q\leq \frac{1}{2}K_{2}\caS_{\alpha\beta}(z_{0})$.
	\end{itemize}
	Then we have 
	\beqs
	\absv{\Omega_{A}(z_{0})-\Omega_{\alpha}(z_{0})}+\absv{\Omega_{B}(z_{0})-\Omega_{\beta}(z_{0})}\leq K_{1}\frac{\norm{r(z_{0})}}{\absv{\caS_{\alpha\beta}(z_{0})}}.
	\eeqs
\end{lem}
\begin{proof}
	Since the proof is analogous to that of \cite[Lemma 3.13]{BEC}, we only sketch the main ideas here. Firstly, using Lemma \ref{lem:stabbound} and the Taylor expansion of $L_{\mu_{\alpha}}(\Omega_{B}(z_{0}))$ around $\Omega_{\beta}(z_{0})$ leads to 
	\begin{equation}\label{eq_sbound1}
		\Absv{L_{\mu_{\alpha}}'(\Omega_{\beta}(z_{0}))\Delta\Omega_{2}(z_{0})-\frac{\Delta\Omega_{1}(z_{0})}{z_{0}}}
		\leq\norm{r(z_{0})}+ K_{3}\norm{\Delta\Omega(z_{0})}^{2}.
	\end{equation}
	Similarly, the same inequality with $\alpha$ and $\beta$ interchanged and $\Omega_{1}$ replaced by $\Omega_{2}$ holds true.
	
	Then, combining the definition of $\caS_{\alpha\beta}$ in \eqref{eq_defn_salphabeta} with \eqref{eq_sbound1}, we get
	\begin{align*}
		&\absv{\caS_{\alpha\beta}(z_{0})}\absv{\Delta\Omega_{1}(z_{0})}\leq K_{4}\norm{r(z_{0})}+K_{3}K_{4}\norm{\Delta\Omega(z_{0})}^{2}
	\end{align*}
	and the same inequality with $\alpha$ and $\beta$ interchanged. This lead to the following quadratic inequality for $\Delta\Omega(z_{0})$:
	\begin{align*}
		\norm{\Delta\Omega(z_{0})}\leq |\Delta \Omega_1(z_0)|+|\Delta \Omega_2(z_0)| &		\leq\frac{1}{\absv{\caS_{\alpha\beta}(z_{0})}}\left(\frac{K_{1}}{2}\norm{r(z_{0})}+\frac{1}{2K_{2}}\| \Delta\Omega(z_{0})\|^{2}\right).
	\end{align*}
	By the assumption (iii) we can solve the quadratic inequality so that the result follows.
\end{proof}

{
	Finally, we prove Lemma \ref{lem:OmegaBound}.
	\begin{proof}[\bf Proof of Lemma \ref{lem:OmegaBound} ] Following \cite[Lemma 3.12]{BEC}, given Lemmas \ref{lem:OmegaBound1} and \ref{lem:OmegaBound2}, Lemma \ref{lem:OmegaBound} can be proved after making the same modification to \cite[Lemma 3.12]{BEC} as pointed out in the proof of {(iii) and (iv) of} Proposition \ref{prop:stabN}. We omit further details here. 
		
	\end{proof}
}

\section{Fluctuation averaging}\label{sec_faall}
In this section, we prove the results for some averaged quantities.  Throughout the paper, we use the following control parameter 
\beq\label{eq_controlparameter1}
\Pi\equiv\Pi(z)\deq \sqrt{\frac{\im m_{\mu_A \boxtimes \mu_B}(z)}{N\eta}},\ .
\eeq
In our proof, we will frequently use the following identities
\begin{equation}\label{eq_gggconnetction}
	G=A^{1/2}\wt{G}A^{-1/2}, \caG=B^{-1/2}\wt{\caG}B^{1/2}, 
\end{equation} 
\beq\label{eq_ggrelationship}
G\adj G =A^{-1/2}\wt{G}\adj A \wt{G}A^{-1/2}, \  \caG\adj \caG= B^{1/2}\wt{\caG}\adj B^{-1}\wt{\caG} B^{1/2}.
\eeq
\subsection{Rough fluctuation averaging}\label{sec_roughfa}
As we have seen from Proposition \ref{prop:entrysubor}, the error bounds are not optimal compared to our final results in Theorem \ref{thm:main}. More specifically, the obtained control for $\Upsilon$ in (\ref{eq_upsilonbound}) is too loose. In this subsection, we improve the control for $\Upsilon$ based on Proposition \ref{prop:entrysubor}, which are better estimates compared to (\ref{eq_locallaweqbound}). More specifically, instead of using (\ref{eq_boundepsilon1}) and (\ref{eq_controlepsilon2}), we will conduct a more careful analysis for $\mathsf{e}_{i1}$ and $\mathsf{e}_{i2}$ defined in (\ref{eq_epsilon1}) and (\ref{eq_defnmathsfe2}), respectively.  The improved estimates will be utilized in Section \ref{subsec_stronglocallawfixed} to prove Theorem \ref{thm:main} for each fixed spectral parameter $z.$ 

The main result of this subsection is Proposition \ref{prop:FA1} below, which provides the estimates for an extension of $\Upsilon.$ More specifically, we will control weighted averages of $Q_{i}$'s, i.e., 
\beq\label{eq_weighted}
\frX \equiv \frX(D):=	\frac{1}{N}\sum_{i}d_{i}Q_{i}=\tr(GA\wt{B})\tr(GD)-\tr(GA)\tr(\wt{B}GD),
\eeq
where $D=\operatorname{diag}\{d_1, \cdots, d_N\}$ and $d_{i} \equiv d_i(H)$ are generic weights which    in general are functions of $H.$ It is necessary and natural for us to consider such a generalization.  For instance, in the decomposition (\ref{eq:Lambda}), we have an extra factor, which is a function of $H,$ in front of $Q_i.$ We first impose some assumptions regarding the concentration properties on $d_i, i=1,2,\cdots, N.$ It will be seen later that the following assumption will be sufficient for most of our applications. Especially, when all $d_i, i=1,2,\cdots, N,$ are functions irrelevant of $H$, Assumption \ref{assum_conditiond} will hold trivially.

\begin{assu}\label{assum_conditiond} Let $X_{i}=I$ or $\wt{B}^{\angi}$, and let $d_{1},\cdots,d_{N}$ be functions of $H$ with $\max_{i}\absv{d_{i}}\prec 1$. Assume that for all $i,j\in \llbra 1,N\rrbra$ the following hold:
	\begin{align}\label{eq:weight_cond}
		\frac{1}{N}\sum_{k}^{(i)}\frac{\partial d_{j}}{\partial g_{ik}}\bse_{k}\adj X_{i}G\bse_{i}&=\rO_{\prec}(\Psi^{2}\Pi_{i}^{2}),  \ \
		\frac{1}{N}\sum_{k}^{(i)}\frac{\partial d_{j}}{\partial g_{ik}}\bse_{k}\adj X_{i}\mr{\bsg}_{i}=\rO_{\prec}(\Psi^{2}\Pi_{i}^{2}),
	\end{align}
	and the same bound also holds with $d_{j}$'s replaced by $\ol{d}_{j}$.
\end{assu}
Now we state the main result of this subsection. 
\begin{prop}\label{prop:FA1}
	Fix $z\in\caD_{\tau}(\eta_{L},\eta_{U})$ and suppose that the assumptions of Proposition \ref{prop:entrysubor} and Assumption \ref{assum_conditiond} hold. Let $\Pi(z)\prec\wh{\Pi}(z)$ for some deterministic positive $\wh{\Pi}(z)$ with $N^{-1/2}\eta^{-1/4}\prec \wh{\Pi}\prec\Psi$. Then we have that 
	\beqs
	\frX \prec\Psi\wh{\Pi}.
	\eeqs
\end{prop}
\begin{proof}[\bf Proof of Proposition \ref{prop:FA1}] Denote
	\begin{align*}
		\frX^{(p,q)}&\deq \frX^{p}\ol{\frX}^{q},\quad  p,q \in\N.
	\end{align*}
	We claim that the following recursive estimates hold for $\frX.$	
	\begin{lem}\label{lem:Avrecmoment}
		For any fixed integer $p\geq 2,$ we have that
		\begin{align}\label{eq_roughfarecuversivestimateeq}
			\expct{\frX^{(p,p)}}\leq \expct{\rO_{\prec}(\wh{\Pi}^{2})\frX^{(p-1,p)}+\rO_{\prec}(\Psi^{2}\wh{\Pi}^{2})\frX^{(p-2,p)}+\rO_{\prec}(\Psi^{2}\wh{\Pi}^{2})\frX^{(p-1,p-1)}}.
		\end{align}
	\end{lem}
	\noindent By a discussion similar to (\ref{eq_arbitraydiscussion}), together with  Lemma \ref{lem:Avrecmoment} and Markov inequality, we can complete the proof.  
\end{proof}

The rest of this subsection is devoted to proving Lemma \ref{lem:Avrecmoment}, where the proof is similar to that of Lemma \ref{lem:PKrecmoment}, except that we use Proposition \ref{prop:entrysubor} as an input. Since the proof is similar to \cite[Lemma 6.2]{BEC}, we focus on {explaining} the main ideas and how our proof differs from that of \cite{BEC}. {Particularly, the most significant difference is that, unlike \cite[Lemma 6.2]{BEC}, both $\mathsf{e}_{i1}$ and $\mathsf{e}_{i2}$ in our model will generate some $\rO_{\prec}(N^{-1/2})$ terms. The weighted summations  of these terms will be canceled out algebraically after we explore some hidden identities; see (\ref{eq_epsilon1details})--(\ref{eq:expan_ei2}) and the associated discussion for more details. }

\begin{proof}[\bf Proof of Lemma \ref{lem:Avrecmoment}] First of all, following the proof of \cite[Lemma 6.2]{BEC}, we see that it suffices to prove the following statement: if $\wh{\Upsilon}(z)$ is another deterministic control parameter such that $\absv{\Upsilon(z)}\prec \wh{\Upsilon}(z)\leq \Psi(z)$, then 
	\begin{align}\label{eq:Avrecmomentrec}
		\expct{\frX^{(p,p)}}\leq & \expct{\rO_{\prec}(\wh{\Pi}^{2}+\Psi\wh{\Upsilon})\frX^{(p-1,p)}+\rO_{\prec}(\Psi^{2}\wh{\Pi}^{2})\frX^{(p-2,p)}  \right. \nonumber \\
			&	\left. 	+\rO_{\prec}(\Psi^{2}\wh{\Pi}^{2})(\frX^{(p-1,p-1)}}.
	\end{align}
	Indeed, the fact that \eqref{eq:Avrecmomentrec} implies \eqref{eq_roughfarecuversivestimateeq} follows from the same proof as \cite[Lemma 6.2]{BEC}, and we simply recall the main arguments. First, we apply \eqref{eq:Avrecmomentrec} to the weights $d_{i}\equiv z$ and use Young's inequality, iteratively with smaller $\wh{\Upsilon}$ each time. Since the average with respect to $d_{i}\equiv z$ is exactly $\Upsilon$, we obtain the bound $\absv{\Upsilon}\prec\Psi\wh{\Pi}$ as a result. Secondly, we feed this bound back into \eqref{eq:Avrecmomentrec} to get \eqref{eq_roughfarecuversivestimateeq}.
	
	It remains to prove \eqref{eq:Avrecmomentrec}. Recall $D=\diag\{d_{1},\cdots,d_{N}\}$ and (\ref{eq_defnq}). We see that 
	\begin{align}\label{eq_keyexpansionfrx}
		\frX=\frac{1}{N}\sum_{i}d_{i}Q_{i}
		&		=\frac{1}{N}\sum_{i=1}^{N}(\wt{B}G)_{ii}\tau_{i1}\tr(GA),
	\end{align}
	where we denoted
	\beq\label{eq_defntauil}
	\tau_{i1}\deq \frac{a_{i}\tr(GD)}{\tr(GA)}-d_{i}.
	\eeq
	
	We state two important properties regarding $\tau_{i1}.$ First, from the definition of $\Omega_{B}^{c}(z)$ and the stability bound, we have $\absv{\tau}_{i1}\prec1$ under the assumptions of Proposition \ref{prop:FA1}. Second, we have
	\beq\label{eq:tau_aver}
	\sum_{i}G_{ii}\tau_{i1}=\frac{1}{\tr(GA)}\sum_{i}G_{ii}a_{i}\tr(GD)-\tr(GD)=0.
	\eeq
	
	Inserting (\ref{eq:BGii}) into (\ref{eq_keyexpansionfrx}), we obtain that 
	\begin{align}
		\expct{\frX^{(p,p)}}
		&=-\frac{1}{N}\sum_{i=1}^{N}\expct{\mr{S}_{i}\tau_{i1}\tr(GA)\frX^{(p-1,p)}}
		+\frac{1}{N}\sum_{i=1}^{N}\expct{G_{ii}\tau_{i1}\tr(GA)\frX^{(p-1,p)}} \nonumber\\
		&+\frac{1}{N}\sum_{i=1}^{N}\expct{T_{i}\tau_{i1}\tr(GA)\frX^{(p-1,p)}} 
		+\frac{1}{N}\sum_{i=1}^{N}\expct{\mathsf{e}_{i1}\tau_{i1}\tr(GA)\frX^{(p-1,p)}}, \label{eq:recmomX1}
	\end{align}
	where $\mathsf{e}_{i1}$ is defined in \eqref{eq_epsilon1}. It suffices to estimate all the terms of the RHS of (\ref{eq:recmomX1}). Due to (\ref{eq:tau_aver}), the second term onf the RHS of (\ref{eq:recmomX1}) vanishes.
	For the remaining terms, we apply (\ref{eq_formulaintergrationbyparts}) to estimate them. We start with the first term of the RHS of (\ref{eq:recmomX1}). Recall (\ref{eq_rewritemrs}). We have that 
	\begin{align}\label{eq:recmomX2}
		\expct{\mr{S}_{i}\tau_{i1}\tr (GA)\frX^{(p-1,p)}}
		&=\frac{1}{N}\sum_{k}^{(i)}\expct{\frac{1}{\norm{\bsg_{i}}}\frac{\partial (\bse_{k}\adj\wt{B}^{\angi}G\bse_{i})}{\partial g_{ik}}\tau_{i1}\tr (GA)\frX^{(p-1,p)}} \quad  \nonumber\\
		&+\frac{1}{N}\sum_{k}^{(i)}\expct{\bse_{k}\adj \wt{B}^{\angi}G\bse_{i}\frac{\partial}{\partial g_{ik}}\left(\norm{\bsg_{i}}^{-1}\tau_{i1}\tr(GA)\frX^{(p-1,p)}\right)}.
	\end{align}
	For the first term of the RHS of (\ref{eq:recmomX2}), by (\ref{eq_finalexpansion}), we follow \cite[(6.20)]{BEC} to get
	\begin{align}\label{eq:recmomX3}\nonumber
		\frac{1}{N^{2}}\sum_{i=1}^N\sum_{k}^{(i)}&\expct{\frac{1}{\norm{\bsg_{i}}}\frac{\partial (\bse_{k}\adj\wt{B}^{\angi}G\bse_{i})}{\partial g_{ik}}\tau_{i1}\tr (GA)\frX^{(p-1,p)}}
		=\frac{1}{N}\sum_{i}\expct{T_{i}\tau_{i1}\tr(GA)\frX^{(p-1,p)}} \nonumber\\
		&+\frac{1}{N^{2}}\sum_{i}\sum_{k}^{(i)}\expct{G_{ki}\frac{\partial }{\partial g_{ik}}\left(\norm{\bsg_{i}}^{-1}\tau_{i1}\tr(A\wt{B}G)\tau_{i1}\frX_{i}^{(p-1,p)}\right)}	\nonumber \\
		&+\frac{1}{N}\sum_{i}\expct{\frac{\mathsf{e}_{i2}\tau_{i1}}{\norm{\bsg_{i}}}\frX^{(p-1,p)}} +\expct{\big(\rO_{\prec}(\wh{\Upsilon}(z)\Psi(z))+\rO_{\prec}(\Pi^{2})\big)\frX^{(p-1,p)}},
	\end{align}
	where $\mathsf{e}_{i2}$ is defined in \eqref{eq_defnmathsfe2}. Here, we used \eqref{eq:tau_aver} and $\absv{\Upsilon}\prec\wh{\Upsilon}$ to the first term of \eqref{eq_finalexpansion}, \eqref{eq:recmomP2} to the second, and the fact that $\frac{1}{N}\sum_{i}\Pi_{i}^{2}\prec\Pi^{2}$ to all terms. Note that the first term on the RHS of \eqref{eq:recmomX3} cancels with the third term of \eqref{eq:recmomX1}.
	
	At this point, we deviate a little bit from the current calculation and revisit the last term on the RHS of (\ref{eq:recmomX1}). By the definition of $\mathsf{e}_{i1},$ (\ref{eq_hiicontrol}), (\ref{eq_licontrol}), (\ref{eq_controlhibhi}) and Proposition \ref{prop:entrysubor}, we have that
	\begin{align}\label{eq_epsilon1details}
		\mathsf{e}_{i1}
		&	=-h_{ii}G_{ii}+\frac{\Omega_{B}^{c}}{z}\frac{\bsh_{i}\adj(\wt{B}^{\angi}-I)\bsh_{i}}{a_{i}-\Omega_{B}^{c}}+\rO_{\prec}(N^{-1/2}\Psi).
	\end{align}
	{Compared to the counterpart of the additive model in equation (6.14) of \cite{BEC}, we have an extra term $-h_{ii}G_{ii}$ on the RHS of (\ref{eq_epsilon1details}). As we will see in (\ref{eq:expan_ei1}) and (\ref{eq:expan_ei2}), a weighted sum of this term will be canceled out with a term generated from $\mathsf{e}_{i2}$ (c.f. the second term on the RHS of (\ref{eq_ei2definition})). 
		By a discussion similar to \cite[(6.26)]{BEC}, it is easy to see that} the last term in \eqref{eq:recmomX1} can be written as 
	\begin{align}\label{eq:expan_ei1}
		\expct{\frac{1}{N}\sum_{i}\mathsf{e}_{i1}\tau_{i1}\tr(GA)\frX^{(p-1,p)}}
		=-\expct{\frac{1}{N}\sum_{i}h_{ii}G_{ii}\tau_{i1}\tr(GA)\frX^{(p-1,p)}} \nonumber \\
		+\frac{1}{z}\frac{1}{N^{2}}\sum_{i}\sum_{k}^{(i)}\expct{\frac{\bse_{k}\adj(\wt{B}^{\angi}-I)\mr{\bsg}_{i}}{\norm{\bsg_{i}}^{2}}\frac{\partial }{\partial g_{ik}}\left(\tau_{i2}\frX^{(p-1,p)}\right)}
		+\expct{\rO_{\prec}(\wh{\Pi}^{2})\frX^{(p-1,p)}},
	\end{align}
	where $\tau_{i2}$ is defined as 
	\beqs
	\tau_{i2}\deq\frac{\Omega_{B}^{c}(a_{i}\tr(GD)-d_{i}\tr(GA))}{a_{i}-\Omega_{B}^{c}}.
	\eeqs
	
	Now we return to (\ref{eq:recmomX3}) and investigate the only remaining term, that is, the third term. Recall (\ref{eq_defnmathsfe2}). By a discussion similar to (\ref{eq_epsilon1details}), we find that 
	\begin{align}
		\mathsf{e}_{i2}
		=&(\norm{\mr{\bsg}_{i}}^{2}-1)\frac{\Omega_{B}^{c}}{z}\frac{\tr(A(\wt{B}-I)G)}{a_{i}-\Omega_{B}^{c}}-h_{ii}G_{ii}\tr(GA)+\rO_{\prec}(N^{-1/2}\Psi).\label{eq_ei2definition}
	\end{align}
	{As discussed above, the second term on the RHS of (\ref{eq_ei2definition}) does not appear for the additive model in \cite{BEC}; see the equation below (6.29) therein. A weighted sum will be canceled out with the extra term generated by $\mathsf{e}_{i1};$ see (\ref{eq:expan_ei1}) and (\ref{eq:expan_ei2}) for more detail. }	
	Using the definitions in (\ref{eq_prd2}), by (\ref{eq_largedeviationbound}), (\ref{eq_ei2definition}) and a discussion similar to (\ref{eq:expan_ei1}), we find that the third term of the RHS of (\ref{eq:recmomX3}) can be written as  	
	\begin{align}\label{eq:expan_ei2}
		\frac{1}{N}\sum_{i}\expct{\frac{\mathsf{e}_{i2}\tau_{i1}}{\norm{\bsg_{i}}}\frX^{(p-1,p)}}  
		&		=\frac{1}{N^{2}z}\sum_{i}\sum_{k}^{(i)}\expct{\bse_{k}\adj \mr{\bsg_{i}}\frac{\partial }{\partial g_{ik}}\left(\tau_{i3}\frX^{(p-1,p)}\right)} \nonumber  \\
		&	-\frac{1}{N}\sum_{i}\expct{h_{ii}G_{ii}\tau_{i1}\tr(GA)\frX^{(p-1,p)}}+\expct{\rO_{\prec}(N^{-1/2}\Psi)\frX^{(p-1,p)}}, 
	\end{align}
	where we denoted $\tau_{i3}\deq \tau_{i2}\tr(A(\wt{B}-I)G).$
	We emphasize that the second term of the RHS of (\ref{eq:expan_ei2}) is canceled out with the first term of the RHS of (\ref{eq:expan_ei1}). 
	
	Collecting all results above, we can write (\ref{eq:recmomX1}) as follows;
	\begin{equation}\label{eq_explicitlysolution}
		\mathbb{E}\left[ \frX^{(p,p)} \right]=\mathbb{E}\left[ \mathfrak{D}_1 \frX^{(p-1,p)} \right]+\mathbb{E}\left[ \mathfrak{D}_2 \frX^{(p-2,p)} \right]+\mathbb{E}\left[ \mathfrak{D}_3 \frX^{(p-1,p-1)} \right],
	\end{equation}	
	where $\mathfrak{D}_k, k=1,2,3,$ can be computed explicitly. Finally, since all derivatives involved in $\mathfrak{D}_{k}$'s concern only (weighted) traces of $G$ or $\norm{\bsg_{i}}$, we can bound $\mathfrak{D}_k, k=1,2,3,$ using a discussion similar to (\ref{eq_defnmathfrackc1}), (\ref{eq_defnmathfrackc2}) and (\ref{eq_defnmathfrackc3}) utilizing Lemmas \ref{lem:DeltaG} and \ref{lem:recmomerror} {as in the proof of Lemma \ref{lem:PKrecmoment}}. We omit the details here.
	
\end{proof}

\subsection{Optimal fluctuation averaging}\label{subsec_stronglocallawfixed}

In this subsection, we will establish an estimate for the key quantities regarding the stability of system which masters the subordination functions (c.f. (\ref{eq_subordinationsystemab}) below). Such an estimation relies on strengthening the estimates obtained in Proposition \ref{prop:FA1} for certain explicit choices of $d_i, i=1,2,\cdots,N;$ see (\ref{eq_optimalfaquantities}) and (\ref{eq_optimalfaquantitiescoeff}) for details.  These results will be a base for the proof of Theorem \ref{thm:main}. Moreover, as one important byproduct, we can show the closeness of the subordination functions and their approximations in Definition \ref{defn_asf}.

Denote 
\begin{align*}
	&\Lambda_{A}(z) \deq \Omega_{A}(z)-\Omega_{A}^{c}(z), &
	&\Lambda_{B}(z)\deq \Omega_{B}(z)-\Omega_{B}^{c}(z), &
	&\Lambda(z) \deq \absv{\Lambda_{A}(z)}+\absv{\Lambda_{B}(z)}.
\end{align*}
As we have seen in Lemma \ref{lem:OmegaBound2}, the control of $\Lambda$ are expected to reduced to studying a system similar to (\ref{eq_suborsystemPhi}). Recall that $\Phi_{AB} \equiv (\Phi_A, \Phi_B) \in \mathbb{C}_+^3$ and the subordination functions $\Omega_A$ and $\Omega_B$ are governed by the following system{
	\begin{equation}\label{eq_subordinationsystemab}
		\Phi_{AB}(\Omega_A(z), \Omega_B(z), z)=0,
\end{equation}}
where $\Phi_A(z)$ and $\Phi_B(z)$ are defined as follows 
\beq\label{eq:def_PhiAB}
\Phi_{A}(\omega_{1},\omega_{2},z)\deq \frac{M_{\mu_{A}}(\omega_{2})}{\omega_{2}}-\frac{\omega_{1}}{z},\AND \Phi_{B}(\omega_{1},\omega_{2},z)\deq \frac{M_{\mu_{B}}(\omega_{1})}{\omega_{1}}-\frac{\omega_{2}}{z}.
\eeq
We further introduce the following shorthand notations
\begin{align}\label{eq_phacdefinition}
	&\Phi_{A}^{c}\equiv \Phi_{A}^{c}(z)\deq \Phi_{A}(\Omega_{A}^{c}(z),\Omega_{B}^{c}(z),z), &
	&\Phi_{B}^{c}\equiv \Phi_{B}^{c}(z)\deq \Phi_{B}(\Omega_{A}^{c}(z),\Omega_{B}^{c}(z),z).
\end{align}
Whenever there is no ambiguity, we will omit the dependence on $z$ and consider $(\Phi_{A},\Phi_{B})$ as a function of the first two variables $(\omega_{1},\omega_{2})$.  Recall that $\mathcal{S}_{AB}$, $\mathcal{T}_{A}$ and $\mathcal{T}_{B}$ are defined analogously as in (\ref{eq_defn_salphabeta}) and (\ref{eq_defn_talpha}) by replacing the pair $(\alpha, \beta)$ with $(A,B).$ We first introduce an estimate regarding the linear combination of $\mathcal{S}_{AB}, \mathcal{T}_{A}, \mathcal{T}_B$ and $\Lambda.$ It serves as a fundamental input for  
the continuity argument in Section \ref{sec_finalsection}. Before stating the result,  we observe that, under Assumption \ref{assu_ansz}, by (\ref{eq:Lambda}) and Proposition \ref{prop:entrysubor}
\begin{equation}\label{eq_lambdainitialbound}
	\Lambda \prec N^{-\gamma/4}.
\end{equation}

\begin{prop}\label{prop:FA2}
	Fix $z\in\caD_{\tau}(\eta_{L},\eta_{U}).$ Suppose the assumptions of Proposition \ref{prop:entrysubor} hold. 	
	Let $\wh{\Lambda}(z)$ be a deterministic positive function such that $\Lambda(z)\prec\wh{\Lambda}(z)\prec N^{-\gamma/4}$. Then we have for $\iota=A,B,$ 
	\beq\label{eq:FA2}
	\Absv{\frac{\caS_{AB}(z)}{z}\Lambda_{\iota}+\caT_{\iota}\Lambda_{\iota}^{2}+O(\absv{\Lambda_{\iota}}^{3})}\prec \mho.
	\eeq
	where $\mho$ is defined as 
	\begin{equation} \label{eq_defnTheta}
		\mho \equiv \mho(z):= \Psi^{2}\left(\sqrt{(\im m_{\mu_{A}\boxtimes\mu_{B}}(z)+\wh{\Lambda}(z))(\absv{\caS_{AB}(z)}+\wh{\Lambda}(z))}+\Psi^{2}\right).	
	\end{equation}
\end{prop}
The proof of Proposition \ref{prop:FA2} relies on the estimate for a special linear combinations of $Q_i$'s.    Recall (\ref{eq_mtrasindenity}). Denote
\begin{align} \label{eq_frz1z2definition}
	&	\frZ_{1}\deq \Phi_{A}^{c}+zL_{\mu_{A}}'(\Omega_{B})\Phi_{B}^{c}, \
	\frZ_{2}\deq \Phi_{B}^{c}+zL_{\mu_{B}}'(\Omega_{A})\Phi_{A}^{c}, \\
	&	\frZ_{1}^{(p,q)}\deq\frZ_{1}^{p}\ol{\frZ}_{1}^{q}, \
	\frZ_{2}^{(p,q)}\deq\frZ_{2}^{p}\ol{\frZ}_{2}^{q}. \nonumber
\end{align}
We collect the recursive moment estimates for $\frZ_{1}$ and $\frZ_{2}$ in the following lemma. Its proof will be provided after we finish proving Proposition \ref{prop:FA2}.  
\begin{lem}\label{lem:Zrecmoment}
	Fix $z\in\caD_{\tau}(\eta_{L},\eta_{U}).$ Suppose the assumptions of Proposition \ref{prop:FA2} hold. Then we have 
	\beq\label{eq_Zrecmomentequation}
	\expct{\frZ_{1}^{(p,p)}}=\expct{\rO_{\prec}(\mho)\frZ_{1}^{(p-1,p)}}+\expct{\rO_{\prec}(\mho^{2})\frZ_{1}^{(p-2,p)}}+\expct{\rO_{\prec}(\mho^{2})\frZ_{1}^{(p-1,p-1)}},
	\eeq
	where $\mho$ is defined in (\ref{eq_defnTheta}). Similar results hold for $\frZ_{2}$.
\end{lem}

Armed with Lemma \ref{lem:Zrecmoment}, we are ready to prove Proposition \ref{prop:FA2}. { Its proof is analogous to that of \cite[Proposition 7.1]{BEC} and we only present the key points here.}
\begin{proof}[\bf Proof of Proposition \ref{prop:FA2}] Due to similarity, we will focus on the discussion for $\iota=A.$ First of all, we provide some identities to connect the left-hand side of (\ref{eq:FA2}) and $\frZ_1$ and $\frZ_2.$ 
	
	{Inspired by equation (7.18) of \cite{BEC}}, we first expand $\Phi_{A}^{c}$ and $\Phi_{B}^{c}$ around $(\Omega_{A},\Omega_{B})$ using (\ref{eq_lambdainitialbound}) and (\ref{eq_subordinationsystemab}), and then combine the two expansions to get
	\beq\label{eq:ZtoFA2}
	\frac{\caS_{AB}}{z}\Lambda_{A}+\caT_{A}\Lambda_{A}^{2}+O(\absv{\Lambda_{A}}^{3})=\frZ_{1}+\rO_{\prec}(\absv{\Phi_{B}^{c}}^{2}+\absv{\Phi_{B}^{c}\Lambda_{A}}).
	\eeq
	By a discussion similar to (\ref{eq_arbitraydiscussion}), using Young's and Markov's inequalities, together with Lemma \ref{lem:Zrecmoment}, we have that 
	\begin{equation}\label{eq_frz1frz2control}
		\frZ_1 \prec \mho, \ \frZ_2 \prec \mho. 
	\end{equation}
	
	In what follows, we will apply Proposition \ref{prop:FA1} to prove that 
	\begin{equation}\label{eq_phibc}
		\Phi_\iota^c \prec \Psi \mho^{1/2}, \ \iota=A,B. 
	\end{equation}
	{We point out that once (\ref{eq_phibc}) is proved, together with \eqref{eq_frz1frz2control} and the fact $\Psi^{2}\wh{\Lambda}^{2}\leq \mho$, we can readily obtain \eqref{eq:FA2}.}
	Recall (\ref{eq_defnq}) and denote 
	\beqs
	\caQ_{i}\deq \tr(\caG\wt{A}B)\caG_{ii}-\tr(B\caG)(\caG\wt{A})_{ii}.
	\eeqs
	Note that $\Phi_A^c$ and $\Phi_B^c$ are linear combinations of $Q_i$'s and $\mathcal{Q}_i$'s, respectively. More specifically, we have
	\begin{align} \label{eq_optimalfaquantities}
		\Phi_{A}^{c}=\frac{1}{N}\sum_{i}\frd_{i1}Q_{i}, \ \
		\Phi_{B}^{c}=\frac{1}{N}\sum_{i}\frd_{i2}\caQ_{i},
	\end{align}
	where we denote
	\begin{align}
		\frd_{i1}\deq z\frac{1-M_{\mu_{A}}(\Omega_{B}^{c})}{(zm_{H}(z)+1)^{2}}\frac{a_{i}\tr(\wt{B}G)-\tr(A\wt{B}G)}{(a_{i}-\Omega_{B}^{c})}, \label{eq_optimalfaquantitiescoeff} \\
		\frd_{i2}\deq z\frac{1-M_{\mu_{B}}(\Omega_{A}^{c})}{(zm_{H}(z)+1)^{2}} \frac{b_{i}\tr(\wt{A}\caG)-\tr(B\wt{A}\caG)}{b_{i}-\Omega_{A}^{c}}. \label{eq_optimalfaquantitiescoeffextra}
	\end{align}
	These identities simply follow from algebraic computations using \eqref{eq_mtrasindenity}, \eqref{eq:approx_subor}, and the definitions of $\Phi_{A}^{c},\Phi_{B}^{c}$.
	
	To apply Proposition \ref{prop:FA1}, we need to check whether its assumptions are satisfied. Recall (\ref{eq_defnTheta}). First, we show that we can choose $\widehat{\Pi}=\mho^{1/2}.$ {Similar to the discussion of equation (7.5) of \cite{BEC},}
	we can use (\ref{eq_mtrasindenity}) and (\ref{eq_suborsystem}) to get 
	\beq\label{eq_rigidityuse}
	\absv{m_{H}(z)-m_{\mu_{A}\boxtimes\mu_{B}}(z)}\prec \Lambda(z)+\Upsilon\prec\wh{\Lambda}+\Psi^{2}.
	\eeq
	{We mention that compared to the counterpart for the additive model in equation (7.5) of \cite{BEC}, we have an extra term $\Upsilon$ in (\ref{eq_rigidityuse}) and the bound $\Upsilon\prec\Psi^{2}$ follows from Proposition \ref{prop:FA1}.}
	
	Combining \eqref{eq_rigidityuse} with the fact that $\im m_{\mu_{A}\boxtimes\mu_{B}}(z)\prec\sqrt{\kappa+\eta}\prec\absv{\caS_{AB}(z)}$, we obtain 
	\beq\label{eq_propfa1condition1}
	\Pi^{2}=\frac{\im m_H(z)}{N\eta}
	\prec\mho^{1/2},
	\eeq
	Moreover, by (ii) and (iii) of Proposition \ref{prop:stabN}, both $\im m_{\mu_{A}\boxtimes\mu_{B}}(z)$ and $\absv{\caS_{AB}(z)}$ are bounded. Consequently, we have that 
	\begin{equation}\label{eq_propfa1condition2}
		N^{-1/2}\eta^{-1/4}\prec \mho^{1/2} \prec \Psi. 
	\end{equation}
	By (\ref{eq_propfa1condition1}) and (\ref{eq_propfa1condition2}), we have seen that we can choose $\widehat{\Pi}=\mho^{1/2}.$  
	
	Second, we show that the coefficients $\{\frd_{i1}\}$ and $\{\frd_{i2}\}$ satisfy Assumption \ref{assum_conditiond}. {Similar to the discussion below equation (7.13) of \cite{BEC}}, we find that $\frd_{j1}$'s can be regarded as smooth functions of $\tr{(\wt{B}G)}$ and $\tr{G}$. Then using the chain rule and Lemma \ref{lem:recmomerror}, it is easy to see that $\frd_{j1}$ satisfies Assumption \ref{assum_conditiond}. Similar results hold for $\frd_{j2}.$ Therefore we find that both $\Phi_A^c$ and $\Phi_B^c$ satisfy the conditions of Proposition \ref{prop:FA1} with $\widehat{\Pi}=\mho^{1/2}.$ Consequently, Proposition \ref{prop:FA1} implies (\ref{eq_phibc}) and the proof of Proposition \ref{prop:FA2}.	
\end{proof}

We next prove Lemma \ref{lem:Zrecmoment}. Before stepping into the proof, we observe from the definitions of $\frZ_{1}$ and $\frZ_{2}$ in (\ref{eq_frz1z2definition}) and (\ref{eq_phibc}) that 
\begin{equation*}
	\frZ_{1} \prec\Psi\mho^{1/2}, \ \frZ_{2}\prec\Psi\mho^{1/2}.
\end{equation*}
{Since the proof of Lemma \ref{lem:Zrecmoment} is analogous to that of \cite[Lemma 7.3]{BEC} for the additive model, we only highlight the main differences. }
\begin{proof}[\bf Proof of Lemma \ref{lem:Zrecmoment}]  We only focus our proof on $\frZ_1,$ and $\frZ_2$ can be handled similarly. 
	Repeating the proof of Lemma \ref{lem:Avrecmoment} with weights $d_{i}=\frd_{i1}$, we find that the result follows from the following bounds
	\begin{align}\label{eq:Zerror}
		&\frac{1}{N^{2}}\sum_{i}\sum_{k}^{(i)}c_{i}\bse_{k}\adj X_{i}G^{\ell}\bse_{i} \frac{\partial \frZ_{1}}{\partial g_{ik}}\prec\mho^{2}, &
		&\frac{1}{N^{2}}\sum_{i}\sum_{k}^{(i)}c_{i}\bse_{k}\adj X_{i}G^{\ell}\mr{\bsg}_{i}\frac{\partial \frZ_{1}}{\partial g_{ik}}\prec\mho^{2},
	\end{align}
	where $c_{i}$'s stand for generic $\rO_{\prec}(1)$ factors, $X_{i}$ can be $I$ or $\wt{B}^{\angi}$, and $\ell$ can be $0$ or $1$. We remark that the quantities on the LHS of \eqref{eq:Zerror} stand for the coefficients of $\frZ_{1}^{(p-2,p)}.$ {Moreover, we point out that (\ref{eq:Zerror}) is  the counterpart of \cite[Lemma 7.4]{BEC}.}
	
	By chain rule, we have
	\begin{align}\label{eq_extraddterm}
		\frac{\partial \frZ_{1}}{\partial g_{ik}}&=\frac{\partial}{\partial g_{ik}}(\Phi_{A}(\Omega_{A}^{c},\Omega_{B}^{c})+zL_{\mu_{A}}'(\Omega_{B})\Phi_{B}(\Omega_{A}^{c},\Omega_{B}^{c}))\nonumber\\
		&	=\left(zL'_{\mu_{A}}(\Omega_{B})L'_{\mu_{B}}(\Omega_{A}^{c})-\frac{1}{z}\right)\frac{\partial \Omega_{A}^{c}}{\partial g_{ik}}+(L'_{\mu_{A}}(\Omega_{B}^{c})-L'_{\mu_{A}}(\Omega_{B}))\frac{\partial \Omega_{B}^{c}}{\partial g_{ik}}.
	\end{align}
	{By Proposition \ref{prop:stabN} and Assumption \ref{assu_ansz}, we find that
		\begin{align*}
			\Absv{zL'_{\mu_{A}}(\Omega_{B})L'_{\mu_{B}}(\Omega_{A}^{c})-\frac{1}{z}}\prec \absv{\caS_{AB}}+\Lambda, \  \
			\absv{L'_{\mu_{A}}(\Omega_{B})-L'_{\mu_{A}}(\Omega_{B}^{c})} \prec \Lambda.
		\end{align*}
		Since they are independent of the indices $i$ and $k$, we can simply pull them out as scaling factors from (\ref{eq_extraddterm}). As a consequence, in light of (\ref{eq_optimalfaquantities}), the remaining weighted sum has the same form as in the second or third estimates in Lemma \ref{lem:recmomerror}, which is $\rO_{\prec}(\Pi^{2}\Psi^{2})$.} Since
	\beqs
	(\absv{\caS_{AB}}+\Lambda)\Pi^{2}\Psi^{2} \prec \mho^{2} \AND 
	\Lambda\Psi^{2}\Pi^{2}\prec \mho^{2},
	\eeqs
	we conclude the proof of \eqref{eq:Zerror}, and therefore that of Lemma \ref{lem:Zrecmoment}.
\end{proof}

\section{Proof of Theorem \ref{thm:main} }\label{sec_finalsection}
We first prove weak local laws in Section \ref{subsec:weaklocallaw} and then prove Theorem \ref{thm:main} in Section \ref{sec_proofofstronglocallaw}. 

\subsection{Weak local laws}\label{subsec:weaklocallaw}
In Section \ref{sec_entrylaw}, we proved estimates for the entries of the resvolents and optimal fluctuation averaging for the linear combinations of them. We also proved the closeness of the subordination functions and their approximates. However, all these results are regarding pointwise control for fixed $z \in \mathcal{D}_{\tau}(\eta_L, \eta_U)$ and under Assumption \ref{assu_ansz}. In this subsection, we will establish a weak local law without imposing Assumption \ref{assu_ansz} and uniformly in $z \in \mathcal{D}_{\tau}(\eta_L, \eta_U),$ using a continuous bootstrapping argument. The weak local law will guarantee that Assumption \ref{assu_ansz} holds uniformly for $z \in \mathcal{D}_{\tau}(\eta_L, \eta_U).$ As a consequence, the results in Section \ref{sec_entrylaw} also hold uniformly for $z \in \mathcal{D}_{\tau}(\eta_L, \eta_U).$  

More specifically, the main result of this subsection is stated in the following proposition. 
\begin{prop}\label{prop:weaklaw} Suppose that Assumptions \ref{assu_limit} and \ref{assu_esd} hold. Let $\tau>0$ be a sufficiently small constant and $\gamma>0$ be any fixed small constant. Then we have{
		\begin{align}\label{eq_weaklaw}
			\Lambda_{d}(z)\prec \frac{1}{(N\eta)^{1/3}}, \
			\Lambda(z)&\prec\frac{1}{(N\eta)^{1/3}}, \
			\Lambda_{T}(z)\prec\Psi(z), 
	\end{align}}
	uniformly in $z\in\caD_{\tau}(\eta_{L},\eta_{U})$.
	The same statements hold  for 
	$\wt{\Lambda}_{d}$ and $\wt{\Lambda}_{T}$.
\end{prop}

The proof of Proposition \ref{prop:weaklaw} will be divided into three steps. In the first step, we prove  (\ref{eq_weaklaw}) on the global scale such that $\eta \geq \eta_U,$ where $\eta_U$ is a sufficiently large constant. The key idea for proving this step is to regard $Q_i$ defined in (\ref{eq_defnq}) as a function of the random unitary matrix $U.$ This strategy has been employed in the study of addition of random matrices in \cite[Theorem 8.1]{BEC} and local single ring theorem in \cite{bao2019}. An advantage of doing so is that we can employ the device of Gromov-Milman concentration inequality. We collect the related results in the following lemma. 

Denote $U(N)$ as the set of $N \times N$ unitary matrices over $\mathbb{C}$ and $SU(N) \subseteq U(N)$ as the set of special unitary matrices defined by
\begin{equation*}
	SU(N):=\left\{ U \in U(N): \det(U)=1 \right\}.
\end{equation*}   
Recall that both $U(N)$ and $SU(N)$ are compact Lie groups with respect to the matrix multiplication. 
\begin{lem}\label{lem:gromovmilman}
	Let $f$ be a real-valued Lipschitz continuous function defined on $U(N).$ We denote its Lipschitz constant $\mathcal{L}_f$ as 
	\beqs
	\caL_{f}\deq \sup\left\{\frac{\Absv{f(U_1)-f(U_2)}}{\norm{U_1-U_2}_{2}}:U_1,U_2\in U(N)\right\},\quad \norm{U_1-U_2}_{2}\deq \sqrt{\Tr((U_1-U_2)\adj (U_1-U_2))}.
	\eeqs
	Moreover, we let $\nu$ be the Haar measure defined on $U(N)$ and $\nu_s$ be that on $SU(N).$  Then for some universal constants $c,C>0,$ we have that for any constant $\delta>0,$
	\beq\label{eq:Gromov-Milman}
	\int_{U(N)}\lone\left(\Absv{f(U)-\int_{SU(N)}f(VU)\dd\nu_{s}(V)}>\delta\right)\dd\nu (U)
	\leq C\mathrm{exp}\left(-c\frac{N\delta^{2}}{\caL_{f}^{2}}\right).
	\eeq
	Finally, let $H_{N}$ be the group of the form $\{\diag(\e{\ii\theta},1,\cdots,1):\theta\in[0,2\pi]\}$ and $\nu_{h}$ be its associated Haar measure in the sense that $\theta$ is uniformly distributed on $[0,2\pi]$, then we have 
	\beq\label{eq:integ_SU}
	\int_{SU(N)}\int_{H_{N}} f(VU_1W)\dd\nu_{h}(W)\dd\nu_{s}(V)=\int_{U(N)}f(U)\dd\nu(U), \quad \text{for all} \ U_1 \in U(N).
	\eeq 
	
\end{lem}
\begin{proof}
	See Corollary 4.4.28 and Lemma 4.4.29 of \cite{Anderson-Guionnet-Zeitouni2010}. 
\end{proof}

In the second step, we employ a continuity argument based on the estimates obtained from step one {for larger $\eta$} to prove that the results hold for each fixed $z \in \mathcal{D}_\tau(\eta_L, \eta_U).$
To this end, for $z\in\caD_{\tau}(\eta_{L},\eta_{U})$ and $\delta,\delta'\in[0,1],$ we define events
\begin{align}
	& \Theta(z,\delta,\delta')\deq \left\{\Lambda_{d}(z)\leq \delta, \wt{\Lambda}_{d}(z)\leq \delta, \Lambda(z)\leq \delta, \Lambda_{T}(z)\leq \delta',\wt{\Lambda}_{T}(z)\leq\delta'\right\}, \label{eq_setdefnTheta}\\
	&\Theta_{>}(z,\delta,\delta',\epsilon')\deq \Theta(z,\delta,\delta')\cap\left\{\Lambda(z)\leq N^{-\epsilon'}\absv{\caS_{AB}(z)}\right\}. \label{eq_setdefnThetag}
\end{align}
Moreover, we decompose the domain $\caD_{\tau}(\eta_{L},\eta_{U})$ into two disjoint sets,
\begin{align*}
	\caD_{>}\equiv\caD_{>}(\tau,\eta_{L},\eta_{U},\epsilon)&\deq\left\{z\in\caD_{\tau}(\eta_{L},\eta_{U}):\sqrt{\kappa+\eta}>\frac{N^{2\epsilon}}{(N\eta)^{1/3}}\right\}, \\
	\caD_{\leq}\equiv\caD_{\leq}(\tau,\eta_{L},\eta_{U},\epsilon)&\deq \caD_{\tau}(\eta_{L},\eta_{U})\setminus\caD_{>}.
\end{align*}
{The main technical input for the second step is Lemma \ref{lem:iteration_weaklaw} below.} As it will be seen from Lemma \ref{lem:iteration_weaklaw}, when we restrict ourselves on some high probability event, it is always possible to gradually improve our estimates.
\begin{lem}\label{lem:iteration_weaklaw}
	Suppose that Assumptions \ref{assu_limit} and \ref{assu_esd} hold. For any fixed $z\in\caD_{\tau}(\eta_{L},\eta_{U})$, any $\epsilon\in(0,\gamma/12)$ and $D>0$, there exists $N_{1}(D,\epsilon)\in\N$ and an event $\Xi(z,D,\epsilon)$ with 
	\beq\label{eq_highprobabilityevent}
	\P(\Xi(z,D,\epsilon))\geq 1-N^{-D}, \quad \text{for all} \ N\geq N_{1}(D,\epsilon),
	\eeq
	such that the followings hold:
	\begin{itemize}
		\item[(i)] For all $z\in\caD_{>}$,
		\beqs
		\Theta_{>}\left(z,\frac{N^{3\epsilon}}{(N\eta)^{1/3}},\frac{N^{3\epsilon}}{\sqrt{N\eta}},\frac{\epsilon}{10}\right)\cap\Xi(z,D,\epsilon)\subset \Theta_{>}\left(z,\frac{N^{5\epsilon/2}}{(N\eta)^{1/3}},\frac{N^{5\epsilon/2}}{\sqrt{N\eta}},\frac{\epsilon}{2}\right).
		\eeqs
		\item[(ii)] For all $z\in\caD_{\leq}$,
		\beqs
		\Theta\left(z,\frac{N^{3\epsilon}}{(N\eta)^{1/3}},\frac{N^{3\epsilon}}{\sqrt{N\eta}}\right)\cap\Xi(z,D,\epsilon)\subset \Theta\left(z,\frac{N^{5\epsilon/2}}{(N\eta)^{1/3}},\frac{N^{5\epsilon/2}}{\sqrt{N\eta}}\right).
		\eeqs
	\end{itemize}
\end{lem}
\begin{proof}
	
	{The proof of Lemma \ref{lem:iteration_weaklaw} is analogous to that of \cite[Lemma 8.3]{BEC} by establishing the quantitative versions of Propositions \ref{prop:entrysubor}, \ref{prop:FA1}, \ref{prop:FA2} via cutoffs. We only state the result of the cutoff version of Proposition \ref{prop:FA2} in the following lemma without providing extra details due to similarity. We remark that applying Lemma \ref{lem:priori_weaklaw} with $k=1/3$ yields Lemma \ref{lem:iteration_weaklaw} and with $k=1$ will serve as a technical input for the proof of Theorem \ref{thm:main}.}
	\begin{lem}\label{lem:priori_weaklaw}
		Let $z\in\caD_{\tau}(\eta_{L},\eta_{U}),\epsilon\in(0,\gamma/12)$, and $k\in(0,1]$ be fixed values. Let $\wh{\Lambda}(z)$ be some deterministic positive control parameter such that $\wh{\Lambda}(z)\leq N^{-\gamma/4}$. Suppose that $\Lambda(z)\leq\wh{\Lambda}(z)$ and
		\beq\label{eq:self_improv_assu}
		\Absv{\frac{\caS_{AB}}{z}\Lambda_{\iota}+\caT_{\iota}\Lambda_{\iota}^{2}+O(\Lambda_{\iota})^{3}}
		\leq N^{2\epsilon/5}\frac{\absv{\caS_{AB}}+\wh{\Lambda}}{(N\eta)^{k}},\quad \iota=A,B,
		\eeq
		hold on some event $\wt{\Xi}(z)$. Then there exists some constant $C>0$ such that for sufficiently large $N,$ the following hold:
		\begin{itemize}
			\item[(i)] If $\sqrt{\kappa+\eta}>N^{-\epsilon}\wh{\Lambda}$, there is a constant $K_{0}>0$ independent of $z$ and $N$ such that
			\beq\label{eq_lemmac1c1}
			\mathbb{I}_{\wt{\Xi}(z)}\mathbb{I}\left(\Lambda\leq \frac{\absv{\caS_{AB}}}{K_{0}}\right)\absv{\Lambda_{\iota}}\leq C\left( N^{-2\epsilon}\wh{\Lambda}+\frac{N^{7\epsilon/5}}{(N\eta)^{k}}\right).
			\eeq
			\item[(ii)] If $\sqrt{\kappa+\eta}\leq N^{-\epsilon}\wh{\Lambda}$, then we have that 
			\beq\label{eq_lemmac1c2}
			\mathbb{I}_{\wt{\Xi}(z)}\absv{\Lambda_{\iota}}\leq C\left(N^{-\epsilon}\wh{\Lambda}+\frac{N^{7\epsilon/5}}{(N\eta)^{k}}\right), \ \iota=A,B.
			\eeq
		\end{itemize}
	\end{lem}
	
\end{proof}
In the third step, we prove the uniformity. With the aid of Lemma \ref{lem:iteration_weaklaw}, we first extend the results to the whole domain from a discrete lattice of mesh size $N^{-5}.$ Then by  
using the Lipschitz continuity of subordination functions and the resolvents, we can extend the bounds to the entire domain $\mathcal{D}_{\tau}(\eta_L, \eta_U).$ In the rest of this subsection, we  prove the weak local law using the above three-step strategy. 
\begin{proof}[\bf Proof of Proposition \ref{prop:weaklaw}]
	
	\vspace{3pt}
	\noindent{\bf Step 1.} The goal of this step is to prove the results for $\eta \geq \eta_U.$ A main technical input is the following estimate
	\beq\label{eq:Qi_global}
	\absv{zQ_{i}}\prec \frac{1}{\sqrt{N}\eta_{U}},
	\eeq
	for all fixed $z$ with $\im z\geq \eta_{U}$. We mention that (\ref{eq:Qi_global}) is slightly weaker than its counterpart, equation (8.35) of \cite{BEC}. Since $\eta_U$ is large, (\ref{eq:Qi_global}) is already sufficient for our discussion. 
	
	To show (\ref{eq:Qi_global}), we employ (\ref{eq:Gromov-Milman}) for a properly chosen function. Specifically, for $U \in U(N),$ we let
	\begin{equation}\label{eq_defnfingeneral}
		f(U)=zQ_i(U),
	\end{equation}
	where $Q_i$ is defined in (\ref{eq_defnq}) and we regard $Q_i$ as a function of $U.$ To show $f(\cdot)$ is Lipschitz for large $\eta_U$, we directly calculate its Lipschitz constant in the spirit of \cite[Section 8.1]{bao2019} with a slightly different justification. {Indeed, since $f$ is differentiable it suffices to bound its derivative in order to obtain the Lipschitz constant. By a discussion similar to \cite[Section 8.1]{bao2019}, we can obtain 
		\begin{equation}\label{eq_lipconstantbound}
			\caL_f \leq \frac{C}{\eta_U}.
		\end{equation}
		On the other hand, we also have
		\begin{equation} \label{eq_controlofmean}
			\int_{SU(N)}f(VU)\dd\nu_{s}(V)=\int f(U)\dd\nu (U)=z\expct{Q_{i}}=0.
		\end{equation}
		Here we used (\ref{eq:integ_SU}) in the first equality and Proposition 3.2 and equation (3.25) of \cite{Vasilchuk2001} in the last inequality. Therefore, using (\ref{eq:Gromov-Milman}), we have proved the claim (\ref{eq:Qi_global}) for each fixed $z$. After a standard application of lattice continuity argument, the results can be extended to $\{z:\im z\geq \eta_{U}\}$.}
	
	Then we prove the estimate
	\beq\label{eq:Lambda_c_macro}
	\sup_i	\sup_{\substack{\im z>\eta_{U}}}\Absv{\Lambda_{di}^c}\prec N^{-1/2}.
	\eeq
	In view of (\ref{eq:Qi_global}) and \eqref{eq:Lambda}, it suffices to show that 
	\begin{equation}\label{eq_reduceddenominator}
		(1+z\tr{G})(a_i-\Omega_B^c)\succ 1. 
	\end{equation}
	{Similar to the discussion of equation (8.41) of \cite{BEC}, \eqref{eq_reduceddenominator} follows from the following two asymptotics
		\begin{align}
			1+z\tr G&=\frac{1}{z}+\rO(\absv{z}^{-2})+\rO_{\prec}((\absv{z}N)^{-1}), \label{eqeqqqqeqqqqeqqq}\\
			\frac{\Omega_{B}^{c}}{z}&=1+\rO(\absv{z}^{-1})+\rO_{\prec}(N^{-1}). \label{eq:approx_subor_macro}
		\end{align}
		The proof of these asymptotics is similar to equation (8.39) of \cite{BEC} except that we need to apply \eqref{eq:Gromov-Milman} to get $\tr(A\wt{B})=1+O_{\prec}(N^{-1})$. 
	}
	
	Next, we show that Assumption \ref{assu_ansz} holds when $\im z=\eta_{U}$. Recall the definitions of $\Phi_{A}^{c}$ and $\Phi_{B}^{c}$ in (\ref{eq_phacdefinition}). By (\ref{eq:approx_subor_macro}), we get
	\beq\label{eq_phacboundbound1}
	\Absv{\Phi_{A}^c}\leq\Absv{\frac{M_{\mu_{A}}(\Omega_{B}^{c})-M_{H}(z)}{\Omega_{B}^{c}}}+\Absv{\frac{zM_{H}(z)-\Omega_{A}^{c}\Omega_{B}^{c}}{z\Omega_{B}^{c}}}\prec\absv{z} N^{-1/2},
	\eeq
	where we used \eqref{eq_averageqdefinitionupsilon}, \eqref{eq:approx_subor}, and \eqref{eq:Qi_global}.
	
	On the other hand,  by Lemma \ref{lem:reprMemp}, $\absv{\Phi_{A}^{c}}$ and $\absv{\Phi_{B}^{c}}$ are both $\rO(\absv{z}^{-1})$. Together with (\ref{eq:approx_subor_macro}) and (\ref{eq_phacboundbound1}), by Lemma \ref{lem:Kantorovich_appl}, we readily find that uniformly in $z\in\caD_{\tau}(\eta_{U},\infty)$
	\begin{align}\label{eq_differencesubordinationdifference}
		\absv{\Lambda_{A}}+\absv{\Lambda_{B}}\leq 4\absv{\Omega_{A}^{c}-\Omega_{A}(z)}\leq 2(\absv{\Phi_{A}^{c}(z)}+\absv{\Phi_{B}^{c}(z)})\prec \absv{z}N^{-1/2}.
	\end{align}
	Thus by \eqref{eq:Lambda_c_macro}, (\ref{eq_differencesubordinationdifference}) and (i) of Proposition \ref{prop:stabN}, we conclude that 
	\beqs
	\sup_i	\Lambda_{di}(E+\ii\eta_{U})=\sup_i \Absv{zG_{ii}-\frac{a_{i}}{a_{i}-\Omega_{B}}}\prec N^{-1/2}.
	\eeqs
	On the other hand, using the trivial bound $\|G(E+\ii \eta_U) \| \leq \eta_U^{-1},$ we find that 
	\begin{equation*}
		\Lambda_T(E+\ii \eta_U) \leq \eta_U^{-1}. 
	\end{equation*}
	
	Based on the above discussion, since $\eta_U$ is a sufficiently large constant,  we see that Assumption \ref{assu_ansz} holds for $z=E+\ii \eta_U.$ Then by Proposition \ref{prop:entrysubor}, we have that for $z=E+\ii \eta_U$
	\begin{align*}
		\Lambda_{T}\prec N^{-1/2}, \ \wt{\Lambda}_{T}&\prec N^{-1/2}, \ \Upsilon \prec N^{-1/2}.
	\end{align*}
	Moreover, by (iii) of Proposition \ref{prop:stabN} and (\ref{eq_differencesubordinationdifference}), we have $\Lambda_{A}(z)\prec N^{-\epsilon}\sqrt{\kappa+\eta_{U}}$ for $z=E+\ii\eta_{U}.$  Quantitively, for any fixed $E \in \mathbb{R},$
	\beq\label{eq:WLL_global}
	\P\left[\Theta_{>}\left(E+\ii\eta_{U},\frac{N^{3\epsilon}}{(N\eta_{U})^{1/3}},\frac{N^{3\epsilon}}{(N\eta_{U})^{1/2}},\frac{\epsilon}{10}\right)\right]\geq 1-N^{-D},
	\eeq
	for all $D>0$ and $N\geq N_{2}(\epsilon,D),$ where $N_{2}(\epsilon,D)$ depends only on $\epsilon$ and $D$.
	
	\vspace{3pt}	
	\noindent{\bf Step 2.} In this step, with the estimate (\ref{eq:WLL_global}) and Lemma  \ref{lem:iteration_weaklaw}, we control the probability of the "good" events $\Theta_{>}$ for $z \in \mathcal{D}_{>}$ and $\Theta$ for $z \in \mathcal{D}_{\leq}.$ Consequently, we can iteratively make $\im z$ smaller.  More specifically, in this step, we prove that for the high probability event $\Xi(\cdot, \cdot, \cdot)$ in (\ref{eq_highprobabilityevent}), the following statements hold: 
	\begin{itemize}
		\item[(i)] For all $z\in\caD_{>}$,
		\beq\label{eq:WLL_iter_event1}
		\Theta_{>}\left(z,\frac{N^{5\epsilon/2}}{(N\eta)^{1/3}},\frac{N^{5\epsilon/2}}{\sqrt{N\eta}},\frac{\epsilon}{2}\right)\cap\Xi(z-N^{-5}\ii,D,\epsilon)
		\subset \Theta_{>}\left(z-N^{-5}\ii,\frac{N^{5\epsilon/2}}{(N\eta)^{1/3}},\frac{N^{5\epsilon/2}}{\sqrt{N\eta}},\frac{\epsilon}{2}\right);
		\eeq
		\item[(ii)] For all $z\in\caD_{\leq}$,
		\beq\label{eq:WLL_iter_event2}
		\Theta\left(z,\frac{N^{5\epsilon/2}}{(N\eta)^{1/3}},\frac{N^{5\epsilon/2}}{\sqrt{N\eta}}\right)\cap\Xi(z-N^{-5}\ii,D,\epsilon)\subset \Theta\left(z-N^{-5}\ii,\frac{N^{5\epsilon/2}}{(N\eta)^{1/3}},\frac{N^{5\epsilon/2}}{\sqrt{N\eta}}\right).
		\eeq
	\end{itemize}
	{The motivation to decompose the domain $\mathcal{D}_{\tau}(\eta_L,\eta_U)$ into $\mathcal{D}_{>}$ and $\mathcal{D}_{\leq}$ is that,} in $\mathcal{D}_{>}$ we need to keep track of the event $\Lambda \leq N^{-\epsilon/2} \mathcal{S}_{AB}$ in order to apply (i) of Lemma \ref{lem:iteration_weaklaw}; see equation (\ref{eq_keeptrackofeventhighprobability}) for more details. 
	
	We start with discussing the event $\Theta.$ First, for generic values $z,\delta$ and $\delta',$ we claim that on the event $\Theta(z,\delta,\delta')$ defined in (\ref{eq_highprobabilityevent}), there exists some constant $C>0,$
	\begin{align}\label{eq_claimoneone}
		\Lambda_{d}(z+w)\leq \delta+CN^{-3}, \ \Lambda_{T}(z+w) \leq \delta'+CN^{-3},
	\end{align}
	for all $w\in\C$ with $\absv{w}\leq 2N^{-5},$ and $z+w\in\caD_{\tau}(\eta_{L},\eta_{U})$. Indeed, the above estimates follow from mean value theorem and 
	\begin{align*}
		\norm{\frac{\dd}{\dd z}(zG+1)}\leq \eta^{-2}\leq N^{2}, \ \absv{\Omega_{A}'(z)}\leq \frac{C}{\sqrt{\kappa+\eta}}\leq N^{1/2},\
		\absv{\Omega_{B}'(z)}\leq \frac{C}{\sqrt{\kappa+\eta}}\leq N^{1/2},
	\end{align*}
	where we used (iv) of Proposition \ref{prop:stabN}. Second, we claim the following result
	\beq\label{eq_claimnlambda}
	\Lambda(z+w)\leq \delta'+CN^{-3}.
	\eeq
	We now prove it. Recall Definition \ref{defn_asf}. Using (\ref{eq:apxsubor}) and the fact $\frac{\dd G}{\dd z}=G^2,$ we find that 
	\beqs
	\frac{\dd}{\dd z}\Omega_{A}^{c}(z)=\frac{\tr(\wt{B}A\wt{B}G^{2})\tr(A\wt{B}G)-\tr(\wt{B}A\wt{B}G)\tr(A\wt{B}G^{2})}{\tr(A\wt{B}G)^{2}}.
	\eeqs
	For some constant $C>0,$ by (\ref{eq:Lambda}), (i) of Proposition \ref{prop:stabN} and the definition of $\Theta(z, \delta, \delta'),$ we have 
	\beqs
	\absv{\tr(\wt{B}A\wt{B}G)}=\absv{\tr(\wt{B}(zG+1))}\leq1+\Absv{\frac{C z}{N}\sum_{i}\frac{1}{a_{i}-\Omega_{B}}}+C\absv{z}\delta\leq C, 
	\eeqs 	
	and 
	\beqs
	\absv{\tr(A\wt{B}G)}\geq \absv{zm_{\mu_{A}\boxtimes\mu_{B}}(z)+1}-\delta.
	\eeqs	
	Combining the above discussions, when $\delta \ll 1,$ we obtain that 
	\begin{equation*}
		\left| \frac{\dd }{\dd z} \Omega_A^c(z)\right| \leq C \eta^2 \leq CN^2. 
	\end{equation*}
	Similar results hold for $\Omega_B^c.$ Then we can complete the proof of (\ref{eq_claimnlambda}) with mean value theorem. Armed with (\ref{eq_claimoneone}) and (\ref{eq_claimnlambda}), we immediately see that 	
	\begin{align}
		\Theta\left(z,\frac{N^{5\epsilon/2}}{(N\eta)^{1/3}},\frac{N^{5\epsilon/2}}{\sqrt{N\eta}}\right)
		&\subset\bigcap_{\absv{w}\leq N^{-5}}\Theta\left(z+w,\frac{N^{5\epsilon/2}}{(N\eta)^{1/3}}+CN^{-3},\frac{N^{5\epsilon/2}}{{\sqrt{N\eta}}}+CN^{-3}\right) \nonumber\\
		&\subset\bigcap_{\absv{w}\leq N^{-5}}\Theta\left(z+w,\frac{N^{3\epsilon}}{(N\eta)^{1/3}},\frac{N^{3\epsilon}}{\sqrt{N\eta}}\right). \label{eq:WLL_lipschitz1}
	\end{align}
	
	We next briefly discuss the event $\Theta_{>}$ due to similarity.  On the event $\Theta_{>}\left(z,\frac{N^{5\epsilon/2}}{(N\eta)^{1/3}},\frac{N^{5\epsilon/2}}{\sqrt{N\eta}},\frac{\epsilon}{10}\right)$, by a discussion similar to (\ref{eq_claimoneone}) and (iv) of Proposition \ref{prop:stabN}, we have that for some constants $c, C>0,$ 
	\begin{align}\label{eq_keeptrackofeventhighprobability}
		\Lambda(z-N^{-5}\ii)&\leq \Lambda(z)+CN^{-3}\leq N^{-\epsilon/2}\absv{\caS_{AB}(z)}+CN^{-3}
		\leq CN^{-\epsilon/2}(\kappa+\eta)^{-1/2}+CN^{-3}	\nonumber \\&\leq cN^{-\epsilon/10}(\kappa+\eta-N^{-5})^{-1/2} 		\leq N^{-\epsilon/10}\absv{\caS_{AB}(z-N^{-5}\ii)}.
	\end{align}
	Consequently, we have that 
	\begin{align}
		\Theta_{>}\left(z,\frac{N^{5\epsilon/2}}{(N\eta)^{1/3}},\frac{N^{5\epsilon/2}}{\sqrt{N\eta}},\frac{\epsilon}{2}\right)	&\subset\bigcap_{\absv{w}\leq N^{-5}}\Theta_{>}\left(z+w,\frac{N^{3\epsilon}}{(N\eta)^{1/3}},\frac{N^{3\epsilon}}{\sqrt{N\eta}},\frac{\epsilon}{10}\right).\label{eq:WLL_lipschitz2}
	\end{align}
	By Lemma \ref{lem:iteration_weaklaw} and taking 	$w=-N^{-5}\ii$ in (\ref{eq:WLL_lipschitz1}) and (\ref{eq:WLL_lipschitz2}), we have proved  \eqref{eq:WLL_iter_event1} and \eqref{eq:WLL_iter_event2}.
	
	\vspace{3pt}	
	\noindent{\bf Step 3.} In this step, we use a continuity argument to extend the bound first to a lattice of mesh size $N^{-5}$ and then the whole domain $\mathcal{D}_{\tau}(\eta_L, \eta_U)$. To simplify the notations, we denote 
	\begin{align*}
		& \caP\deq\caD_{\tau}(\eta_{L},\eta_{U})\cap (N^{-5}\Z)^{2}, \ \ \caP_{>}\deq \caD_{>}\cap (N^{-5}\Z)^{2}, \\ 
		& \caP_{\leq}\deq \caD_{\leq}\cap(N^{-5}\Z)^{2}, \ \ \caP_{E}\deq [E_{+}-\tau,\tau^{-1}]\cap N^{-5}\Z.
	\end{align*}
	Repeatedly applying \eqref{eq:WLL_iter_event1} and \eqref{eq:WLL_iter_event2}, we have
	\begin{align}\label{eq:WLL_conc}
		&\bigcap_{z\in\caP}\Xi(z,D,\epsilon)\cap \bigcap_{E\in\caP_{E}}\Theta_{>}\left(E+\ii\eta_{U},\frac{N^{5\epsilon/2}}{(N\eta_{U})^{1/3}},\frac{N^{5\epsilon/2}}{(N\eta_{U})^{1/2}},\frac{\epsilon}{2}\right)\\
		& \subset\bigcap_{z\in\caP_{>}}\Theta_{>}\left(E+\ii\eta,\frac{N^{5\epsilon/2}}{(N\eta)^{1/3}},\frac{N^{5\epsilon/2}}{(N\eta)^{1/2}},\frac{\epsilon}{2}\right)\cap \bigcap_{z\in\caP_{\leq}}\Theta\left(E+\ii\eta,\frac{N^{5\epsilon/2}}{(N\eta)^{1/3}},\frac{N^{5\epsilon/2}}{(N\eta)^{1/2}}\right) \nonumber \\
		& \subset \bigcap_{z\in\caD_{>}}\Theta_{>}\left(E+\ii\eta,\frac{N^{3\epsilon}}{(N\eta)^{1/3}},\frac{N^{3\epsilon}}{(N\eta)^{1/2}},\frac{\epsilon}{10}\right)\cap \bigcap_{z\in\caD_{\leq}}\Theta\left(E+\ii\eta,\frac{N^{3\epsilon}}{(N\eta)^{1/3}},\frac{N^{3\epsilon}}{(N\eta)^{1/2}}\right),	\nonumber	
	\end{align}
	where in the third step we used \eqref{eq:WLL_lipschitz1} and \eqref{eq:WLL_lipschitz2}. Moreover, by Lemma \ref{lem:iteration_weaklaw} and \eqref{eq:WLL_global}, when $N\geq \max(N_{1}(D,\epsilon),N_{2}(D,\epsilon)),$ the probability of the first event of (\ref{eq:WLL_conc}) is at least  
	\beqs
	1-\sum_{E\in\caP_{E}}N^{-D}-\sum_{z\in\caP}N^{-D}\geq 1- CN^{10-D}.
	\eeqs
	Since $D$ is arbitrary, we can prove for the discrete lattice by choosing $D>10$ large enough. Finally, by the Lipschitz continuity of the resolvent and the subordination functions in (iv) of Proposition \ref{prop:stabN}, we can extend all the bounds from the discrete lattice to $\mathcal{D}_{\tau}(\eta_L, \eta_U).$ This concludes the proof of Proposition \ref{prop:weaklaw}. 
	
\end{proof}

Before concluding this section, we briefly discuss how we handle the off-diagonal entries of the resolvents. considering the decomposition of $(\wt{B}G)_{ij}$ using the strategy similar to (\ref{eq:BGii-S}). In what follows, we introduce the counterparts of the diagonal entries. Corresponding to (\ref{eq:Lambda}), we denote
\beq\label{eq_qij}
Q_{ij}\deq\tr(GA)(\wt{B}G)_{ij}-G_{ij}\tr(A\wt{B}G).
\eeq
Moreover, using the fact $G_{ij}=a_{i}(\wt{B}G)_{ij}/z,$ which follows from $zG+I=A\wt{B}G,$
we see that
\beq\label{eq_qijdecomposution}
Q_{ij}=\left(\frac{z}{a_{i}}\tr(GA)-\tr(A\wt{B}G)\right)G_{ij}=\frac{\tr (GA\wt{B})}{a_{i}}(\Omega_{B}^{c}-a_{i})G_{ij}.
\eeq
In this sense, it suffices to bound $Q_{ij}$ in order to prove \eqref{eq:mainoff2}. Similar to the discussion for the diagonal entries, we will need the following quantities 
\begin{align}\label{eq_quantitydefinitionoffdiagonal}
	S_{ij}\deq\bsh_{i}\adj\wt{B}^{\angi}G\bse_{j}, \
	\mr{S}_{ij}\deq\mr{\bsh}_{i}\adj\wt{B}^{\angi}G\bse_{j}, \quad \quad \quad  \quad \quad \quad \quad \quad \quad \quad \quad \quad \quad \quad \quad \quad \\
	T_{ij}\deq \bsh_{i}\adj G\bse_{j},\
	\mr{T}_{ij}\deq \mr{\bsh}_{i}\adj G\bse_{j},\quad \quad \quad \quad \quad \quad  \quad \quad \quad \quad \quad \quad  \quad \quad \quad \quad \quad \quad \quad \quad  \nonumber\\
	P_{ij}\deq Q_{ij}+(G_{ij}+T_{ij})\Upsilon, \
	K_{ij}=T_{ij}+\tr(GA)(b_{i}T_{ij}+(\wt{B}G)_{ij})-\tr(GA\wt{B})(G_{ij}+T_{ij}). \nonumber
\end{align}
Based on the above discussion and the estimates in Lemma \ref{lem:recmomerror_off}, we can prove the results for the off-diagonal entries. For example, denote 
\begin{equation*}
	\Lambda_o:=\max_{i \neq j} |G_{ij}|, \ \Lambda_{T,o}:=\max_{i \neq j} |T_{ij}|.
\end{equation*}
By a discussion similar to Proposition \ref{prop:weaklaw},  we can obtain that uniformly in $z \in \mathcal{D}_\tau(\eta_L, \eta_U)$
\begin{equation*}
	\Lambda_o \prec \frac{1}{(N \eta)^{1/3}}, \  \Lambda_{T, o} \prec \Psi(z). 
\end{equation*}
{This implies that an analog of Assumption \ref{assu_ansz} for off-diagonal entries holds true, so that the analogous proof as Proposition \ref{prop:entrysubor} yields the desired bounds for off-diagonal entries.} We omit the details here.

\subsection{Strong local law: proof of Theorem \ref{thm:main}}\label{sec_proofofstronglocallaw}
In this subsection, using the weak local law Proposition \ref{prop:weaklaw} as an initial input, we prove Theorem \ref{thm:main}. Indeed, the proof follows from the same three-step strategy as mentioned in the proof of Proposition \ref{prop:weaklaw}, except that we use better estimates.

{We first prepare some technical ingredients. The first ingredient is Lemma \ref{lem:priori_weaklaw} above with the choice $k=1$. Now that we have Proposition \ref{prop:weaklaw}, we see that the assumptions of Proposition \ref{prop:entrysubor} hold true and so does the conclusion of Proposition \ref{prop:FA2}. This in turn implies \eqref{eq:self_improv_assu} with $k=1$, so that we can simply take the results of Lemma \ref{lem:priori_weaklaw} as granted. We will use this result to gradually improve the estimate of $\Lambda(z)$ with fixed $\im z.$  }

The second input, Lemma \ref{lem:self_improv_S}, is an analogue of Lemma \ref{lem:iteration_weaklaw}. It serves as the key input to gradually extend the bounds to the entire domain $\mathcal{D}_\tau(\eta_L, \eta_U).$ To state the result rigorously, we define the domains $\wt{\caD}_{>},\wt{\caD}_{\leq}$ and events $\wt{\Theta},\wt{\Theta}_{>}$ as follows:
\begin{align*}
	&\wt{\caD}_{>}\deq \left\{z\in\caD_{\tau}(\eta_{L},\eta_{U}):\sqrt{\kappa+\eta}>\frac{N^{2\epsilon}}{N\eta}\right\},&
	&\wt{\caD}_{\leq}\deq \left\{z\in\caD_{\tau}(\eta_{L},\eta_{U}):\sqrt{\kappa+\eta}>\frac{N^{2\epsilon}}{N\eta}\right\}, \\
	&\wt{\Theta}(z,\delta)\deq \left\{\Lambda\leq \delta \right\},&
	&\wt{\Theta}_{>}(z,\delta,\epsilon')\deq \wt{\Theta}(z,\delta)\cap \left\{\Lambda(z)\leq N^{-\epsilon'}\absv{\caS_{AB}}\right\}.
\end{align*}

{The proof is identical to Lemma \ref{lem:iteration_weaklaw} except that we take $k=1$ in Lemma \ref{lem:priori_weaklaw}, instead of $k=1/3$. We omit the details here.}

\begin{lem}\label{lem:self_improv_S}
	For any fixed $\epsilon_{0}\in(0,\gamma/12)$ and large $D>0$, there exists $N_0 \equiv N_{0}(\epsilon_{0},D)$ such that the followings hold for all $N\geq N_{0}$ and $z\in\caD_{\tau}(\eta_{L},\eta_{U})$: 
	\begin{align}
		&\prob{\wt{\Theta}_{>}\left(z,\frac{N^{3\epsilon_{0}}}{N\eta},\frac{\epsilon_{0}}{10}\right)\setminus\wt{\Theta}_{>}\left(z,\frac{N^{5\epsilon_{0}/2}}{N\eta},\frac{\epsilon_{0}}{2}\right)}\leq N^{-D}, \ \ \text{ for }z\in\wt{\caD}_{>},\\
		&\prob{\wt{\Theta}\left(z,\frac{N^{3\epsilon_{0}}}{N\eta}\right)\setminus\wt{\Theta}\left(z,\frac{N^{5\epsilon_{0}/2}}{N\eta}\right)}\leq N^{-D}, \ \ \text{ for }z\in\wt{\caD}_{\leq} .
	\end{align} 
\end{lem}

Then we prove Theorem \ref{thm:main} using Lemmas \ref{lem:priori_weaklaw} and \ref{lem:self_improv_S}. As above, the proof is analogous to that of \cite[Theorem 2.5]{BEC}.
\begin{proof}[\bf Proof of Theorem \ref{thm:main}]
	First, we prove that 
	\begin{equation}\label{lambdaz_uniform}
		\Lambda \prec \frac{1}{N \eta},
	\end{equation}
	holds uniformly in $z \in \mathcal{D}_\tau(\eta_L, \eta_U).$ As we mentioned earlier, the proof follows the same three-step proof strategy of Proposition \ref{prop:weaklaw}. We focus on explaining step 1, which differs the most from its counterpart of the proof of Proposition \ref{prop:weaklaw}. As in (\ref{eq:WLL_global}), step 1 will prove that for fixed $\im z=\eta_U,$ 
	\beq\label{eq_limitstrongglobal}
	\inf_{E\in[E_{+}-\tau,\tau^{-1}]}\prob{\wt{\Theta}_{>}\left(E+\ii\eta_{U},\frac{N^{3\epsilon_{0}}}{N\eta_U},\frac{\epsilon}{10}\right)}\geq 1-N^{-D}.
	\eeq
	To see (\ref{eq_limitstrongglobal}), on one hand, we notice that by Propositions \ref{prop:weaklaw} the conclusion of \ref{prop:FA2} hold true without any assumption. Thus we can take $k=1$ and initially choose $\wh{\Lambda}(z)=N^{3\epsilon_{0}}(N\eta_U)^{-1/3}$ in Lemma \ref{lem:priori_weaklaw}; on the other hand, we will repeatedly apply Lemma \ref{lem:priori_weaklaw} to gradually improve the control parameter $\wh{\Lambda}(z)$ until it reaches the bound $(N\eta_U)^{-1}$(up to a factor of $N^{\epsilon_{0}}$). To facilitate our discussion, we denote
	\begin{align*}
		\wh{\Lambda}_{1}(z)&\deq\frac{N^{3\epsilon_{0}}}{(N\eta_U)^{1/3}},& 
		\wh{\Lambda}_{k}(z)&\deq N^{(4-k)\epsilon_{0}}\frac{1}{(N\eta_U)^{1/3}}+\frac{N^{2\epsilon_{0}}}{N\eta_U},\quad\text{for } k\in\left\llbra 2,\left\lceil\frac{2}{3\epsilon_{0}}+4\right\rceil\right\rrbra.
	\end{align*}	
	Within the above preparation, we will prove (\ref{eq_limitstrongglobal}) by induction in the sense that for each $k$, there exists an $N_k \equiv N_{k}(\epsilon_{0},D)$ so that 
	\beq\label{eq:SLL_global_indct}
	\prob{\Lambda(E+\ii\eta_{U})\leq\wh{\Lambda}_{k}(E+\ii\eta_{U})}\geq 1-kN^{-D},
	\eeq
	for all $N\geq N_{k}$ and uniformly in $E\in[E_{+}-\tau,\tau^{-1}]$. Clearly, we have that $(N\eta)^{-1}\leq \wh{\Lambda}_{k}(z)\leq N^{-\gamma/4}$ for all $k$. Therefore, Lemma \ref{lem:priori_weaklaw} is applicable with all the choices $\wh{\Lambda}=\wh{\Lambda}_{k}$. Since $\sqrt{\kappa+\eta_{U}}>N^{-\epsilon_{0}}\wh{\Lambda}_{k}$, by Lemma \ref{lem:priori_weaklaw}, we have
	\beq\label{eq_withindicator}
	\prob{\left[\Lambda(E+\ii\eta_{U})\leq \wh{\Lambda}_{k}(E+\ii\eta_{U})\right]\cap\left[\lone\left(\Lambda\leq \frac{\absv{\caS_{AB}}}{K_{0}}\right)\absv{\Lambda_{\iota}}>N^{-2\epsilon_{0}}\wh{\Lambda}+\frac{N^{7\epsilon_{0}/5}}{N\eta_U}\right]}\leq N^{-D},
	\eeq
	for all $N\geq N_{k}' \equiv N_k'(\epsilon_{0},D)$ and uniformly in $E\in[E_{+}-\tau,\tau^{-1}]$. By (iii) of Proposition \ref{prop:stabN}, we find that for some constant $c>0,$
	\beqs
	\frac{\absv{\caS_{AB}}}{K_{0}}\geq c\sqrt{\kappa+\eta_{U}}\geq \wh{\Lambda}_{k}\geq \Lambda,
	\eeqs
	on the event $[\Lambda\leq \wh{\Lambda}_{k}]$. Therefore, we can safely remove the factor of indicator function in (\ref{eq_withindicator}). Consequently, \eqref{eq:SLL_global_indct} leads to
	\begin{align*}
		&\prob{\absv{\Lambda_{\iota}}> N^{-2\epsilon_{0}}\wh{\Lambda}_{k}+\frac{N^{7\epsilon_{0}/5}}{N\eta_U}}\\
		&\leq \prob{\Lambda> \wh{\Lambda}_{k}}+
		\prob{\left[\Lambda\leq \wh{\Lambda}_{k}\right]\cap\left[\lone\left(\Lambda\leq \frac{\absv{\caS_{AB}}}{K_{0}}\right)\absv{\Lambda_{\iota}}>N^{-2\epsilon_{0}}\wh{\Lambda}+\frac{N^{7\epsilon_{0}/5}}{N\eta_U}\right]} \\
		&		\leq (k+1)N^{-D},
	\end{align*}
	whenever $N\geq \max\{N_{k}',N_{k} \}$ and $z=E+\ii\eta_{U}$, uniformly in $E\in[E_{+}-\tau,\tau^{-1}]$. Together with the fact
	\beqs
	N^{-2\epsilon_{0}}\wh{\Lambda}_{k}+\frac{N^{7\epsilon_{0}/5}}{N\eta_U} =N^{-\epsilon_{0}}\wh{\Lambda}_{k+1}+\frac{N^{7\epsilon_{0}/5}-N^{\epsilon_{0}}}{N\eta_U}\leq \wh{\Lambda}_{k},
	\eeqs
	we have proved that for\eqref{eq:SLL_global_indct}, the same inequality holds for the case $k+1$ if the case $k$ holds true. Recall that by Proposition \ref{prop:weaklaw} we have that \eqref{eq:SLL_global_indct} holds for $k=1$. By induction, \eqref{eq:SLL_global_indct} holds for all fixed $k$. In particular, taking $k=\lceil 2/(3\epsilon_{0})+4\rceil$, we have
	\beqs
	\wh{\Lambda}_{k}\leq \frac{N^{4\epsilon_{0}-k\epsilon_{0}+2/3}+N^{2\epsilon_{0}}}{N\eta_U} \leq \frac{N^{3\epsilon_{0}}}{N\eta_U}.
	\eeqs
	Since by (iii) of Proposition \ref{prop:stabN}, $\Lambda\leq\wh{\Lambda}_{k}$ implies $\Lambda\leq N^{-\epsilon_{0}/10}\absv{\caS_{AB}}$ for $\eta=\eta_{U}$, we have proved (\ref{eq_limitstrongglobal}).  Steps 2 and 3 of the proof of (\ref{lambdaz_uniform}) follows analogously as the counterparts of the proof of Proposition \ref{prop:weaklaw} using Lemma \ref{lem:self_improv_S} and (\ref{eq_limitstrongglobal}), we omit the details here. 
	
	Second, with (\ref{lambdaz_uniform}) and the local laws, we can see that the bounds in Theorem \ref{thm:main} hold. In detail, with the local law in Proposition \ref{prop:weaklaw}, we see that Assumption \ref{assu_ansz}  holds uniformly in $z \in \mathcal{D}_{\tau}(\eta_L, \eta_U).$ Consequently, a result analogous to Proposition \ref{prop:entrysubor} holds uniformly in $z \in \mathcal{D}_{\tau}(\eta_L, \eta_U).$ Then the bounds for the off-diagonal entries, i.e.  (\ref{eq_off1}) and (\ref{eq:mainoff2}), follow immediately. For the diagonal entries, we focus on explaining (\ref{eq:main}). Recall (\ref{eq:Lambda}). We see that for any deterministic $v_1, \cdots, v_N \in \mathbb{C},$ we can  write  
	\begin{equation*}
		\frac{1}{N} \sum_{i=1}^N v_i \left( zG_{ii}+1-\frac{a_i}{a_i-\Omega_B^c(z)} \right)=\frac{1}{N} \sum_{i=1}^N \frac{zv_i a_i}{(1+z\tr G)(a_i-\Omega_B^c)} Q_i.
	\end{equation*}	
	Regarding $\frac{zv_i a_i}{(1+z\tr G)(a_i-\Omega_B^c)}$ as the random coefficients $d_i$ in (\ref{eq_weighted}), it is easy to see that the conditions in (\ref{eq:weight_cond}) hold. Hence, by Proposition \ref{prop:FA1} and the weak local law, we find that 
	\begin{equation*}
		\left|\frac{1}{N} \sum_{i=1}^N v_i \left( zG_{ii}+1-\frac{a_i}{a_i-\Omega_B^c(z)} \right) \right| \prec \Psi \wh{\Pi}.
	\end{equation*}
	Together with (\ref{lambdaz_uniform}), we conclude the proof.  
	
\end{proof}

\section{Proof of Theorems  \ref{thm_rigidity} and \ref{thm_delocalization}} \label{sec_proofofspikedmodel}

\begin{proof}[\bf Proof of Theorem \ref{thm_rigidity}] 
	The first part of the results, i.e., the eigenvalues are close to the quantiles $\gamma_i^*,$ follows from a standard argument established in \cite{erdos2012}, by translating the closeness of the resolvent into the closeness of the eigenvalues and the quantitles of $\mu_{A} \boxtimes \mu_{B}.$ The proof is rather standard and we only discuss the two important inputs:
	\begin{enumerate}
		\item For the largest eigenvalue $\lambda_1,$ we have 
		\begin{equation}\label{eq_controllargestdifference}
			|\lambda_1-\gamma_1^*| \prec N^{-2/3}. 
		\end{equation}
		\item Given sufficiently small constant $\varsigma>0,$ we have 
		\begin{equation}\label{eq_ctronlclaimone}
			\sup_{x \leq E_++\varsigma}|\mu_{H}((-\infty,x])-\mu_{A} \boxtimes \mu_B((-\infty,x])|\prec N^{-1}.
		\end{equation} 
	\end{enumerate}
	First, (\ref{eq_ctronlclaimone}) follows from an argument regardless of the underlying random matrix models once the local laws are established.  In detail, it follows from a standard application of the Helffer-Sj{\" o}strand formula on $\mathcal{D}_{\tau}(\eta_L, \eta_U)$; see \cite{erdos2012} and \cite[Section 10]{BKnotes} for details. {Second, the proof of  (\ref{eq_controllargestdifference}) basically follows from its counterpart of the additive model; see equation (10.6) of \cite{BEC}.} It relies on an improved estimate for $\Lambda(z)$, i.e., Lemma \ref{lem_improvedestimate} below, {which is an analog of \cite[Lemma 10.1]{BEC} of the additive model.} We omit the proof of Lemma \ref{lem_improvedestimate} which can be easily checked by a calculation similar to Lemma 10.1 of \cite{BEC}.
	Denote $\wt{\mathcal{D}}$ as
	\begin{equation}\label{eq_wtmathcald}
		\wt{\mathcal{D}}_{>}:=\{z=E+\ii \eta \in \mathcal{D}_{\tau}(\eta_L, \eta_U): \sqrt{\kappa+\eta}>\frac{N^{2 \epsilon}}{N \eta} \ \text{and} \ E>E_+\}.
	\end{equation}
	
	\begin{lem}\label{lem_improvedestimate} Under the assumptions of Theorem \ref{thm_rigidity}, we have the following holds uniformly in $z \in \wt{\mathcal{D}}_{>}$
		\begin{equation}\label{eq_improveboundequation}
			\Lambda(z) \prec \frac{1}{N \sqrt{(\kappa+\eta)\eta}}+\frac{1}{\sqrt{\kappa+\eta}} \frac{1}{(N \eta)^2}. 
		\end{equation}
	\end{lem}
	Recall (\ref{eq:priorisupp}). For any (small) constant $\epsilon>0,$ we define the following line segment
	\begin{equation*}
		\wt{\mathcal{D}}(\epsilon):=\{z=E+\ii \eta: E_++N^{-2/3+6\epsilon} \leq E <\tau^{-1}, \ \eta=N^{-2/3+\epsilon} \},
	\end{equation*}
	where $\tau>0$ is some small fixed constant. Clearly, we have that $\wt{\mathcal{D}}(\epsilon) \subset \wt{\mathcal{D}}_{>}.$ By (\ref{eq_improveboundequation}), we readily obtain that $\Lambda \prec \frac{N^{-\epsilon}}{N \eta}$ holds uniformly in $z \in \wt{\mathcal{D}}(\epsilon).$ Together with (\ref{eq_rigidityuse}), we see that for $z \in \wt{\mathcal{D}}(\epsilon)$ uniformly, $|m_H(z)-m_{\mu_A \boxtimes \mu_B}(z)| \prec \frac{N^{-\epsilon}}{N \eta}.$ 
	Together with (ii) of Proposition \ref{prop:stabN}, we arrive at 
	\begin{equation}\label{eq_contradiction}
		\im m_H(z) \prec \frac{N^{-\epsilon}}{N \eta},
	\end{equation}
	holds uniformly in $z \in \wt{\mathcal{D}}(\epsilon).$ By (\ref{eq:priorisupp}), we see that the proof of (\ref{eq_controllargestdifference}) follows from the claim that  with high probability $\lambda_1 \leq E_++N^{-2/3+6\epsilon},$ {which can be justified by contradiction using a discussion similar to its counterpart as in the proof of Theorem 2.6 of \cite{BEC}. We omit the detail here. }

	For the second part of the results, once the first part is proved,  it suffices to bound the differences between $\gamma_i^*$ and $\gamma_i$, i.e., Lemma \ref{lem:rigidity_AB}, {whose proof follows closely from its counterpart of the additive model as in \cite[Lemma 3.14]{BEC} and we omit the details here.} 
	
	\begin{lem} \label{lem:rigidity_AB} Suppose the assumptions of Theorem \ref{thm_rigidity} hold. Then for some sufficiently large $N_0 \equiv N_0(\epsilon,c),$ when $N \geq N_0,$ we have 
		\begin{equation}\label{eq_closeABalphabeta}
			|\gamma_i-\gamma_i^*| \leq j^{-1/3} N^{-2/3+\epsilon}, \ 1 \leq j \leq cN. 
		\end{equation}
	\end{lem}
\end{proof}

\begin{proof}[\bf Proof of Theorem \ref{thm_delocalization}] We focus our discussion on the left singular vectors. By spectral decomposition, we have that 
	\begin{equation}\label{spectral_singularvector}
		\sum_{k=1}^N \frac{\eta |\mathbf{u}_k(i)|^2}{(\lambda_k-E)^2+\eta^2}=\im \wt{G}_{ii}(z), \ z=E+\ii \eta. 
	\end{equation} 
	For $1 \leq k_0 \leq cN,$ let $z_0=\lambda_{k_0}+\ii \eta_0,$ where $\eta_0=N^{-1+\epsilon_0}, \epsilon_0>\gamma$ is a small constant. Inserting $z=z_0$ into (\ref{spectral_singularvector}), by Theorem \ref{thm:main} and (i) of Proposition \ref{prop:stabN}, we immediately see that 
	\begin{equation*}
		|\ub_k(i)|^2 \prec \eta_0 \lesssim N^{-1}.  
	\end{equation*}
	
\end{proof}

\section{Some auxiliary lemmas}\label{appendix_deriavtive}
\subsection{Collection of derivatives}
In this subsection, for the reader's convenience,  we collect some results involving derivatives. They can be easily checked by elementary calculation. 
\begin{lem}\label{lem_derivative} Recall the notations in (\ref{defn_greenfunctions}), (\ref{eq_prd}), (\ref{eq_prd2}), (\ref{eq_prd1}), (\ref{eq_pk}) and (\ref{eq_defnupsilon}). For $i\neq k$, we have the  following identities;
	\beq \label{eq_householdderivative}
	\frac{\partial R_{i}}{\partial g_{ik}}
	=-\frac{\partial (\ell_{i}^{2})}{\partial g_{ik}}(\bse_{i}+\bsh_{i})(\bse_{i}+\bsh_{i})\adj 
	-\ell_{i}^{2}\left(\frac{\partial \bsh_{i}}{\partial g_{ik}}\bse_{i}\adj 
	+\bse_{i}\frac{\partial \bsh_{i}\adj}{\partial g_{ik}}
	+\frac{\partial \bsh_{i}}{\partial g_{ik}}\bsh_{i}\adj 
	+\bsh_{i}\frac{\partial \bsh_{i}\adj}{\partial g_{ik}}\right),
	\eeq
	\beq\label{eq_hiderivative}
	\frac{\partial \bsh_{i}}{\partial g_{ik}}
	=\norm{\bsg_{i}}^{-1}\bse_{k}	-\frac{1}{2}\norm{\bsg_{i}}^{-3}\ol{g}_{ik}\bsg_{i}=\norm{\bsg_{i}}^{-1}\bse_{k}-\frac{1}{2}\norm{\bsg_{i}}^{-1}\ol{h}_{ik}\bsh_{i},
	\eeq
	{
		\beq\label{eq_histartderivative}
		\frac{\partial \bsh_{i}\adj}{\partial g_{ik}}=-\frac{1}{2}\norm{\bsg_{i}}^{-3}\bar{g}_{ik}\bsg_{i}\adj	=-\frac{1}{2}\norm{\bsg_{i}}^{-1}\ol{h}_{ik}\bsh_{i}\adj,
		\eeq
	}
	\begin{align}\label{eq_partialli2}
		\frac{\partial \ell_{i}^{2}}{\partial g_{ik}}
		=2\frac{\partial\norm{\bse_{i}+\bsh_{i}}^{-2}}{\partial g_{ik}} =\ell_{i}^{4}\norm{\bsg_{i}}^{-3}\ol{g}_{ik}g_{ii}=\ell_{i}^{4}\norm{\bsg_{i}}^{-1}\ol{h}_{ik}h_{ii},
	\end{align}
	\begin{align}
		\frac{\partial P_{i}}{\partial g_{ik}}
		=&\bse_{i}\adj\frac{\partial G}{\partial g_{ik}}\bse_{i}\tr(A\wt{B}G)
		+zG_{ii}\frac{\tr G}{\partial g_{ik}}-\tr(\frac{\partial G}{\partial g_{ik}}A)(\wt{B}G)_{ii}
		-\tr(GA)\bse_{i}\adj A^{-1}\frac{\partial G}{\partial g_{ik}}\bse_{i} \nonumber	\\
		+&(\bse_{i}\adj \frac{\partial G}{\partial g_{ik}}\bse_{i}+\frac{\partial T_{i}}{\partial g_{ik}})\Upsilon
		+(G_{ii}+T_{i})\frac{\partial \Upsilon}{\partial g_{ik}}, \label{eq_partialpi}\\
		\frac{\partial \Upsilon}{\partial g_{ik}}
		=&z\left(\tr(2zG+1)\tr\frac{\partial G}{\partial g_{ik}}-\tr\left(\frac{\partial G}{\partial g_{ik}}A\right)\tr(\wt{B}G)-z\tr(GA)\tr(A^{-1}\frac{\partial G}{\partial g_{ik}})\right), \nonumber	\\
		\frac{\partial K_{i}}{\partial g_{ik}}
		=&\frac{\partial T_{i}}{\partial g_{ik}}
		+\tr(\frac{\partial G}{\partial g_{ik}}A)(b_{i}T_{i}+(\wt{B}G)_{ii})
		+\tr(GA)\left(b_{i}\frac{\partial T_{i}}{\partial g_{ik}}+\frac{z}{a_{i}}\bse_{i}\frac{\partial G}{\partial g_{ik}}\bse_{i}\right)	\nonumber \\
		-&z\tr\frac{\partial G}{\partial g_{ik}}(G_{ii}+T_{i})
		-\tr(GA\wt{B})\left(\bse_{i}\frac{\partial G}{\partial g_{ik}}\bse_{i}+\frac{\partial T_{i}}{\partial g_{ik}}\right). \label{eq_partialKderivative}
	\end{align}
\end{lem}

\subsection{Large deviation inequalities}\label{sec_largedeviation}
In this subsection, we provide some large deviation type controls. 

\begin{lem}[Large deviation estimates]\label{lem:LDE}
	Let $X$ be an $N \times N$ complex-valued deterministic matrix and let $\bsy\in\C^{N}$ be a deterministic complex vector. For a real or complex Gaussian random vector $\bsg\in\C^{N}$ with covariance matrix $\sigma^{2}I_{N}$, we have
	\beq\label{eq_largedeviationbound}
	\absv{\bsy\adj\bsg}\prec\sigma \norm{\bsy},\quad \absv{\bsg\adj X\bsg-\sigma^{2}N\tr X}\prec\sigma^{2}\norm{X}.
	\eeq
	
\end{lem}
\begin{proof}
	See Lemma A.1 of \cite{BEC}. 
\end{proof}
{In what follows, we prove several lemmas regarding the estimates of some crucial quantities which will be used in our proof. The analogs for the additive model are Lemmas 5.3, 6.3 and 7.4 of \cite{BEC}. }
\begin{lem}\label{lem:DeltaG}
	Let $X_{i}$ be $\wt{B}^{\angi}$ or $I$, $X_{A}$ be $A$ or $I$ or $A^{-1}$, and $D=(d_{i})$ be a random diagonal matrix with $\norm{D}\prec 1.$  Recall (\ref{eq_defndeltag}) and (\ref{eq_controlparameter}). Under the assumption of (\ref{eq_locallaweqbound}), for each fixed $z \in \mathcal{D}_{\tau}(\eta_L, \eta_U),$ we have
	\begin{gather}
		\frac{1}{N}\sum_{k}^{(i)}\bse_{k}\adj X_{i}\Delta_{G}(i,k)\bse_{i}\prec\Pi_{i}^{2}, \quad
		\frac{1}{N}\sum_{k}^{(i)}\bse_{k}\adj X_{i}G\bse_{i}\tr(DX_{A}\Delta_{G}(i,k))\prec\Pi_{i}^{2}\Psi^{2},  \label{eq:DeltaG1}\\
		\frac{1}{N}\sum_{k}^{(i)}\bse_{k}\adj X_{i} \mr{\bsg_{i}}\tr(DX_{A}\Delta_{G}(i,k))\prec\Pi_{i}^{2}\Psi^{2}, \  \label{eq:DeltaG2}\\
		\frac{1}{N}\sum_{k}^{(i)}\bse_{k}\adj X_{i}G\bse_{i}\bse_{i}\adj X_{A}\Delta_{G}(i,k)\bse_{i}\prec	\Pi_{i}^{2}, \quad  \frac{1}{N}\sum_{k}^{(i)}\bse_{k}\adj X_{i}G\bse_{i}\bsh_{i}\adj \Delta_{G}(i,k)\bse_{i}\prec\Pi_{i}^{2}. \label{eq:DeltaG3}
	\end{gather}
\end{lem}
\begin{proof}
	Recall the definitions in (\ref{eq_defndeltag}) and (\ref{eq_deltarg}).
	It is easy to see that every term in $\Delta_{G}(i,k)$ belongs to one of the following forms
	\begin{align}\label{eq_twotermschoice}
		&d_{i}\ol{h}_{ik} GA\bsalp_{i}\bsbet_{i}\adj\wt{B}^{\angi}R_{i}G & &\text{or} & 
		&d_{i}\ol{h}_{ik}GAR_{i}\wt{B}^{\angi}\bsalp_{i}\bsbet_{i}\adj G,
	\end{align}
	where $d_{i}$'s are $\rO_{\prec}(1)$, $k$-independent generic constants and $\bsalp_{i},\bsbet_{i}$ are either $\bse_{i}$ or $\bsh_{i}$. Due to similarity, we focus on the first inequality in (\ref{eq:DeltaG1}) and briefly explain the other terms.
	
	Due to similarity, we only prove first inequality in (\ref{eq:DeltaG1}). In view of (\ref{eq_twotermschoice}), we find that	every term in the first quantity in \eqref{eq:DeltaG1} has either one of the following two forms:
	\begin{align}\label{eq:Expan_DeltaG}
		&d_{i}\frac{1}{N}\sum_{k}^{(i)}\ol{h}_{ik}\bse_{k}\adj X_{i}GA\bsalp_{i}\bsbet_{i}\adj \wt{B}^{\angi}R_{i}G\bse_{i}
		=\frac{d_{i}}{N}(\mr{\bsh}_{i}\adj X_{i}GA\bsalp_{i})(\bsbet_{i}\adj\wt{B}^{\angi}R_{i}G\bse_{i}),\\ 
		&d_{i}\frac{1}{N}\sum_{k}^{(i)}\ol{h}_{ik}\bse_{k}\adj X_{i}GAR_{i}\wt{B}^{\angi}\bsalp_{i}\bsbet_{i}\adj G\bse_{i}
		=\frac{d_{i}}{N}(\mr{\bsh}_{i}\adj X_{i}GAR_{i}\wt{B}^{\angi}\bsalp_{i})(\bsbet_{i}\adj G\bse_{i}), \nonumber
	\end{align}
	where used the definition of $\mr{\bsh}_{i}$ in (\ref{eq_prd2}). Then we provide controls for terms on the right-hand side of the above two equations. Since $\| X_i \|$ is bounded and $\|\bm{h}_i \|_2^2=1,$ by (\ref{eq_hbwtr}) and (v) of Assumption \ref{assu_esd}, we find that for some constant $C>0,$  
	\begin{align}
		\absv{\mr{\bsh}_{i}\adj X_{i}GA\bsalp_{i}}
		\leq C(\bsalp_{i}\adj AG\adj GA\bsalp_{i})^{1/2}, \
		\absv{\bsbet_{i}\adj \wt{B}^{\angi}R_{i}G\bse_{i}}
		\leq C(\bse_{i}\adj G\adj G\bse_{i})^{1/2},&	\label{eq_termonecontrolone}\\
		\absv{\mr{\bsh}_{i}\adj X_{i}GAR_{i}\wt{B}^{\angi}\bsalp_{i}}
		\leq C(\bsalp_{i}\adj R_{i}\wt{B}AG\adj G A\wt{B}R_{i}\bsalp_{i})^{1/2}, \ \absv{\bsbet_{i}\adj G\bse_{i}}\leq C(\bse_{i}\adj G\adj G\bse_{i})^{1/2}. \label{eq_termtwocontroltwo}
	\end{align}
	It remains to study the bounds in (\ref{eq_termonecontrolone}) and (\ref{eq_termtwocontroltwo}). We first study the bounds of the first terms in (\ref{eq_termonecontrolone}) and (\ref{eq_termtwocontroltwo}). When $\bsalp_{i}=\bse_{i}$, we have that for some constant $C>0$
	\begin{align}\label{eq_transfernonhermitiantohermitian}
		\bsalp_{i}\adj AG\adj GA\bsalp_{i}
		=a_{i}\bse_{i}\adj \wt{G}\adj A\wt{G}\bse_{i} &\leq C\frac{\im G_{ii}}{\eta},	
	\end{align}
	where in the first step we used (\ref{eq_ggrelationship}) and in the second step 	we used (v) of Assumption \ref{assu_esd}, Ward identity and (\ref{eq_connectiongreenfunction}). Similarly,  we have that 
	\begin{align*}
		\bsalp_{i}\adj R_{i}\wt{B}AG\adj GA\wt{B}R_{i}\bsalp_{i}
		=\bsh_{i}\adj G\adj \wt{B}A^{2}\wt{B}G\bsh_{i}
		=\bse_{i}\adj B^{1/2}\wt{\caG}B^{1/2}\wt{A}^{2}B^{1/2}\wt{\caG}B^{1/2}\bse_{i}
		&\leq C\frac{\im \caG_{ii}}{\eta}.
	\end{align*}
	Analogously, when $\bsalp_{i}=\bsh_{i}$, we have that
	\begin{align*}
		\bsalp_{i}\adj AG\adj GA\bsalp_{i}=&\bse_{i}\wt{A}\caG\adj \caG\wt{A}\bse_{i}
		=\bse_{i}\adj\wt{A}B^{1/2}\wt{\caG}\adj B^{-1}\wt{\caG}B^{1/2}\wt{A}\bse_{i}
		=b_{i}^{-1}\bse_{i}\adj\wt{\caG}B\wt{A}^{2}B\wt{\caG}\bse_{i} \leq C\frac{\im\caG_{ii}}{\eta},\\
		&	\bsalp_{i}\adj R_{i}\wt{B}AG\adj GA\wt{B}R_{i}\bsalp_{i}	=\bse_{i}\adj G\adj G\adj \wt{B}A^{2}\wt{B}G\bse_{i}	\leq C\frac{\im G_{ii}}{\eta}.
	\end{align*}
	Moreover, for the second terms in (\ref{eq_termonecontrolone}) and (\ref{eq_termtwocontroltwo}), we readily see that
	$\bse_{i}\adj G\adj G\bse_{i}\leq C\frac{\im G_{ii}}{\eta}.$ Combining all the above bounds, we conclude that
	\beqs
	\frac{1}{N}\sum_{k}^{(i)}\bse_{k}\adj X_{i}\Delta_{G}(i,k)\bse_{i}\prec \frac{1}{N} \frac{\im G_{ii}+\im \caG_{ii}}{\eta} =\Pi_{i}^{2}.
	\eeqs
	
	For the remaining inequalities, we can replace $\Delta_{G}$ as a sum of quantities in \eqref{eq:Expan_DeltaG} and estimate every term of the summands. This concludes our proof. 
\end{proof}
\begin{lem}\label{lem:recmomerror}
	Let $X_{i}$ be $\wt{B}^{\angi}$ or $I$, $X_{A}$ be $A$ or $I$ or $A^{-1}$. Suppose that $D$ is a random diagonal matrix satisfying $\norm{D}\prec 1$. Then under the assumptions of Lemma \ref{lem:DeltaG}, we have
	\begin{align}
		&\frac{1}{N}\sum_{k}^{(i)}\frac{\partial \norm{\bsg_{i}}^{-1}}{\partial g_{ik}}\bse_{k}\adj X_{i}G\bse_{i}\prec\frac{1}{N}, &
		&\frac{1}{N}\sum_{k}^{(i)}\bse_{k}\adj X_{i} G\bse_{i}\tr(DX_{A}\frac{\partial G}{\partial g_{ik}})\prec \Pi_{i}^{2}\Psi^{2}, \label{eq:condiff1} \\
		&\frac{1}{N}\sum_{k}^{(i)}\bse_{k}\adj X_{i}\mr{\bsg_{i}}\tr(DX_{A}\frac{\partial G}{\partial g_{ik}})\prec \Pi_{i}^{2}\Psi^{2},&
		&\frac{1}{N}\sum_{k}^{(i)}\bse_{k}\adj X_{i}G\bse_{i}\bse_{i}\adj X_{A}\frac{\partial G}{\partial g_{ik}}\bse_{i}\prec\Pi_{i}^{2}, \label{eq:condiff2}  \\
		&\frac{1}{N}\sum_{k}^{(i)}\bse_{k}\adj X_{i}G\bse_{i}\frac{\partial T_{i}}{\partial g_{ik}}\prec \Pi_{i}^{2}. && \label{eq:condiff3} 
	\end{align}
\end{lem}
\begin{proof}
	The proof is similar to Lemma \ref{lem:DeltaG}. We only discuss the proof of the first inequality in (\ref{eq:condiff1}). Using the identity$\frac{\partial \norm{\bsg_{i}}^{-1}}{\partial g_{ik}}=-\ol{g}_{ik}/(2\norm{\bsg_{i}}^{3}),$
	(\ref{eq_prd2}) and (\ref{eq_shorhandnotation}), we find that 
	\begin{align*}
		\frac{1}{N}\sum_{k}^{(i)}\frac{\partial \norm{\bsg_{i}}^{-1}}{\partial g_{ik}}\bse_{k}\adj X_{i}G\bse_{i}
		&	=-\frac{1}{\norm{\bsg_{i}}^{3}}\frac{1}{2N}\sum_{k}^{(i)}\ol{g}_{ik}\bse_{k}\adj X_{i}G\bse_{i}
		=-\frac{1}{2N\norm{\bsg_{i}}^{2}}\mr{\bsh_{i}}\adj X_{i}G\bse_{i} \\
		& 		=\left\{
		\begin{array}{cc}
			-\frac{1}{2N}\mr{S}_{i}, & \ \text{if }X_{i}=\wt{B}^{\angi},\\
			-\frac{1}{2N}\mr{T}_{i}, & \ \text{if }X_{i}=I.
		\end{array}
		\right.
	\end{align*}
	It suffices to control $\mr{S}_{i}$ and $\mr{T}_{i}.$ By the definition of $\mr{T}_i$ in (\ref{eq_shorhandnotation}), $\mr{T}_{i}\prec 1$ follows from Assumption \ref{assu_ansz}.   Moreover, under Assumption \ref{assu_ansz}, by (\ref{eq_gbgindeti}), (\ref{eq_wtbgiipointwise}),  (\ref{eq:BGii}) and (\ref{eq_boundepsilon1}),  we have  
	\beqs
	\mr{S}_{i}=-(\wt{B}G)_{ii}+G_{ii}+T_{i}+\mathsf{e}_{i1}=\left(\frac{\Omega_{B}}{z}-1\right)\frac{1}{a_{i}-\Omega_{B}}+T_{i}+\rO_{\prec}(N^{-\gamma/4}) \prec 1,
	\eeqs
	where we used (i) of Proposition  \ref{prop:stabN}. This proves the first inequality in (\ref{eq:condiff1}). The other inequalities can be proved similarly.   
	
\end{proof}

Following the same proof strategy of Lemma \ref{lem:recmomerror}, we can prove the following result for the errors arising in the expansion of off-diagonal entries. We omit the details of the proof. 
\begin{lem}\label{lem:recmomerror_off}
	Let $X_{i}$ be $\wt{B}^{\angi}$ or $I$, $X_{A}$ be $A$ or $I$ or $A^{-1}$, and $D=(d_{i})$ be a random diagonal matrix with $\norm{D}\prec 1$. Fix $z \in \mathcal{D}_{\tau}(\eta_L, \eta_U)$ and $j\neq i$, and assume that \eqref{eq_locallaweqbound} holds. Then we have
	\begin{align*}
		&\frac{1}{N}\sum_{k}^{(i)}\frac{\partial \norm{\bsg_{i}}^{-1}}{\partial g_{ik}}\bse_{k}\adj X_{i}G\bse_{j}\prec\frac{1}{N}, &
		&\frac{1}{N}\sum_{k}^{(i)}\bse_{k}\adj X_{i} G\bse_{j}\tr(DX_{A}\frac{\partial G}{\partial g_{ik}})\prec (\Pi_{i}+\Pi_{j})^{2}\Psi^{2},  \\
		&\frac{1}{N}\sum_{k}^{(i)}\bse_{k}\adj X_{i}G\bse_{j}\bse_{i}\adj X_{A}\frac{\partial G}{\partial g_{ik}}\bse_{j}\prec(\Pi_{i}+\Pi_{j})^{2}, &
		&\frac{1}{N}\sum_{k}^{(i)}\bse_{k}\adj X_{i}G\bse_{j}\frac{\partial T_{ij}}{\partial g_{ik}}\prec (\Pi_{i}+\Pi_{j})^{2}. &
	\end{align*}
\end{lem}

\subsection{Stability of perturbed linear system}
In this subsection, we provide a stability result concerning the perturbations of the linear system \eqref{eq_suborsystemPhi} for a sufficiently large $\eta$.  This result serves as the starting point of our bootstrapping arguments in Section \ref{sec:freeN} for the system \eqref{eq:def_Phi_ab} and in Section \ref{subsec_stronglocallawfixed} for \eqref{eq:def_PhiAB}. The proof of Lemma \ref{lem:Kantorovich_appl} relies on the well-known Kantorovich's theorem and we omit the details here.

\begin{lem}\label{lem:Kantorovich_appl}
	For $\eta>0$ and $\theta\in(0,\pi/2)$, define 
	\beq\label{eq_epsilonetatheta}
	\caE(\eta,\theta)\deq \{z\in\C_{+}:\im z>\eta, \theta<\arg z <\pi-\theta\}.
	\eeq
	Let $(\mu_{1},\mu_{2})$ be either $(\mu_{\alpha},\mu_{\beta})$ or $(\mu_{A},\mu_{B})$, and let $\Phi$ be either $\Phi_{\alpha\beta}$ or $\Phi_{AB}$ accordingly. Moreover, we set $\wt{\Omega}_{1}$ and $\wt{\Omega}_{2}$ be analytic functions mapping $\caE(\wt{\eta}_{1},\wt{\theta}),$ for some $\wt{\eta}_{1}$ and $\wt{\theta},$ to $\C_{+}$. Denote $\wt{r}(z)\deq \Phi(\wt{\Omega}_{1}(z),\wt{\Omega}_{2}(z),z)$. Assume that there exists a positive constant $0<c<1$ such that the following hold for all $z\in\caE(\wt{\eta}_{1},\wt{\theta})$:
	\begin{align}\label{eq:kanappconidition1}
		\Absv{\frac{\wt{\Omega}_{1}(z)}{z}-1}\leq c, \ \
		\Absv{\frac{\wt{\Omega}_{2}(z)}{z}-1}\leq c,	\ \
		\norm{\wt{r}(z)}\leq c.
	\end{align}
	
	Then there exist $\eta_{1}>\wt{\eta}_{1}$ and $\theta>\wt{\theta}$ depending only on $\mu_{\alpha},\mu_{\beta}$ and $c$ such that for all sufficiently large $N,$ we have
	\begin{align}
		\absv{\wt{\Omega}_{1}(z)-\Omega_{1}(z)}\leq 2\norm{\wt{r}(z)}, \ \
		\absv{\wt{\Omega}_{2}(z)-\Omega_{2}(z)}\leq 2\norm{\wt{r}(z)},
	\end{align}
	hold	for all $z\in\caE(\eta_{1},\theta)$, where $\Omega_{1}$ and $\Omega_{2}$ are subordination functions corresponding to the pair $(\mu_{1},\mu_{2})$ via Lemma \ref{lem_subor}.
\end{lem}

\end{document}